\documentclass[a4paper,11pt]{article}

\usepackage{anyfontsize}
\usepackage{lmodern}
\usepackage{microtype}
\usepackage{amsfonts,amssymb,amsmath,amsthm}
\usepackage{mathtools} 
\mathtoolsset{showonlyrefs} 
\usepackage{nicefrac,xfrac}
\usepackage{tikz-cd}
\usepackage{enumitem}
\usepackage[colorlinks=true, urlcolor=blue, linkcolor=blue, citecolor=blue]{hyperref}

\usepackage[
  backend=biber,
  style=alphabetic,
  sorting=nyt,
  maxnames=6,
  maxalphanames=6,
  giveninits=true,
  date=year,
  isbn=false,
  url=false,
]{biblatex}
\DeclareFieldFormat{extraalpha}{#1}
\DeclareLabelalphaTemplate{
  \labelelement{
    \field[final]{shorthand}
    \field{label}
    \field[strwidth=3,strside=left,ifnames=1]{labelname}
    \field[strwidth=1,strside=left]{labelname}
  }
}

\numberwithin{equation}{section}
\newtheorem{theorem}{Theorem}[section]
\newtheorem{proposition}[theorem]{Proposition}
\newtheorem{corollary}[theorem]{Corollary}

\newtheorem{lemma}[theorem]{Lemma}
\newtheorem{claim}[theorem]{Claim}
\theoremstyle{definition}
\newtheorem{definition}[theorem]{Definition}
\newtheorem{remark}[theorem]{Remark}

\theoremstyle{remark}

\newcommand{\R}{\mathbb{R}}  
\newcommand{\Sph}{\mathbb{S}}  
\DeclareMathAlphabet{\mathbbold}{U}{bbold}{m}{n}
\newcommand*{\1}{\mathbbold{1}}

\newcommand*{\id}{id}

\DeclareMathOperator{\ind}{ind}

\DeclareMathOperator{\Tr}{Tr}

\DeclareMathOperator{\supp}{supp}

\DeclareMathOperator{\Vol}{Vol}
\DeclareMathOperator{\BMO}{BMO}

\DeclarePairedDelimiter\floor{\lfloor}{\rfloor}
\DeclarePairedDelimiter\ceil{\lceil}{\rceil}
\DeclarePairedDelimiterX\set[2]{\{}{\}}{\,#1 \;\delimsize\vert\; #2\,}
\newcommand*{\SHM}[1]{u_{#1}} 
\newcommand*{\abs}[1]{\left|#1\right|} 
\newcommand*{\norm}[1]{\left\|#1\right\|} 
\newcommand*{\brr}[1]{\left(#1\right)} 
\newcommand*{\brs}[1]{\left[#1\right]} 
\newcommand*{\brc}[1]{\left\{#1\right\}} 
\newcommand*{\brt}[1]{\left\langle #1\right\rangle } 
\mathchardef\mhyphen="2D 
\renewcommand*{\epsilon}{\varepsilon}
\renewcommand*{\phi}{\varphi}
\DeclareMathOperator{\iml}{\mathfrak{Im}}
\makeatletter
\newcommand*{\oset}[3][0.45ex]{%
  \mathrel{\mathop{#3}\limits^{
    \vbox to#1{\kern-2\ex@
    \hbox{$\scriptstyle#2$}\vss}}}}
\makeatother
\def\Xint#1{\mathchoice
{\XXint\displaystyle\textstyle{#1}}%
{\XXint\textstyle\scriptstyle{#1}}%
{\XXint\scriptstyle\scriptscriptstyle{#1}}%
{\XXint\scriptscriptstyle\scriptscriptstyle{#1}}%
\!\int}
\def\XXint#1#2#3{{\setbox0=\hbox{$#1{#2#3}{\int}$ }
\vcenter{\hbox{$#2#3$ }}\kern-.6\wd0}}

\def\dashint{\Xint-}
\newcommand{\overbar}[1]{\mkern 1.5mu\overline{\mkern-1.5mu#1\mkern-1.5mu}\mkern 1.5mu}
\newcommand*{\nEigen}[2]{\overline{\lambda}_{#1,#2}} %
\newcommand*{\NEigen}[2]{\Lambda_{#1,#2}} %
\newcommand*{\m}[1]{\brt{#1}} 
\newcommand*{\tensor}[1]{\mathfrak{N}[#1]} 
\renewcommand*{\t}{\boldsymbol{\tau}} 
\newcommand*{\tf}[1]{#1^*#1} 
\newcommand*{\F}{u} 
\newcommand*{\wH}[2][]{H^{1}\ifx\empty#1\empty(#2)\else(#1,#2)\fi} 

\begin{document}
\author{Denis Vinokurov}
\date{}
\title{Eigenvalue optimization in higher dimensions and $p$-harmonic maps}
\maketitle
\begin{abstract}
  We prove existence results for optimization problems for the $k$th Laplace eigenvalue
  on closed Riemannian manifolds of dimension $m \geq 3$, depending on the choice of normalization.
  One such normalization leads to eigenvalue optimization within a conformal class,
  for which existence of maximizers was previously known only in dimension two.
  We also prove that all absolutely continuous maximizers of the normalized eigenvalue functionals are
  always induced by $p$-harmonic maps into spheres, where $p \in [2,m]$.
  For $p$ sufficiently close to $m$, the maximizers are Hölder-continuous, whereas for $p<m$ no bubbling occurs.
  A key tool in our analysis is the application of techniques from the theory of topological tensor products,
  which appear to be well suited for studying eigenvalue-related optimization problems.
\end{abstract}
\tableofcontents

\section{Introduction and main results}
\subsection{Eigenvalue optimization in higher dimensions}
Let $(M,g)$ be a closed connected Riemannian manifold with $m = \dim M$, and let
$\Delta_g = \delta_g d \colon C^\infty(M) \to C^\infty(M)$ be the associated
Laplace-Beltrami operator, where $\delta_g$ is the formal adjoint of $d$.
The spectrum of $\Delta_g$ forms the following
nondecreasing sequence:
\begin{equation}
  0 = \lambda_0(g) < \lambda_1(g) \leq \lambda_2(g) \leq
  \cdots\nearrow \infty,
\end{equation}
where the eigenvalues are repeated according to their multiplicities. The problem of geometric eigenvalue optimization
within a conformal class is concerned with identifying metrics on which the following supremum is achieved:
\begin{equation}\label{eq:conf-opt}
  \Lambda_k([g]) = \sup_{\tilde{g}\in[g]} \overline{\lambda}_k(\tilde{g}) \leq C([g])k^{2/m},
\end{equation}
where $\overline{\lambda}_k(g) :=\lambda_k(g)\Vol_g(M)^{2/m}$, and the upper bound was established in~\cite{Korevaar:1993:upper-bounds-for-eigenval}.
By the same paper, the supremum over all metrics on $M$ (equivalently, $\sup_{[g]} \Lambda_k([g])$) is also finite when $m = 2$.
In higher dimensions, however, this supremum is infinite (see~\cite{Colbois-Dodziuk:1994:metrics-with-large-eigenval}).

One of the main reasons for studying these optimization problems is their connection to
$m$-harmonic maps $u\colon M \to \Sph^n$ (see~\cite{Colbois-ElSoufi:2003:extremal,Karpukhin-Metras:2022:higher-dim}): the critical metrics for the functional
$\tilde{g} \mapsto \overline{\lambda}_k(\tilde{g})$ are of the form $\tilde{g} = \abs{du}^2_g g$.

Lots of progress has been made in studying eigenvalue optimization on surfaces; see~\cite{Karpukhin-Nadirashvili-Penskoi-Polterovich:2022:existence,
Karpukhin:2021:index-of-min-shperes, Penskoi:2013:extrem-metrics} and references therein. In particular, the works \cite{
  Nadirashvili-Sire:2015:conf-spec-and-harm-maps,
  Nadirashvili-Sire:2015:higher-order-max,
  Petrides:2014:heat-kernel,
  Petrides:2018:exist-of-max-eigenval-on-surfaces,
  Karpukhin-Nadirashvili-Penskoi-Polterovich:2022:existence} establish the existence of maximizing metrics on surfaces for~\eqref{eq:conf-opt} (under some natural assumptions); see also~\cite{Vinokurov:2025:sym-eigen-val-lms}
for a simplified version of the proof. Far less is known about conformal class optimization in higher dimensions.
Explicit examples of maximizers can be found in \cite{ElSoufi-Ilias:1986:hersch, Kim:2022:second-sphere-eigenval}.
See  also \cite*{Kim:2025:proj-space-2d-eigenval,
Karpukhin-Metras-Polterovich:2024:dirac-eigenval-opt,
Perez-Ayala:2021:conf-laplace-sire-xu-norm,
Perez-Ayala:2022:extr-metrics-paneitz-operator,
Premoselli-Vetois:2024:nonexist-2nd-conformal,
Humbert-Petrides-Premoselli:2025:extrem-eigenval-GJMS-subcrit,
Gursky-Perez-Ayala:2022:2d-eigenval-conform-laplacian,
Colbois-Girouard-Gordon-Sher:2024:recent-develop-on-steklov}
for some related eigenvalue optimization results.

The current paper addresses the question of the existence of maximizing metrics in dimensions $m \geq 3$. In dimension two,
a priori eigenvalue multiplicity bounds were one of the key ingredients for the proof. In dimensions $m \geq 3$, we no longer have
such multiplicity bounds, as shown in~\cite{ColinDeVerdiere:1986:unbounded-mult}. Here, we present a framework, based on the theory
of topological tensor products, that allows working with uncontrolled multiplicities in the context of eigenvalue-related optimization problems.

More generally, we also prove the existence of maximizers for a family of problems related to $p$-harmonic maps to spheres.
Let $\alpha \in L^1_+(M)$ be a nonnegative function and let
$\mu \in \mathcal{M}_+^c(M)$ be a nonatomic (sometimes called continuous) nonnegative Radon measure. We fix a background metric $g$ and define the following variational eigenvalues
$\lambda_k(\alpha,\mu) \in [0,\infty)$:
\begin{equation}\label{eq:eigen-var-char}
  \lambda_k(\alpha,\mu) = \inf_{F_{k+1}} \sup_{\phi \in F_{k+1}} \frac{\int \abs{d\phi}^2_g \alpha dv_g}{\int \phi^2 d\mu},
\end{equation}
where $v_g$ denotes the volume measure associated with $g$, and the infimum is taken over all subspaces $F_{k+1} \subset C^\infty(M)$
whose images in $L^2(\mu)$ have dimension $(k+1)$. (We can also define $\lambda_k(\alpha,\mu)$ for atomic measures $\mu$ by
setting $\lambda_k(\alpha,\mu) = \infty$ whenever $\dim L^2(\mu) \leq k$.)
Note that $\lambda_k(g)=\lambda_k(1,v_g)$; equivalently, in the notation $\lambda_k(\alpha,\mu)$ we write $\lambda_k(g)=\lambda_k(1,1)$ by identifying the density $1$ with the measure
$v_g$, so we have $\lambda_k(\rho^2 g) = \lambda_k(\rho^{m-2},\rho^m)$.
For each $p \geq 2$, we can define the following normalization:
\begin{equation}
  \nEigen{k}{p}(\alpha,\mu) = \lambda_{k}(\alpha,\mu)\frac{\mu(M)}{\norm{\alpha}_{L^{\frac{p}{p-2}}}}.
\end{equation}
The critical points of $\nEigen{k}{p}$ correspond to $p$-harmonic maps into spheres
(see Lemma~\ref{lem:lambda-alph-mu-subdiff} and also \cite[Section~5]{Petrides-Tewodrose:2024:eigenval-clarke-subdiff}).
Namely, if the pair
$(\alpha,\mu)$ is critical for $\nEigen{k}{p}$ and sufficiently regular, one necessarily has
\begin{equation}\label{p-crit-meas}
  \alpha = \abs{du}^{p-2}_g,\ \mu = \abs{du}^p_g,\
  \text{ and }\ \delta_g(\abs{du}^{p-2}_g du) = \abs{du}^p_g u,
\end{equation}
for some $p$-harmonic map $u \colon M \to \Sph^n$.

We now investigate the following optimization problem:
\begin{equation}\label{eq:eigen-p-sup}
  \NEigen{k}{p}(g) = \sup \set*{\nEigen{k}{p}(\alpha,\mu)}{\alpha, \mu \in C^\infty(M),\ \alpha,\mu > 0}.
\end{equation}
Note that $\NEigen{k}{m}(g)$ is conformally invariant, $\Lambda_k([g]) \leq \NEigen{k}{m}([g])$, and
both problems share the same critical points, which implies that the two suprema coincide provided $\NEigen{k}{m}([g])$ admits a maximizer.

Continuing our investigation, let us remark that $\NEigen{k}{p}(g) = \infty$ when $p > m$.
This can be seen by choosing $\mu_{\epsilon,\epsilon'} = \alpha_{\epsilon,\epsilon'} =\chi_{B_\epsilon(x)} + \epsilon'$, where $\chi_{B_\epsilon(x)}$ is
the characteristic function of the geodesic ball $B_\epsilon(x)$.
For $p=2$, since $\alpha /\norm{\alpha}_{L^\infty}\leq 1$, we obtain
$\lambda_k(\alpha,\mu)/\norm{\alpha}_{L^\infty} \leq \lambda_k(1, \mu)$.
Thus, $\NEigen{k}{2}(g) = \sup_\mu \nEigen{k}{2}(1, \mu)$, and this case has been addressed in \cite{Karpukhin-Stern:2024:harm-map-in-higher-dim, Vinokurov:2025:higher-dim-harm-eigenval}.
Assuming that $\Vol_g (M) = 1$, Hölder's inequality yields
\begin{equation}\label{ineq:for-dif-p}
  \NEigen{k}{2}(g) \leq \NEigen{k}{p}(g) \leq \NEigen{k}{m}([g])\leq C([g])k^\frac{2}{m},
\end{equation}
where the final bound follows from
\cite[Theorem~2.1]{Colbois-ElSoufi-Savo:2015:-laplace-with-density};
see also \cite[Theorem~5.3]{Grigoryan-Netrusov-Yau:2004:eignval-of-ellipt-oper},
combined with an estimate $\int_{B_r(x)} \alpha \leq Cr^2\norm{\alpha}_{L^{\frac{m}{m-2}}}$.

In Section~\ref{sec:applic-tensor-prod}, we develop techniques based on the theory of topological tensor products (see Lemma~\ref{lem:main}), which we
believe are useful for many eigenvalue-related optimization problems, especially when dealing with uncontrolled multiplicities or
functionals depending on infinitely many eigenvalues.
As an application, we present the following main results of the paper.

Let $\floor{x}$ and $\ceil{x}$ denote the floor and ceiling functions for $x\in \R$, and let $\Sph^\infty = \Sph(\ell^2)$ denote the Hilbert unit sphere.
\begin{theorem}\label{thm:main-0}
  Let $(M,g)$ be a closed connected Riemannian manifold of dimension $m \geq 2$, $p\in [2,m]$, $d = 3 + \floor{p+2\sqrt{p-1}}$,
  $(\alpha,\mu) \in L^{p/(p-2)}_+(M)\times L^1_+(M)$, and we fix $\alpha \equiv 1$ if $p=2$. Then $\nEigen{k}{p}(\alpha,\mu) \leq \NEigen{k}{p}(g)$.
  If equality occurs, then, up to rescaling,
  \begin{equation}
    \alpha = \abs{du}^{p-2}
    \quad\text{and}\quad
    \mu = \abs{du}^{p}
  \end{equation}
  for some locally spectrally stable (see Definition~\ref{def:stable-map})
  $p$-harmonic map $u \in H^{1,p}(M,\Sph^\infty)$. In particular, $u$ is locally $p$-energy minimizing, and there exist
  a closed set $\Sigma \subset M$ of Hausdorff dimension at most $m-d$ and $\gamma = \gamma(m,p) \in (0,1)$ such that
  $u \in C^{1,\gamma}(M\setminus\Sigma, \Sph^\infty)$.
\end{theorem}
\begin{remark}
  Thus, all absolutely continuous maximizing densities for $\nEigen{k}{p}$ are induced by $p$-harmonic maps.
  When $p=2$, for $\lambda_2$ in dimension $m=2$ and for $\lambda_1$ in dimensions $2\le m\le 5$, there is a complementary result due to
  \cite{Karpukhin-Stern:2024:new-character-of-conf-eigenval,Karpukhin-Stern:2024:harm-map-in-higher-dim}:
  in the theorem above, $\mu$ may be any Radon measure for which the identity map $H^1(M)\to L^2(\mu)$ is compact.
  Note that absolutely continuous measures need not be of this kind.
\end{remark}
\begin{remark}
  The theorem above also holds for $\mu \in H^{1,\frac{m}{m-1}}(M)^*$ when $m \geq 3$ (see Remark~\ref{rem:gkl-sobolev-duals} and Proposition~\ref{prop:upper-cont-and-bound}).
  Those measures induce the compact identity map $H^1(M,\alpha)\to L^2(M,\mu)$ if $\alpha \geq c > 0$.
  However, we do not know whether the eigenvalues associated with an arbitrary measure for which the identity map $H^1(M, \alpha)\to L^2(M,\mu)$ is compact
  can be approximated (or at least bounded from above) by the eigenvalues associated with smooth measures.
\end{remark}
The next theorem complements the existence theory in dimension two by providing a higher-dimensional analogue; see \cite{
  Nadirashvili-Sire:2015:higher-order-max,
  Petrides:2018:exist-of-max-eigenval-on-surfaces,
  Karpukhin-Nadirashvili-Penskoi-Polterovich:2022:existence,
  Petrides:2024:conf-class-opt-lms,
  Vinokurov:2025:sym-eigen-val-lms}.
  Note that the $C^{1,\gamma}$-regularity assertions in Theorems~\ref{thm:main-2} and~\ref{thm:main-1} follow directly from
  Theorem~\ref{thm:main-0} once we know that $\Lambda_{k,p}(g)$ is achieved on some $(\alpha,\mu)$.
\begin{theorem}\label{thm:main-2}
  Let $(M,g)$ be a closed connected Riemannian manifold of dimension $m \geq 3$.
  Then
  \begin{equation}
    \Lambda_k([g])^{m/2}= \Lambda_{k,m}([g])^{m/2} \geq \max_{1\leq r \leq k}
    \brc{\Lambda_{k-r}([g])^{m/2} + \Lambda_{r}(\Sph^m,[g_{\Sph^m}])^{m/2}}.
  \end{equation}
  If the inequality is strict, then there exists an
  $m$-harmonic map $u \in C^{1,\gamma}(M, \Sph^n)$ with $\gamma = \gamma(m) \in (0,1)$
  such that $\lambda_{k}(\abs{du}_g^2 g) = 1$ and
  \begin{equation}
    \Lambda_k([g]) = \overline{\lambda}_{k}(\abs{du}_g^2 g)
    = \brr{\int_M \abs{du}_g^m dv_g}^{2/m}.
  \end{equation}
\end{theorem}
\begin{remark}
  In other words, the theorem asserts that $\Lambda_k([g])$ is always achieved on a disjoint union
  \begin{equation}
    (M,\abs{du_0}^2_g g) \sqcup (\Sph^m, \abs{du_1}^2_{g_{\Sph^m}} g_{\Sph^m}) \sqcup \cdots \sqcup (\Sph^m, \abs{du_b}^2_{g_{\Sph^m}} g_{\Sph^m})
  \end{equation}
  for some $m$-harmonic maps $u_i$, where each map maximizes some $\overline{\lambda}_{k_i}$ with $\sum_i k_i = k$.
  This can be seen by inductively applying Theorem~\ref{thm:main-2} with $M = \Sph^m$.
\end{remark}
As noted by Romain Petrides, the following generalization of the 2-dimensional result \cite[Theorem~1]{Petrides:2014:heat-kernel} follows
immediately from \cite[Theorem~1]{Petrides:2015:Lambda-strictly-large-sph}, Theorem~\ref{thm:main-2}, and Corollary~\ref{cor:1st-max-on-minimal}.
\begin{corollary}\label{cor:1st-eigenval-acheived}
  For any closed connected manifold $(M,[g])$ of dimension $m \geq 3$, there exists an $m$-harmonic map
  $u \in C^{1,\gamma}(M, \Sph^n)$ such that
  \begin{equation}
    \Lambda_1([g]) = \overline{\lambda}_{1}(\abs{du}_g^2 g)
    = \brr{\int_M \abs{du}_g^m dv_g}^{2/m}.
  \end{equation}
\end{corollary}
\begin{theorem}\label{thm:main-1}
  Let $(M,g)$ be a closed connected Riemannian manifold of dimension $m$ and let $2 \leq p  < m$.
  Then there exists a locally spectrally stable $p$-harmonic map $u \in H^{1,p}(M,\Sph^\infty)$ such that
  \begin{equation}
    \lambda_k(\abs{du}^{p-2}_g, \abs{du}^p_g) = 1,
    \quad\text{and}\quad
    \Lambda_{k,p}(g) = \brr{\int \abs{du}^p_g dv_g}^{2/p}.
  \end{equation}
  In particular, $u$ is locally $p$-energy minimizing, and $u \in C^{1,\gamma}(M\setminus\Sigma, \Sph^\infty)$ away from
  a closed set $\Sigma$ of Hausdorff dimension at most $m-d$ with $\gamma = \gamma(m,p) \in (0,1)$ and
  $d$ as in Theorem~\ref{thm:main-0}.
\end{theorem}
In particular, for some values of $p < m$, neither bubbling nor singularities occur:
\begin{corollary}\label{cor:smooth-maximizers}
  Let $2 \leq p  < m$. If $m \leq 2 + \floor{p+2\sqrt{p-1}}$, or equivalently,
  $p \geq m-2\sqrt{m-2}$,
  the supremum $\Lambda_{k,p}(g)$ is achieved by a $p$-harmonic map
  $u \in C^{1,\gamma}(M, \Sph^n)$ for some finite $n$.
\end{corollary}
\begin{remark}
  When $p=2$, we are able to prove (see \cite[Section 5.2]{Vinokurov:2025:higher-dim-harm-eigenval}) that the image $u(M)$ of the harmonic maps
  from Theorem~\ref{thm:main-0} always lies (up to a measure zero set) in some $\Sph^n\subset \Sph^\infty$. The same methods should extend to the
  remaining case $2<p<m-2\sqrt{m-2}$ provided the critical set $\set{x\in M}{d_xu=0}$
  has zero $p$-capacity. Unfortunately, very little seems to be known about the size of this set. Already for scalar $p$-harmonic functions in
  dimensions $m\ge3$, no general bound on the Hausdorff dimension of the critical set is known, and even positive Lebesgue measure is not excluded in general;
  in contrast, see \cite{Manfredi:1988:plane-p-harm-func} for $m=2$.

  On the other hand, if the critical set is sufficiently large, the degeneracy of the weights creates the possibility for the weighted Schrödinger operator
  $\phi\mapsto \delta(|du|^{p-2}d\phi)-|du|^p\phi$ to have infinite-dimensional kernel while retaining finite index.
\end{remark}

\subsection{Examples on selected manifolds}\label{sec:examples}
If $u \colon M \to \Sph^n$ is a minimal immersion, one always has (see~\cite{Takahashi:1966:minimal-immersions}) that the coordinate functions of $u$ are the eigenfunctions
of the Laplacian $\Delta_g$, where $g = u^*g_{\Sph^n}$. The following result is a generalization of \cite{ElSoufi-Ilias:1986:hersch}.
\begin{theorem}\label{thm:1st-max-on-minimal}
  Let $u \colon (M,g) \to \Sph^n$ be an isometric minimal immersion of a closed manifold $M$ given by $\lambda_1(g)$-eigenfunctions of $\Delta_g$.
  Then for any $p \in [2,m]$ and any $(\alpha,\mu) \in L^{p/(p-2)}_+(M) \times \mathcal{M}_+^c(M)$, we have
  \begin{equation}
    \nEigen{1}{p}(\alpha,\mu) \leq  m\Vol_g(M)^{2/p} = \NEigen{1}{p}(M, g),
  \end{equation}
  with the equality if and only if $\alpha = 1$ and $\mu = 1$ up to rescaling (and a conformal diffeomorphism of $\Sph^m$ if $p=m$ and $M = \Sph^m$).
\end{theorem}
Under the conditions of the theorem above, $\Vol_g(M)$ is the conformal volume of $(M,[g])$ (see~\cite{ElSoufi-Ilias:1986:hersch, Li-Yau:1982:conf-volume}).
\begin{corollary}\label{cor:1st-max-on-minimal}
  Let $M \in \brc{\mathbb{R}P^m, \mathbb{C}P^d, \mathbb{H}P^d, \mathbb{C}aP^2, \mathbb{S}^a(\sqrt{\tfrac{a}{a+b}})\times \mathbb{S}^b(\sqrt{\tfrac{b}{a+b}})}$
  with the corresponding canonical metric $g$; see more examples in \cite[Section~3.6]{ElSoufi-Ilias:1986:hersch}. Then
  Theorem~\ref{thm:1st-max-on-minimal} holds for $(M,g)$.

  In particular, one can take $u = \id_{\Sph^m} \colon \Sph^m \to \Sph^m$ for $M = \Sph^m$.
\end{corollary}

As in~\cite{Vinokurov:2025:higher-dim-harm-eigenval} for $p = 2$, we can look at the optimization problems of higher eigenvalues on $(\Sph^m, g_{\Sph^m})$.
From~\cite{Kim:2022:second-sphere-eigenval} we also know that
\begin{equation}
  \overline{\lambda}_k(g) \leq  2^{2/m} \Lambda_{1}(\Sph^m, [g_{\Sph^m}]) = \Lambda_{2}(\Sph^m, [g_{\Sph^m}]),
\end{equation}
where the supremum is achieved on the disjoint union of two spheres, and the inequality is always strict.
By using Proposition~\ref{prop:var-character}, one can check that the inequality also extends to $\nEigen{1}{m}(\alpha,\mu)$.

Now, consider the maps
$\SHM{k} \colon \Sph^m \to \Sph^{m-1-k}$ given by
\begin{equation}
  \SHM{k}\colon (x,y) \mapsto \frac{x}{\abs{x}},
  \quad\text{where } (x,y) \in \Sph^{m}\subset \R^{m-k}\times\R^{k+1}.
\end{equation}
The map $\SHM{k}$ is smooth away from $\brc{0}\times\Sph^k$, and one can check (see~\cite[(1.8)]{Vinokurov:2025:higher-dim-harm-eigenval}) that
$\SHM{k} \in H^{1,p}(\Sph^m)$ only when $k < m - p$.
\begin{theorem}\label{thm-max-on-sphere}
  Let $\SHM{k} \colon \Sph^m \to \Sph^{m-1-k}$ be as above and
  $(\alpha,\mu) \in L^{p/(p-2)}_+(\Sph^m) \times \mathcal{M}_+^c(\Sph^m)$, where $p \in [2, m)$. If $0 \leq k \leq m - 3 - \ceil{2p}$, then
  \begin{equation}
    \nEigen{k+2}{p}(\alpha,\mu) \leq \brr{\int \abs{d\SHM{k}}^p_{g_{\Sph^m}} dv_{g_{\Sph^m}}}^{2/p} = \NEigen{k+2}{p}(\Sph^m, g_{\Sph^m}),
  \end{equation}
  and, up to rescaling, the pair $(\abs{d\SHM{k}}^{p-2}, \abs{d\SHM{k}}^p)$ is the unique maximizer. Furthermore, we have the strict inequality
  $\NEigen{k+2}{p}(\Sph^m, g_{\Sph^m}) < (\int \abs{d\SHM{k}}^p_{g_{\Sph^m}} dv_{g_{\Sph^m}})^{2/p}$
  if $m - 2 - \ceil{2p} \leq k \leq m - 1 - \floor{p}$.
\end{theorem}
\begin{remark}
  We see that, in general, maximizers for $\nEigen{k}{p}$ can be singular.
  However, unless $p=2$, the Hausdorff dimensions of the singularities of $\SHM{k}$ do not reach
  the maximal possible value $m - 3 - \floor{p + 2\sqrt{p-1}}$ from Theorem~\ref{thm:main-0}.
\end{remark}

\subsection{Ideas of the proofs of existence and regularity}

The ideas of the proof originate in~\cite{Vinokurov:2025:higher-dim-harm-eigenval}. We employ the direct variational approach via maximizing sequences; however,
instead of working with eigenmaps $\F \in H^{1,p}(M,\R^\infty)$, we
work with the projective tensor product $H^{1,p}(M)\widehat{\otimes}_\pi H^{1,p}(M)$. The motivation for considering such a space
comes from viewing eigenvalues as Lipschitz functionals on the space of compact
operators or quadratic forms, and the projective tensor product
can be identified with the corresponding dual space (see Proposition~\ref{prop:tensor-duality}).
The advantage of the space $H^{1,p}\widehat{\otimes}_\pi H^{1,p}$ is that the embedding
$H^{1,p}\widehat{\otimes}_\pi H^{1,p} \to L^{p}\widehat{\otimes}_\pi L^{p}$ is compact, which is a key ingredient in proving the existence of maximizing
$p$-harmonic maps.

Another observation is that any vector-valued map $\F \in H^{1,p}(M,\R^\infty)$ induces
an absolutely $p$-summable operator $\F\colon H^{1,p}(M)^* \to \R^\infty := \ell^2$, defined by $\F\colon \xi \mapsto (\cdots, \xi(\F^i),\cdots)$;
see~\cite[Proposition~25.10]{Defant-Floret:1993:tensor-norms}. Then the composition
\begin{equation}
  \F^* \F = \sum_i \F^i \otimes \F^i \in \mathfrak{L}[(H^{1,p})^*, H^{1,p}]
\end{equation}
belongs to the ideal of $(p,p)$-dominated operators for $p\geq 2$. On $L^p$ spaces (see~\cite[Proposition~26.1]{Defant-Floret:1993:tensor-norms}),
this operator ideal coincides with the ideal of nuclear operators $\mathfrak{N}[(H^{1,p})^*, H^{1,p}] \approx H^{1,p}\widehat{\otimes}_\pi H^{1,p}$,
with an equivalent norm so that $\norm{\F^* \F}_{\pi} \sim \norm{\F}_{H^{1,p}}^2$.

In the current paper we also observe that the convergence $\F^*_n \F_n \to \F^* \F$ in $H^{1,p}\widehat{\otimes}_\pi H^{1,p}$ is equivalent
(up to isometry) to convergence in $H^{1,p}(M,\ell^2)$ (see Lemma~\ref{lem:main}); that is, there exist (linear) isometric embeddings $I_n \colon \iml  \F_n \to \ell^2$, $I \colon \iml  \F \to \ell^2$
from the linear images of the maps $u_n$ and $u$ such that $I_n \F_n \to I\F$ in $H^{1,p}(M,\ell^2)$. Although one can still establish the existence working solely with tensors (cf.~\cite{Vinokurov:2025:higher-dim-harm-eigenval})
using Corollary~\ref{cor:seq-spaces-tensor-convergence} directly,
the equivalence of the two convergences allows us to work with $\ell^2$-valued maps instead, repeating nearly the same proof as in the two-dimensional case (cf.~\cite{Vinokurov:2025:sym-eigen-val-lms}).

It then remains to study the regularity of the limiting $p$-harmonic maps into $\Sph^\infty \subset \ell^2$.
It turns out that the Hilbert-valued theory is essentially identical to the finite-dimensional case.
Moreover, the local spectral stability of such maps (see Definition~\ref{def:stable-map})
permits the derivation of improved regularity results in the spirit of~\cite{Xin-Yang:1995:stable-p-harm-maps}.


\subsection{Acknowledgements}
This paper forms part of the author's PhD thesis under the supervision of Mikhail Karpukhin and Iosif Polterovich,
whom the author thanks for their guidance and many fruitful discussions.
The author is particularly grateful to Mikhail Karpukhin for pointing out the link between different eigenvalue
normalizations and $p$-harmonic maps into spheres, and to Mikhail Karpukhin and Daniel Stern for sharing their notes
on the spectral properties of $m$-harmonic maps, which inspired the compactness estimate~\eqref{ineq:m-harm-comp-estim}.
The author also thanks Romain Petrides for stimulating discussions and for kindly pointing out Corollary~\ref{cor:1st-eigenval-acheived}.
This work was partially supported by an ISM scholarship.

\section{Preliminaries}
Throughout the paper, we adopt the convention that $C, c, c_1, c_2, \dots$ denote
some positive constants; their dependencies will be specified if necessary.
By $\Omega$, we will denote a bounded open set in a Riemannian manifold $M$
unless specified otherwise (that is, $\Omega$ is a domain, possibly with boundary, while $\partial M = \varnothing$).
We write $\epsilon \in \brc{\epsilon_i}$ to denote a sequence of positive numbers.
We will work mostly with real vector spaces, so complex conjugation is redundant, except in Section~\ref{sec:tensor-prod}.

\subsection{Topological tensor products of Banach spaces}\label{sec:tensor-prod}

Let $E$ and $F$ be normed spaces over a field $\mathbb{K} \in \brc{\mathbb{R}, \mathbb{C}}$. Their continuous duals will be denoted by $E^*$ and $F^*$, respectively.
By $\mathfrak{Bil}[E,F]$, we will denote the space of bounded bilinear forms, and
$\mathfrak{L}[E,F]$ stands for the space of bounded linear operators. We write $\mathfrak{Bil}[E]$, $\mathfrak{L}[E]$, etc., as shorthand for
$\mathfrak{Bil}[E,E]$ and $\mathfrak{L}[E,E]$, etc. Note that by the standard correspondence between
bilinear forms and linear operators, $\mathfrak{L}[F,E^*] \cong  \mathfrak{Bil}[E,F] \cong  \mathfrak{L}[E,F^*]$
are isometrically isomorphic.

\paragraph{A zoo of convergences}
\begin{itemize}
  \item The convergence in norm: $x_n \to x \implies \norm{x_n - x}_E \to 0$ for $x_n, x \in E$;
  \item The weak convergence: $x_n \oset[0.7pt]{w}{\to} x \implies \brt{x_n,x^*} \to \brt{x,x^*}$ for all $x^* \in E^*$;
  \item The weak$^*$ convergence: $x_n^* \oset{w^*}{\to} x^* \implies \brt{x,x^*_n} \to \brt{x,x^*}$ for all $x \in E$;
  \item The strong/pointwise convergence in $\mathfrak{L}[E,F]$: $T_n \oset[0.5pt]{s}{\to} T \implies T_n x \to T x$
  for all $x \in E$;
  \item The weak operator topology convergence: $T_n \oset{wot}{\to} T \implies \brt{T_n x, y^*} \to \brt{T x, y^*}$
  for all $x \in E$ and $y^* \in F^*$, where $T_n,T \in \mathfrak{L}[E,F]$;
  \item The pointwise convergence of bilinear forms: $\mathfrak{q}_n \oset{pt}{\to} \mathfrak{q} \implies \mathfrak{q}_n[x,y] \to \mathfrak{q}[x,y]$
  for all $x \in E$ and $y \in F$, where $\mathfrak{q}_n, \mathfrak{q} \in \mathfrak{Bil}[E,F]$.
\end{itemize}

\subsubsection{Two main tensor products}
For a more detailed exposition, see, for example, \cite{Defant-Floret:1993:tensor-norms}.
Consider an element $\t \in E \otimes F$ of the algebraic tensor product of $E$ and $F$. Its projective norm is defined as
\begin{equation}
  \norm{\t}_\pi = \inf \set*{\sum \norm{x_i}\norm{y_i}}{\t = \sum x_i \otimes y_i}.
\end{equation}
\begin{definition}
  The completion of $E \otimes_\pi F := (E \otimes F, \norm{\cdot}_{\pi})$ is denoted by $E \widehat{\otimes}_\pi F$ and is called the \emph{projective tensor product} of $E$ and $F$.
  Note that $E \widehat{\otimes}_\pi F = \widehat{E} \widehat{\otimes}_\pi \widehat{F}$, where $\widehat{E}$ and $\widehat{F}$ are the corresponding
  completions.
\end{definition}
The projective tensor norm has the universal property of the tensor product in the category of normed spaces:
each bilinear map $E\times F \to G$ is continuous if and only if
its linearization $E \otimes_\pi F \to  G$ is continuous and has the same norm.
\begin{equation}
  \begin{tikzcd}
    E \times F & E \otimes_\pi F \\
    & G
    \arrow[from=1-1, to=1-2]
    \arrow[from=1-1, to=2-2]
    \arrow[dashed, from=1-2, to=2-2]
  \end{tikzcd}
\end{equation}
In particular,
\begin{equation}
  (E \widehat{\otimes}_\pi F)^* = \mathfrak{Bil}[E,F] \quad\text{isometrically}.
\end{equation}
The algebraic tensor product $E^* \otimes F$ can be considered as a space of \emph{finite-rank operators} $\mathfrak{F}[E,F] \subset \mathfrak{L}[E,F]$, where
$ y^* \otimes x \mapsto (v \mapsto \brt{v,y^*} x)$. This embedding then continuously extends to $E^* \widehat{\otimes}_\pi F \to \mathfrak{L}[E,F]$, but the latter
is not injective in general.
\begin{definition}
  The image of $E^* \widehat{\otimes}_\pi F  \to \mathfrak{L}[E,F]$ with the corresponding factor norm is denoted by $\mathfrak{N}[E,F]$,
  and the operators from $\mathfrak{N}[E,F]$ are
  called \emph{nuclear operators}.
\end{definition}

Analogously to $E^* \otimes F = \mathfrak{F}[E,F]$, one can identify $E \otimes F \subset \mathfrak{Bil}[E^*,F^*]$ as the space of finite-rank
bilinear forms. The induced norm will be denoted by $\norm{\cdot}_\epsilon$.
\begin{definition}
  The completion of $E \otimes_\epsilon F := (E \otimes F, \norm{\cdot}_{\epsilon})$ is denoted by $E \widehat{\otimes}_\epsilon F$ and is called the \emph{injective tensor product}
  of $E$ and $F$.
\end{definition}
The universal property of the projective norm yields a natural map $E \widehat{\otimes}_\pi F \to E \widehat{\otimes}_\epsilon F$.
Because of the isometric embedding $\mathfrak{L}[E,F] \hookrightarrow \mathfrak{L}[F^*,E^*] \cong \mathfrak{Bil}[E^{**},F^*]$, given by taking adjoints, and
the discussion above, it follows that $E^*\widehat{\otimes}_\epsilon F$ is simply the closure $\overline{\mathfrak{F}[E,F]}$ of finite-rank operators
in $\mathfrak{L}[E,F]$ when $E,F$ are Banach spaces.

Let $i \in \brc{1,2}$, $\alpha \in \brc{\epsilon, \pi}$ and $T_i \in \mathfrak{L}[E_i,F_i]$. We can then construct a linear map
$T_1 \widehat{\otimes}_\alpha T_2 \colon E_1 \widehat{\otimes}_\alpha E_2 \to F_1 \widehat{\otimes}_\alpha F_2$, which is continuous and
has the norm
\begin{equation}\label{ineq:tens-norm-operator}
  \norm{T_1 \widehat{\otimes}_\alpha T_2} \leq \norm{T_1} \norm{T_2}.
\end{equation}
The following diagram also commutes:
\begin{equation}
  \begin{tikzcd}
	E_1 \widehat{\otimes}_\alpha E_2 && F_1 \widehat{\otimes}_\alpha F_2 \\
	\mathfrak{L}[E_1^*, E_2] && \mathfrak{L}[F_1^*, F_2]
	\arrow["{T_1\widehat{\otimes}_\alpha T_2}", from=1-1, to=1-3]
	\arrow[from=1-1, to=2-1]
	\arrow[from=1-3, to=2-3]
	\arrow["{A \mapsto T_2AT_1^*}", from=2-1, to=2-3]
\end{tikzcd}.
\end{equation}
For this reason, we will use the operator notation $T_2 \t T_1^*$ as a shorthand for $(T_1 \widehat{\otimes}_\alpha T_2) (\t)$.
Also, note that $\norm{T_2 \t T_1^*} \leq \norm{T_1} \norm{T_2}\norm{\t}_\alpha $ regardless of whether we view $\t$ as an operator or as
a tensor.

\begin{definition}
  A normed space $E$ is said to have the
  \begin{itemize}
    \item \emph{approximation property} if for every absolutely convex compact $K \subset E$ and $\epsilon > 0$ there exists a $T \in \mathfrak{F}[E]$
    with $\sup_{x\in K} \norm{Tx - x} \leq \epsilon$;
    \item \emph{bounded approximation property} if there exists a constant $C$ such that $E$ has the approximation property with $\norm{T} \leq C$.
  \end{itemize}
\end{definition}
On bounded sets of operators, the pointwise convergence coincides with the convergence on absolutely convex compact sets.
Therefore, if $E$ is separable, the bounded approximation property is equivalent to the existence of a constant $C$ and a sequence
$T_n \in \mathfrak{F}[E]$ such that
$T_n \oset[0.5pt]{s}{\to} \id_E$ and $\norm{T_n}\leq C$; otherwise, one replaces the sequence by a net.

For example, it is clear that a Banach space with a (Schauder) basis has the bounded approximation property. In particular, Hilbert spaces have it.
Moreover, the (bounded) approximation property passes to complemented subspaces (via inclusion and projection).

The following proposition tells us that for spaces with the approximation property, we can identify tensors with operators if it is useful for our purposes.
\begin{proposition}[{\cite[Corollaries~5.3,~5.7.1]{Defant-Floret:1993:tensor-norms}}]
  Let $E,F$ be Banach spaces and $\mathfrak{K}[E,F]$ be the space of compact operators. If either $E^*$ or $F$ has the approximation property, then
  \begin{itemize}
    \item $\mathfrak{K}[E,F] =  E^*\widehat{\otimes}_\epsilon F$;
    \item $\mathfrak{N}[E,F] =  E^*\widehat{\otimes}_\pi F$.
  \end{itemize}
\end{proposition}
\begin{proposition}\label{prop:compact-maps}
  Let $i \in \brc{1,2}$, $\alpha \in \brc{\epsilon, \pi}$ and $T_i \colon E_i \to F_i$ be compact linear maps of normed spaces.
  If for each $i\in\brc{1,2}$, either $E_i^*$ or $F_i$ has the approximation property, then
  the induced map $T_1 \widehat{\otimes}_\alpha T_2 \colon E_1 \widehat{\otimes}_\alpha E_2 \to F_1 \widehat{\otimes}_\alpha F_2$
  is compact.
\end{proposition}
\begin{proof}
  By the previous proposition, we have $T_i \in \mathfrak{K}[E_i,F_i] = \overline{\mathfrak{F}[E_i,F_i]}$. Then the norm estimate~\eqref{ineq:tens-norm-operator}
  shows that $T_1 \widehat{\otimes}_\alpha T_2 \in \overline{\mathfrak{F}[E_1 \widehat{\otimes}_\alpha E_2,F_1 \widehat{\otimes}_\alpha F_2]}$,
  which is included in $\mathfrak{K}[E_1 \widehat{\otimes}_\alpha E_2,F_1 \widehat{\otimes}_\alpha F_2]$.
\end{proof}

\begin{proposition}\label{prop:injective-maps}
  Let $i \in \brc{1,2}$, $T_i \in \mathfrak{L}[E_i,F_i]$ be injective maps of Banach spaces. Then
  \begin{itemize}
    \item $T_1 \widehat{\otimes}_\epsilon T_2$ is injective;
    \item if there exist $i,i'\in\brc{1,2}$ such that both $E_i$ and $F_{i'}$ have the approximation property, $T_1 \widehat{\otimes}_\pi T_2$
    is injective as well.
  \end{itemize}
\end{proposition}
\begin{proof}
  The injectivity of $T_1 \widehat{\otimes}_\epsilon T_2$ follows from \cite[Proposition~4.3.(2)]{Defant-Floret:1993:tensor-norms}.
  Under the second assumption and \cite[Theorem~5.6.(d)]{Defant-Floret:1993:tensor-norms}, we have that the maps
  $E_1 \widehat{\otimes}_\pi E_2 \to E_1 \widehat{\otimes}_\epsilon E_2$ and $F_1 \widehat{\otimes}_\pi F_2 \to F_1 \widehat{\otimes}_\epsilon F_2$
  are injective. Then the injectivity of $T_1 \widehat{\otimes}_\pi T_2$ follows from the commutative diagram:
  \begin{equation}
     \begin{tikzcd}
      E_1 \widehat{\otimes}_\pi E_2 && F_1 \widehat{\otimes}_\pi F_2 \\
      E_1 \widehat{\otimes}_\epsilon E_2 && F_1 \widehat{\otimes}_\epsilon F_2
      \arrow["{T_1\widehat{\otimes}_\pi T_2}", from=1-1, to=1-3]
      \arrow[from=1-1, to=2-1]
      \arrow[from=1-3, to=2-3]
      \arrow["{T_1\widehat{\otimes}_\epsilon T_2}", from=2-1, to=2-3]
    \end{tikzcd}.
  \end{equation}
\end{proof}

To pass to weakly$^*$ converging subsequences, it is important to know when $E^* \widehat{\otimes}_\pi F^*$ has a predual.
Let $L^p(\mu) = L^p(X,\mu)$ denote a Banach space of $p$-integrable (complex- or real-valued) functions over a measure space $(X,\mu)$.
\begin{definition}
  A Banach space $E$ is said to have the \emph{Radon-Nikodým property} if for each finite measure $\mu$, every operator $T \in \mathfrak{L}[L^1(X,\mu), E]$ is representable
  by a measurable function $f\colon X \to E$, i.e. $T\phi = \int f\phi d\mu$.
\end{definition}
\begin{proposition}[{\cite[Appendix~D3]{Defant-Floret:1993:tensor-norms}}]
  Let $E$ be a Banach space. If $E^*$ is either separable or reflexive then $E^*$ has the Radon-Nikodým property.
\end{proposition}
\begin{proposition}[{\cite[Theorem~16.6]{Defant-Floret:1993:tensor-norms}}]\label{prop:tensor-duality}
  Let $E$ and $F$ be Banach spaces.
  If either $E^*$ or $F^*$ has the approximation property, and either $E^*$ or $F^*$ has the Radon-Nikodým property, then
  \begin{equation}
    (E \widehat{\otimes}_\epsilon F)^* = E^* \widehat{\otimes}_\pi F^* \quad \text{isometrically.}
  \end{equation}
\end{proposition}
We conclude with the following elementary observation.
\begin{proposition}\label{prop:weak-star-conv}
  Let $i\in \brc{1,2}$, and let $T_i \in \mathfrak{L}[E_i,F_i]$ be bounded linear maps between reflexive Banach spaces
  with the approximation property. Then
  \begin{equation}
    T_1 \widehat{\otimes}_\pi T_2 = (T_1^* \widehat{\otimes}_\epsilon T_2^*)^*,
  \end{equation}
  and hence $T_1 \widehat{\otimes}_\pi T_2$ is weakly$^*$ continuous.
\end{proposition}

\subsubsection{Tensor products of Hilbert spaces}
Recall that for a vector space $E$, $\overbar{E}$ denotes its complex conjugate; that is, scalar multiplication is defined by
$\alpha \cdot x := \overline{\alpha} x$ for $x \in \overbar{E}$. Note that $\overbar{E}$ is (anti-)linearly isometric
to $E$, depending on whether or not there exists an antilinear involution ($x \mapsto \overbar{x}$). Also,
$\mathfrak{L}[E,F] = \mathfrak{L}[\overbar{E},\overbar{F}]$ and $\mathfrak{L}[\overbar{E},F] = \mathfrak{L}[E,\overbar{F}]$ as sets,
and $x^* \in (\overbar{E})^* = \overbar{E^*}$ acts as $\brt{y,x^*}_{\overbar{E},\overbar{E^*}} := \overbar{\brt{y,x^*}}_{E,E^*}$.

Let $H$ be a Hilbert space over $\mathbb{K}$. We identify $H^* \cong \overbar{H}$, where $\overbar{H}$ is the complex conjugate; that is,
$x$ is identified with $x^*(y) = \brt{x,y}_H$ (clearly, $\overbar{H} = H$ if $\mathbb{K} = \mathbb{R}$). Under this identification, the sesquilinear
inner product on $H\times H$
is viewed as the bilinear pairing $\brt{\cdot,\cdot}\colon \overbar{H} \times H \to \mathbb{K}$, and we have the following isometries:
\begin{itemize}
  \item $\mathfrak{Bil}[\overbar{H}, H] \cong \mathfrak{L}[H] \cong \mathfrak{Bil}[H, \overbar{H}]$;
  \item $\overbar{H} \widehat{\otimes}_{\epsilon} H \cong \mathfrak{K}[H]$;
  \item $\overbar{H}\widehat{\otimes}_{\pi} H \cong \mathfrak{N}[H]$.
\end{itemize}

Moreover, one checks that for
$\t \in \mathfrak{N}[H]$, the projective norm is actually the trace norm:
\begin{equation}
  \norm{\t}_\pi = \norm{\t}_{\mathfrak{N}[H]} = \Tr \sqrt{\t^* \t}.
\end{equation}
In particular, for a positive (self-adjoint) tensor $\t = \sum x_i \otimes x_i$,
the norm is simply the trace:
\begin{equation}\label{eq:positive-tens-norm}
  \norm{\t}_{\pi} = \Tr \t = \sum \norm{x_i}^2.
\end{equation}
\begin{lemma}\label{lem:tensor-decomp}
  Let $E$ be a normed space, $\boldsymbol{\iota}\colon E \to H$ and $\boldsymbol{\iota}_\epsilon\colon E \to H_\epsilon$
  be continuous injections into Hilbert spaces such that
  $\norm{\boldsymbol{\iota}_\epsilon} \leq C$ and $\boldsymbol{\iota}_\epsilon^*\boldsymbol{\iota}_\epsilon \oset[0.5pt]{s}{\to} \boldsymbol{\iota}^*\boldsymbol{\iota} \in \mathfrak{L}[E,\overbar{E}^*]$.
  If $\t_\epsilon, \t \in \overbar{E}\widehat{\otimes}_{\pi}E$ with $\t_\epsilon \oset{pt}{\to} \t $, then
  there exists a collection of finite-rank orthogonal projectors $\mathfrak{p}_\epsilon \in \mathfrak{L}[H_\epsilon]$
  such that
  $\norm{\mathfrak{p}_\epsilon (\boldsymbol{\iota}_\epsilon\t_\epsilon\boldsymbol{\iota}_\epsilon^*) \mathfrak{p}_\epsilon - \boldsymbol{\iota}_\epsilon\t\boldsymbol{\iota}_\epsilon^*}_{\pi} \to 0$.
\end{lemma}
\begin{proof}
  We may write $\t = \sum x_i \otimes y_i$ with
  $\sum \norm{x_i}\norm{y_i} < \infty$. For each $n \in \mathbb{N}$, we decompose $\t$ as
  \begin{equation}
    \t = \sum_{i \leq N} x_i \otimes y_i + \mathfrak{r},
  \end{equation}
  so that $\norm{\mathfrak{r}}_{\overbar{E}\widehat{\otimes}_{\pi}E} < 1/n$. Let $\mathfrak{p}_{n,\epsilon} \in \mathfrak{L}[H_\epsilon]$ be the orthogonal
  projector onto $\boldsymbol{\iota}_\epsilon(V_n)$, where $V_{n} = Span \set{x_i, y_i}{i \leq N}$. On the one hand,
  \begin{equation}
    \norm{\mathfrak{p}_{n,\epsilon} (\boldsymbol{\iota}_\epsilon\t\boldsymbol{\iota}_\epsilon^*) \mathfrak{p}_{n,\epsilon} - (\boldsymbol{\iota}_\epsilon\t\boldsymbol{\iota}_\epsilon^*)}_{\pi} \leq
    \norm{\mathfrak{p}_{n,\epsilon} (\boldsymbol{\iota}_\epsilon\mathfrak{r}\boldsymbol{\iota}_\epsilon^*) \mathfrak{p}_{n,\epsilon}}_{\pi} +  \norm{\boldsymbol{\iota}_\epsilon\mathfrak{r}\boldsymbol{\iota}_\epsilon^*}_{\pi} \leq \frac{2C^2}{n}.
  \end{equation}
  Now, fix $n$ and choose $\brc{e_i} \subset V_n $ such that $\brc{\boldsymbol{\iota}(e_i)} \subset \boldsymbol{\iota}(V_n)$  forms an orthonormal basis. Then the convergences
  $\t_\epsilon \overset{pt}{\to} \t$ and $\boldsymbol{\iota}_\epsilon^*\boldsymbol{\iota}_\epsilon \oset[0.5pt]{s}{\to} \boldsymbol{\iota}^*\boldsymbol{\iota} \in \mathfrak{L}[E,\overbar{E}^*]$ yield
  \begin{equation}
    \brt{\mathfrak{p}_{n,\epsilon} (\boldsymbol{\iota}_\epsilon(\t_\epsilon - \t)\boldsymbol{\iota}_\epsilon^*) \mathfrak{p}_{n,\epsilon}\boldsymbol{\iota}_\epsilon(e_i),\boldsymbol{\iota}_\epsilon(e_j)}_{H_\epsilon}
    = \brt{\t_\epsilon - \t, (\boldsymbol{\iota}_\epsilon^*\boldsymbol{\iota}_\epsilon)e_i \otimes (\boldsymbol{\iota}_\epsilon^*\boldsymbol{\iota}_\epsilon)e_j}
    \to 0
  \end{equation}
  and
  \begin{equation}
    \brt{\boldsymbol{\iota}_\epsilon(e_i),\boldsymbol{\iota}_\epsilon(e_j)}_{H_\epsilon} = \brt{(\boldsymbol{\iota}_\epsilon^*\boldsymbol{\iota}_\epsilon)e_i,e_j}
    \to \brt{(\boldsymbol{\iota}^*\boldsymbol{\iota})e_i,e_j} = \brt{\boldsymbol{\iota}(e_i),\boldsymbol{\iota}(e_j)}_{H} = \delta_{ij}.
  \end{equation}
  Thus, $\brc{\boldsymbol{\iota}_\epsilon(e_i)} \subset \boldsymbol{\iota}_\epsilon(V_n)$ is an almost orthonormal basis, and for each fixed $n$, the matrix elements of
  $\mathfrak{p}_{n,\epsilon} (\boldsymbol{\iota}_\epsilon\t_\epsilon\boldsymbol{\iota}_\epsilon^*) \mathfrak{p}_{n,\epsilon} - \mathfrak{p}_{n,\epsilon} (\boldsymbol{\iota}_\epsilon\t\boldsymbol{\iota}_\epsilon^*) \mathfrak{p}_{n,\epsilon}$ converge to zero as
  $\epsilon \to 0$.
  Therefore, one can choose a decreasing sequence $\brc{\epsilon_n}$ so that for $\epsilon \leq \epsilon_n$, one has
  \begin{equation}
    \norm{\mathfrak{p}_{n,\epsilon} (\boldsymbol{\iota}_\epsilon\t_\epsilon\boldsymbol{\iota}_\epsilon^*) \mathfrak{p}_{n,\epsilon} -
    \mathfrak{p}_{n,\epsilon} (\boldsymbol{\iota}_\epsilon\t\boldsymbol{\iota}_\epsilon^*) \mathfrak{p}_{n,\epsilon} }_{\pi} \leq \frac{1}{n}.
  \end{equation}
  By setting $\mathfrak{p}_{\epsilon} := \mathfrak{p}_{n,\epsilon}$ for $\epsilon \in (\epsilon_{n+1}, \epsilon_n]$, we obtain
  \begin{equation}
    \norm{\mathfrak{p}_\epsilon (\boldsymbol{\iota}_\epsilon\t_\epsilon\boldsymbol{\iota}_\epsilon^*) \mathfrak{p}_\epsilon - \boldsymbol{\iota}_\epsilon\t\boldsymbol{\iota}_\epsilon^*}_{\pi}
    \leq \frac{2C^2 + 1}{n} \to 0.
  \end{equation}
\end{proof}

A proof of the next proposition can be found, for example, in the appendix of~\cite*{Arazy:1981:convergence--unitary-matrix-spaces}.

\begin{proposition}\label{prop:nuclear-norm-ineq}
  For a nuclear operator $ \t \in \mathfrak{N}[H]$ on a Hilbert space $H$ and a pair
  of orthogonal projections $\mathfrak{p} + \mathfrak{q} = \id$, one has
  \begin{equation}
    \norm{\mathfrak{p}\t\mathfrak{p}}_{\pi}^2+\norm{\mathfrak{p}\t\mathfrak{q}}_{\pi}^2 +
    \norm{\mathfrak{q}\t\mathfrak{p}}_{\pi}^2+\norm{\mathfrak{q}\t\mathfrak{q}}_{\pi}^2 \leq \norm{\t}_{\pi}^2.
  \end{equation}
\end{proposition}
\begin{corollary}\label{cor:seq-spaces-tensor-convergence}
  Under the conditions of Lemma~\ref{lem:tensor-decomp}, let us additionally assume that
  $\norm{\boldsymbol{\iota}_\epsilon\t_\epsilon\boldsymbol{\iota}_\epsilon^*}_{\pi} - \norm{\boldsymbol{\iota}_\epsilon\t\boldsymbol{\iota}_\epsilon^*}_{\pi} \to 0$. Then
  $\norm{\boldsymbol{\iota}_\epsilon(\t_\epsilon - \t)\boldsymbol{\iota}_\epsilon^*}_{\pi} \to 0$.
\end{corollary}
\begin{proof}
  By Lemma~\ref{lem:tensor-decomp}, there exist orthogonal projectors $\mathfrak{p}_\epsilon \in \mathfrak{L}[H_\epsilon]$
  such that
  $\norm{\mathfrak{p}_\epsilon (\boldsymbol{\iota}_\epsilon\t_\epsilon\boldsymbol{\iota}_\epsilon^*) \mathfrak{p}_\epsilon - \boldsymbol{\iota}_\epsilon\t\boldsymbol{\iota}_\epsilon^*}_{\pi} \to 0$.
  Then by Proposition~\ref{prop:nuclear-norm-ineq} (with $\mathfrak{q}_\epsilon = \id - \mathfrak{p}_\epsilon$), one has
  \begin{multline}
    \norm{\boldsymbol{\iota}_\epsilon(\t_\epsilon - \t)\boldsymbol{\iota}_\epsilon^*}_{\pi}
    \leq \norm{\mathfrak{p}_\epsilon (\boldsymbol{\iota}_\epsilon\t_\epsilon\boldsymbol{\iota}_\epsilon^*) \mathfrak{p}_\epsilon - \boldsymbol{\iota}_\epsilon\t\boldsymbol{\iota}_\epsilon^*}_{\pi}
    \\+ \sqrt{3(\norm{\boldsymbol{\iota}_\epsilon\t_\epsilon\boldsymbol{\iota}_\epsilon^*}_{\pi}^2 - \norm{\mathfrak{p}_\epsilon(\boldsymbol{\iota}_\epsilon\t_\epsilon\boldsymbol{\iota}_\epsilon^*)\mathfrak{p}_\epsilon}_{\pi}^2)}
    \to 0.
  \end{multline}
\end{proof}
By taking $E = H_\epsilon = H$ in the previous corollary, we obtain:
\begin{corollary}\label{cor:tensor-convergence}
  Let $\t_\epsilon, \t \in \mathfrak{N}[H]$ with $\t_\epsilon \overset{\operatorname{wot}}{\to} \t$
  and $\norm{\t_\epsilon}_{\pi} \to \norm{\t}_{\pi}$. Then
  $\t_\epsilon \to \t$ in $\mathfrak{N}[H]$.
\end{corollary}

Let $\ell^2$ be the space of square-summable sequences.
Denote by $\mathcal{HS}[H,\ell^2]$ the space of Hilbert--Schmidt operators between Hilbert spaces $H$ and $\ell^2$.
\begin{lemma}\label{lem:HS-isom}
  Let $A, A_n \in \mathcal{HS}[H,\ell^2]$. Then
  \begin{equation}
    \norm{(A^*A)^{1/2}-(A_n^*A_n)^{1/2}}_{\mathcal{HS}} \leq
    \norm{A^*A-A_n^*A_n}_{\pi}^{1/2},
  \end{equation}
  and there exist isometric embeddings $I \colon \operatorname{Im} A \hookrightarrow \ell^2$, $I_n \colon \operatorname{Im} A_n \hookrightarrow \ell^2$
  such that
  \begin{equation}
    \norm{I_nA_n-IA}_{\mathcal{HS}} = \norm{(A^*A)^{1/2}-(A_n^*A_n)^{1/2}}_{\mathcal{HS}}.
  \end{equation}
\end{lemma}
\begin{proof}
  For a proof of the inequality, see, for example, \cite[Lemma~4.1]{Powers-Stormer:1970:square-root-ineq}.
  By polar decomposition, one has $A = I (A^*A)^{1/2}$ for an isometry $I \colon \operatorname{Im} \sqrt{A^*A} \to \operatorname{Im} A$.
  Then the isometries $I^{-1}, I_n^{-1}$, together with a choice of an embedding $\overline{Span}\, (\operatorname{Im}  A^*A + \sum_n\operatorname{Im}  A_n^*A_n) \hookrightarrow \ell^2$ prove the last statement.
\end{proof}

\subsubsection{Vector-valued Sobolev spaces}
Let $E$ be a Banach space and $p \in [1, \infty]$. We denote by $L^p(\mu, E) := L^p(X,\mu, E)$ the space of $E$-valued
(Bochner) $p$-integrable functions on a measure space $(X,\mu)$. See, for example, \cite[Appendix~B]{Defant-Floret:1993:tensor-norms} and
\cite[Section~1]{Hytonen-Neerven-Veraar-Weis:2016:analysis-banach-1} for the definition and basic properties.
In particular (see~\cite[Corollary~1.3.13]{Hytonen-Neerven-Veraar-Weis:2016:analysis-banach-1}), if $E^*$ has the Radon-Nikodým property and
$p \in (1,\infty)$, then $L^p(\mu,E)^* = L^{p'}(\mu, E^*)$, where $1/p + 1/p' = 1$. For $p=1$, one needs $\mu$ to be $\sigma$-finite.

Let $\Omega \subset \R^m$ be a domain. One naturally defines
the $E$-valued Sobolev spaces $W^{k,p}(\Omega,E)$, where $k \in \mathbb{N}$.
Some properties of $W^{k,p}(\Omega,E)$ can be found, for example, in~\cite[Section~2.5.b]{Hytonen-Neerven-Veraar-Weis:2016:analysis-banach-1}. The usual difference quotient criterion applies:
\begin{proposition}[{\cite[Theorem~2.2]{Arendt-Kreuter:2018:Lipsh-vect-Soboelv}}]
  \label{prop:diff-quot}
  Let $E$ be a Banach space with the Radon-Nikodým property, $p \in (1,\infty]$, and $f\in L^p(\R^m,E)$. Then
  \begin{equation}
    \sup_{\substack{0<t\leq1 \\ 1\leq i \leq m}} \frac{1}{t}\norm{f(\cdot + te_i) - f(\cdot)}_{L^p} \leq C \implies
    f \in W^{1,p}(\R^m,E),\ \text{with} \ \norm{\partial_i f}_{L^p} \leq C.
  \end{equation}
\end{proposition}
Note that this is a local criterion.
Another useful property is that Sobolev regularity is preserved under composition with Lipschitz maps. Combining \cite[Theorems~3.1,~3.2,~4.2 and Corollary~4.6]{Arendt-Kreuter:2018:Lipsh-vect-Soboelv},
we obtain:
\begin{proposition}\label{prop:lip-sobolev-comp}
  Let $p \in [1,\infty]$, $T \colon E \to F$ be a Lipschitz mapping between Banach spaces with a Lipschitz constant $L$,
  $f \in W^{1,p}(\Omega,E)$,
  and assume that either $\Vol (\Omega) < \infty$ or $T(0) = 0$.
   \begin{itemize}
    \item If $F$ has the Radon-Nikodým property, then $T \circ f \in W^{1,p}(\Omega,F)$, and the weak derivatives satisfy
    \begin{equation}
      \abs{\partial_i(T \circ f)}_F \leq L\abs{\partial_i f}_E.
    \end{equation}
    \item If $T$ is one-sided Gâteaux differentiable, then $\partial_i(T \circ f) = D_{\partial_i f} T(f)$.
    In particular, $\abs{f}_E \in W^{1,p}(\Omega)$, with $\abs{\partial_i\abs{f}_E} \leq \abs{\partial_i f}_E$.
  \end{itemize}
\end{proposition}
Because these results are local, they extend naturally to $W^{1,p}(\Omega,E)$ when $\Omega$ is a compact manifold (possibly with boundary);
in that case, one replaces partial derivatives by derivatives along vector fields.

We begin with the following extension lemma, whose proof can be found, for example,
in \cite[Theorem~2.1.9]{Hytonen-Neerven-Veraar-Weis:2016:analysis-banach-1}.
\begin{lemma}
  \label{lem:scalar-to-vector}
  Let $p, q \in [1,\infty)$, and let $T\colon L^q(\mu) \to L^p(\nu)$ be a continuous linear operator.
  Consider
  $L^q(\mu) \otimes \ell^2 \cong Span \set{f(\cdot)v}{f\in L^q(\mu),\ v \in \ell^2}$,
  which is a dense subset of $L^q(\mu,\ell^2)$. The operator $T\otimes \id_{\ell^2} \colon
  L^q(\mu) \otimes \ell^2 \to L^p(\nu)\otimes \ell^2$ then extends continuously to
  \begin{equation}
    T\otimes \id_{\ell^2} \colon L^q(\mu, \ell^2) \to L^p(\nu,\ell^2).
  \end{equation}
\end{lemma}
\begin{remark}
  In fact, $L^p(\mu,\ell^2)$ coincides with the completion of the algebraic tensor product $L^p(\mu) \otimes \ell^2$
  with a certain tensor norm (see~\cite[Proposition~25.10]{Defant-Floret:1993:tensor-norms}).
\end{remark}

Let $\mathcal{S}'(\R^m, E) = \mathfrak{L}[\mathcal{S}(\R^m), E]$ be the space of $E$-valued Schwartz distributions, and let
$f \in \mathcal{S}'(\R^m, E)$. Define the Bessel potential operator $J_s$ by $J_sf := \mathcal{F}^{-1}[(1+\abs{\xi}^2)^{s/2}\mathcal{F}f]$, where $\mathcal{F}f$ is
the Fourier transform of $f$.
Let $p \in [1,\infty]$ and $s\in \R$. Define
\begin{equation}
  \begin{split}
    H^{s,p}(\R^m,E) &:= \set*{f \in \mathcal{S}'(\R^m, E) }{ J_{s}f \in L^p(\R^m, E)},
    \\ H^{s}(\R^m,E) &:= H^{s,2}(\R^m,E).
  \end{split}
\end{equation}
As in the scalar case, the inverse of $J_{s}$ on $\mathcal{S}'(\R^m, E)$ is $J_{-s}$, so that $H^{s,p}(\R^m,E)$ becomes isomorphic to $L^p(\R^m,E)$.
For $p \in (1,\infty)$ and $k \in \mathbb{N}$, we also have
$W^{k,p}(\R^m) = H^{k,p}(\R^m)$. This equivalence can be partially extended to Sobolev spaces
with values in UMD spaces (see~\cite[Theorem~5.6.11]{Hytonen-Neerven-Veraar-Weis:2016:analysis-banach-1}). In particular, if $p \in (1,\infty)$, then
\begin{equation}\label{eq:Sobolev-is-Bessel}
  W^{k,p}(\R^m,\ell^2) = H^{k,p}(\R^m,\ell^2)\quad \text{(with equivalent norms)}.
\end{equation}
For $\ell^2$-valued functions, one also has corresponding complex interpolation results and so on (see~\cite[Section~5]{Hytonen-Neerven-Veraar-Weis:2016:analysis-banach-1}).
In the present paper, we require only the spaces $H^{1,p}(\Omega,\ell^2)$, but one can also define $H^{s,p}(\Omega,\ell^2)$ for any $s\in\R$ when $\Omega$
is a compact manifold (possibly with boundary), exactly as in the scalar case (see, for example,
\cite[Section~4]{Taylor:2023:pde-1} and \cite[Section~13.6]{Taylor:2023:pde-3}).
\begin{remark}
  For $f\in W^{1,p}(\Omega,\ell^2)$, the weak differential $d_xf\in\mathfrak L[T_x\Omega,\ell^2]$ is defined for almost every $x$,
  and we set
  \begin{equation}
    \abs{d_xf}^2 := \sum_i \abs{d_xf(e_i)}^2_{\ell^2},
  \end{equation}
  Thus, $\abs{d_xf}$ is the Hilbert--Schmidt norm of $d_xf$ and is therefore independent of the choice of local orthonormal basis $\{e_i\}$.
\end{remark}
From now on, we will be working only with $\ell^2$-valued functions and will use the advantages of the equality~\eqref{eq:Sobolev-is-Bessel}.
Therefore, we will retain the notation $H^{k,p}$ even for nonsmooth domains $\Omega \subset M$, defining $H^{k,p} := W^{k,p}$ in that case.

\subsection{Applications of tensor products to Sobolev spaces}\label{sec:applic-tensor-prod}
The following lemma allows us to transfer the topological properties of $L^p(\mu,\ell^2)$ spaces to the Sobolev spaces $H^{s,p}(\Omega,\ell^2)$.
\begin{lemma}\label{lem:sobolev-lp-isom}
  Let $p \in (1,\infty)$ and $s \in \R$, and let $\Omega$ be a compact Riemannian manifold (possibly with boundary). Then
  there exists a (Banach-space) isomorphism $J \colon H^{s,p}(\Omega) \to L^p(I)$, where $I = (0,1)$.
  Furthermore, $J \otimes \id_{\ell^2}$ extends to a (Banach-space) isomorphism $J \otimes \id_{\ell^2} \colon H^{s,p}(\Omega,\ell^2) \to L^p(I,\ell^2)$.

  As a consequence, $H^{s,p}(\Omega)$ is reflexive, has the bounded approximation property, and hence
  $H^{s,p}(\Omega)\widehat{\otimes}_{\pi}H^{s,p}(\Omega)$ is a dual space.
\end{lemma}
\begin{proof}
  First, $L^p(\mu)$ spaces have the bounded approximation property, as one can see by considering
  the net of conditional expectations projections with respect to finite partitions.

  If $\partial \Omega = \varnothing$ (so $\Omega$ is a closed manifold, which we denote by $M$), then
  \begin{equation}
    (1+\Delta)^{s/2} \otimes \id_{\ell^2} \colon H^{s,p}(M,\ell^2) \to L^p(M,\ell^2)
  \end{equation}
  is an isomorphism by elliptic regularity and Lemma~\ref{lem:scalar-to-vector}
  (alternatively, by partition of unity and Bessel potentials).
  Then $L^p(M) \approx L^p(I)$, since $L^p(M)$ is separable, and the volume measure is nonatomic. Again,
  by Lemma~\ref{lem:scalar-to-vector}, we obtain $L^p(M,\ell^2) \approx L^p(I,\ell^2)$.

  Finally, if $\partial \Omega \neq \varnothing$ and $M$ denotes its double, then there is a scalar
  extension operator $H^{s,p}(\Omega,\ell^2) \to H^{s,p}(M,\ell^2)$, that is, of the form $\operatorname{Ext}\otimes \id_{\ell^2}$. As a consequence, $H^{s,p}(\Omega,\ell^2)$ is a complemented subspace of $H^{s,p}(M,\ell^2) \approx L^p(I,\ell^2)$.
  It remains to show that $H^{s,p}(\Omega,\ell^2)$ contains a complemented subspace isomorphic to $L^p(I,\ell^2)$. One may then apply Pełczyński's decomposition method;
  see~\cite[Theorem~11]{Pelczynski-Wojciechowski:2003:sobolev-spaces}
  and~\cite[Section~4]{Johnson-Lindenstrauss:2001:basic-concepts} for details. These identifications are scalar in nature, so the resulting isomorphism is scalar as well.
\end{proof}

\begin{lemma}\label{lem:vect-lp-proj-bound}
  Let $H^{s,p}(\Omega, \ell^2)$ be as in Lemma~\ref{lem:sobolev-lp-isom}, and let $f,g \in H^{s,p}(\Omega, \ell^2)$, with coordinate functions $f^i$ and $g^i$. Then there exists a constant $C > 0$
  such that
  \begin{equation}
    \norm{\sum\nolimits_{i} f^i \otimes g^i}_{H^{s,p}(\Omega)\widehat{\otimes}_\pi H^{s,p}(\Omega)} \leq C \norm{f}_{H^{s,p}(\Omega, \ell^2)}\norm{g}_{H^{s,p}(\Omega, \ell^2)}.
  \end{equation}
\end{lemma}
\begin{proof}
  By Lemma~\ref{lem:sobolev-lp-isom}, it is enough to prove the statement for $L^p(\mu,\ell^2)$ spaces.
  It suffices to prove the inequality for the partial sums $\sum_{i \leq N} f^i \otimes g^i$, so only finitely many of $f^i$ and $g^i$ are nonzero, and each
  $f^i = \sum_j f^i_j \chi_{S_j}$, $g^i = \sum_j g^i_j \chi_{S_j}$ is a sum of step functions with $S_{j'} \cap S_j = \varnothing$. If $\t := \sum_i f^i \otimes g^i$,
  it suffices to show that
  $\brt{\t, \mathfrak{q}} \leq C\norm{f}_{L^p(\mu; \ell^2)}\norm{g}_{L^p(\mu; \ell^2)}$ for any $\mathfrak{q} \in \mathfrak{Bil}[L^p(\mu)]$ with $\norm{\mathfrak{q}} \leq 1$.
  Let $B_E$ denote the unit ball of the normed space $E$.
  We will use Grothendieck's inequality in matrix form (see~\cite[Section~14.5]{Defant-Floret:1993:tensor-norms}):
  there exists a constant $C$ such that for any $n\times n$ matrix $(a_{ij}) \in \mathfrak{L}[\mathbb{K}^n]$ and each Hilbert space $H$,
  \begin{equation}\label{ineq:grothendieck}
    \begin{multlined}
      0 \leq \sup \set*{\sum\nolimits_{i,j} a_{ij} \brt{x_i,y_j}}{x_i, y_j \in B_H}
      \\ \leq C \sup \set*{\sum\nolimits_{i,j} a_{ij}s_i t_j}{ s_i, t_j \in B_{\mathbb{K}}}.
    \end{multlined}
  \end{equation}
  Let $\abs{f}_j = (\sum_i |f^i_j|^2)^{1/2}$, i.e. $\abs{f} = \sum_j \abs{f}_j \chi_{S_j}$. Then
  \begin{equation}
    \brt{\t, \mathfrak{q}} = \sum_{j,k} \mathfrak{q}\brs{\abs{f}_j \chi_{S_j}, \abs{g}_k \chi_{S_k}} \brr{\sum_i \frac{f^i_j}{\abs{f}_j}
    \frac{g^i_k}{\abs{g}_k}}.
  \end{equation}
  We apply~\eqref{ineq:grothendieck} with $a_{jk} = \mathfrak{q}\brs{\abs{f}_j \chi_{S_j}, \abs{g}_k \chi_{S_k}}$ to obtain
  \begin{multline}
    \brt{\t, \mathfrak{q}} \leq C \sup_{s_j, t_k \in B_{\mathbb{K}}}\sum\nolimits_{j,k}\mathfrak{q}\brs{\abs{f}_j \chi_{S_j}, \abs{g}_k \chi_{S_k}}s_j t_k
     \\ = C \sup_{s_i, t_j \in B_{\mathbb{K}}}\mathfrak{q}\brs{f_s, g_t} \leq C  \norm{f}_{L^p(\mu; \ell^2)}\norm{g}_{L^p(\mu; \ell^2)},
  \end{multline}
  where $f_s := \sum_j s_j\abs{f}_j \chi_{S_j}$ and $\norm{\mathfrak{q}} \leq 1$.
\end{proof}
\begin{remark}
  One can show that the series $\sum_i f^i \otimes g^i$ converges absolutely (that is, $\sum_i \norm{f^i}\norm{g^i} < \infty$) if $p \leq 2$,
  but this need not hold when $p >2$.
\end{remark}

Let $p \in [2,\infty)$, and denote by
\begin{equation}\label{eq:nucler-top-H1}
  \m{\cdot}\colon \overbar{L}^p(\Omega)\widehat{\otimes}_\pi L^p(\Omega) \to L^{p/2}(\Omega)
\end{equation}
the linear map induced by pointwise multiplication; that is,
$\m{f \otimes g }= \overbar{f}g$. Regarding $d$ as the exterior derivative $d \colon H^{1,p}(\Omega) \to L^p(\Omega, T^*\Omega)$, we also define
$\brt{d(f \otimes g)d^*} := \brt{dg,df}$ for $f,g \in H^{1,p}(\Omega)$.
We keep the same notation $\m{\cdot}$ for the map $\m{\cdot}\colon \overbar{H}^{1,p}\hat{\otimes}_\pi H^{1,p} \to L^{p/2}$ and so on.
Then equation~\eqref{eq:positive-tens-norm} reads as
\begin{equation}\label{eq:tens-norm-decomp}
  \norm{\t}_{\tensor{H^1(\Omega)}} =  \norm{\t}_{\overbar{H}^1(\Omega)\widehat{\otimes}_\pi H^1(\Omega)} = \int_\Omega \m{\t}dv_g + \int_\Omega \m{d\t d^*}dv_g,
\end{equation}
where $\t = \sum \phi^i \otimes \phi^i$ and $H^1 = H^{1,2}$.
\begin{remark}\label{rem:weak-star-conv}
  If $\Omega' \subset \Omega$ and  $\t_\epsilon \overset{w^*}{\to}
  \t \in \overbar{H}^{s,p}(\Omega) \widehat{\otimes}_\pi H^{s,p}(\Omega)$, it follows from Proposition~\ref{prop:weak-star-conv}
  that $\t_\epsilon|_{\Omega'} \oset{w^*}{\to} \t|_{\Omega'} \in \overbar{H}^{s,p}(\Omega') \widehat{\otimes}_\pi H^{s,p}(\Omega')$.
\end{remark}
As a corollary of Proposition~\ref{prop:compact-maps}, we obtain a generalization
of the weak$^*$-to-strong convergence under the compact embedding $H^{1,p} \hookrightarrow L^p$.
\begin{corollary}\label{cor:l-inf-convergence}
  Let $s > 0$ and $\t_\epsilon \overset{w^*}{\to}
  \t$ in $\overbar{H}^{s,p}(\Omega) \widehat{\otimes}_\pi H^{s,p}(\Omega)$. Then
  $\t_\epsilon \to \t$ in $\overbar{L}^{p}(\Omega) \widehat{\otimes}_\pi L^{p}(\Omega)$. In particular,
  $\m{\t_\epsilon} \to \m{\t}$ in $L^{p/2}(\Omega)$.
\end{corollary}
\begin{proof}
  Proposition~\ref{prop:weak-star-conv} implies that $\t_\epsilon \overset{w^*}{\to}
  \t$ in $\overbar{L}^{p}(\Omega) \widehat{\otimes}_\pi L^{p}(\Omega)$ as well.
  Since the map $\overbar{H}^{s,p} \widehat{\otimes}_\pi H^{s,p} \to \overbar{L}^{p} \widehat{\otimes}_\pi L^{p}$ is compact, and
  $\t$ can be the only accumulation point of the sequence $\brc{\t_\epsilon}$, it follows
  that in fact, $\t_\epsilon \to \t$ in $\overbar{L}^{p}(\Omega) \widehat{\otimes}_\pi L^{p}(\Omega)$.
\end{proof}

Given $f \in H^{s,p}(\Omega,\ell^2)$ with coordinate functions $f^i$, we view $f$ as an operator
$f \in \mathfrak{L}[(H^{s,p})^*, \ell^2] = \mathfrak{L}[(\overbar{H}^{s,p})^*, \overbar{\ell^2}]$, acting by
\begin{equation}\label{def:lp-as-operator}
  \phi \mapsto (\cdots, \brt{f^i, \phi}_{H^{s,p},(H^{s,p})^*},\cdots).
\end{equation}
Under such identification, note that $L^2(\mu,\ell^2) = \mathcal{HS}[L^2(\mu), \overbar{\ell^2}]$ isometrically (it is enough to compare the norms).
If $f^* \colon \overbar{\ell^2} \to H^{s,p}$ denotes its dual operator, by Lemma~\ref{lem:vect-lp-proj-bound}, we obtain
\begin{equation}
  \tf{f} = \sum f^i \otimes f^i  \in \overbar{H}^{s,p}\widehat{\otimes}_\pi H^{s,p}  \approx \mathfrak{N}[(\overbar{H}^{s,p})^*,H^{s,p}].
\end{equation}
By $\iml  f$, we denote the range (image) of the operator $f \in \mathfrak{L}[(H^{s,p})^*,\ell^2]$.
The following lemma tells us that the projective convergence on the tensors $\tf{f}$ is equivalent to the convergence up to isometry of the corresponding maps.
\begin{lemma}[Tensor-product convergence toolkit]\label{lem:main}
  Let $H^{s,p}(\Omega, \ell^2)$ be as in Lemma~\ref{lem:sobolev-lp-isom}, and let $p \in [2,\infty)$.
  \begin{enumerate}[label={(\roman*)}]
    \item\label{lem:main:it:0} There exists a constant $C>0$ such that for all $f \in H^{s,p}(\Omega,\ell^2)$, we have
    \begin{equation}
      C^{-1}\norm{f}_{H^{s,p}(\Omega,\ell^2)}^2 \leq \norm{f^*f}_{\overbar{H}^{s,p}(\Omega)\widehat{\otimes}_\pi H^{s,p}(\Omega)} \leq C\norm{f}_{H^{s,p}(\Omega,\ell^2)}^2.
    \end{equation}
    \item\label{lem:main:it:1} The image of $H^{s,p}(\Omega, \ell^2) \ni f \mapsto \tf{f} \in \overbar{H}^{s,p}\widehat{\otimes}_\pi H^{s,p}$ is weakly$^*$ closed
    and consists of all the positive self-adjoint tensors (in the sense of bilinear forms on the duals).
    \item\label{lem:main:it:2} $f_1^*f_1 = f_2^*f_2$ if and only if $f_2 = I f_1$ for some isometry $I \colon \iml  f_1 \to \iml  f_2$.
    \item\label{lem:main:it:3} $f_n^*f_n \to \tf{f}$ in $\overbar{H}^{s,p}\widehat{\otimes}_\pi H^{s,p}$ if and only if there exist isometric embeddings
    $I_n \colon \iml  f_n \hookrightarrow \ell^2$ and $I \colon \iml  f \hookrightarrow \ell^2$ such that
    $I_n f_n \to If$ in $H^{s,p}(\Omega, \ell^2)$.
    \item\label{lem:main:it:4} If a functional $\mathcal{R} \colon \overbar{H}^{s,p}\widehat{\otimes}_\pi H^{s,p} \to \R$ has the form
    $\mathcal{R}(\t) = \sup_{\alpha} \norm{\boldsymbol{\iota}_\alpha\t \boldsymbol{\iota}_\alpha^*}_{\pi}$ for some family of maps $\boldsymbol{\iota}_\alpha \colon H^{s,p} \to H_\alpha$
    into Hilbert spaces, then such a functional is weakly$^*$ lower semicontinuous.

    For example, take $f,f_n \in H^{1,p}(\Omega, \ell^2)$ with
    $f_n^*f_n \oset{w^*}{\to} \tf{f}$ and
    \begin{equation}
      \boldsymbol{\iota}_{\tilde{f}} \in \set{d \colon H^{1,p}(\Omega) \to L^2(\Omega, |d\tilde{f}|^{p-2}, T^*\Omega)}{||d\tilde{f}||_{L^{p}(\Omega, \ell^2 \otimes T^*\Omega)} \leq 1}.
    \end{equation}
    This yields
    \begin{equation}
      \int \abs{df}^p \leq \liminf_{n \to \infty} \int \abs{df_n}^p.
    \end{equation}

    \item\label{lem:main:it:5} If $f,f_n \in H^{1,p}(\Omega, \ell^2)$ with $f_n^*f_n \oset{w^*}{\to} \tf{f}$ and
    $\abs{df_n}^q \to \abs{df}^q$ in $L^{p/q}(\Omega)$ for some $q \in (0,\infty)$, then $f_n^*f_n \to \tf{f}$ in $\overbar{H}^{1,p}\widehat{\otimes}_\pi H^{1,p}$.
  \end{enumerate}
\end{lemma}
\begin{remark}
    Actually, the condition $\abs{df_n}^q \to \abs{df}^q$ in $L^{p/q}(\Omega)$ from property~\ref{lem:main:it:5} is too strong when $p=2$. A more uniform assumption
    (valid for all $p\geq2$) would be
    \begin{equation}
      \abs{du_n}^{p-2} \oset{w^*}{\to} \abs{du}^{p-2} \text{ in } L^{{p}/({p-2})}(\Omega) \quad\text{and}\quad \norm{du_n}_{L^p} \to \norm{du}_{L^p}
    \end{equation}
    However,
    one immediately sees that when $p > 2$, this is equivalent to $\abs{df_n}^q \to \abs{df}^q$ in $L^{p/q}(\Omega)$ (first for $q = p-2$, and hence for all $0 < q < \infty$).
\end{remark}
\begin{proof}
  By Lemma~\ref{lem:sobolev-lp-isom}, it is enough to prove the statements for $L^p(\mu,\ell^2)$ spaces with finite measure. Let
  $f_n \in L^p(\mu,\ell^2)$.

  \ref{lem:main:it:0}: We already have the second inequality by Lemma~\ref{lem:vect-lp-proj-bound}. For the first one,
  consider a bilinear form $\mathfrak{q} \in \mathfrak{Bil}[\overbar{L}^p(\mu),L^p(\mu)]$ defined by $\brt{\phi \otimes \phi, \mathfrak{q}} =  \int \abs{\phi}^2\abs{f}^{p-2}d\mu$.
  Then
  \begin{equation}
    \int \abs{f}^pd\mu = \brt{f^*f,\mathfrak{q}} \leq \norm{f^*f}_{\pi} \norm{\mathfrak{q}} = \norm{f^*f}_{\pi} \norm{f}^{p-2}_{L^p(\mu,\ell^2)}.
  \end{equation}

  \ref{lem:main:it:1}: Suppose $f_n^*f_n \oset{w^*}{\to} \t \in \overbar{L}^p(\mu)\widehat{\otimes}_\pi L^p(\mu)$. Then
  any such $\t$ is a positive (self-adjoint) tensor since $f_n^*f_n$ are. Define a measure $\nu$ as
  $d\nu = (1+\m{\t}^{(p-2)/2}) d\mu$ so that there is a continuous injection
  \begin{equation}
    \boldsymbol{\iota}\colon L^p(\mu) \to L^2(\nu),
  \end{equation}
  since $\m{\t}^{\frac{p-2}{2}} \in L^{\frac{p}{p-2}}(\mu)$.
  Then $\boldsymbol{\iota}\t\boldsymbol{\iota}^*$ is itself a positive (self-adjoint) operator, and can be decomposed by its eigenvectors, say,
  \begin{equation}
    \boldsymbol{\iota}\t\boldsymbol{\iota}^* = \sum f^i \otimes f^i.
  \end{equation}
  Note that multiplication $\m\cdot \colon \overbar{L}^p(\mu)\widehat{\otimes}_\pi L^p(\mu) \to L^1(\mu)$, see~\eqref{eq:nucler-top-H1}, factors through $\overbar{L}^2(\nu)\widehat{\otimes}_\pi L^2(\nu)$, hence we have
  \begin{equation}
    \m{\t} = \m{\boldsymbol{\iota}\t\boldsymbol{\iota}^*} = \sum_i \abs{f^i}^2 =: \abs{f}^2,
  \end{equation}
  Since $\boldsymbol{\iota}$ is injective and each $f^i$ lies in $\iml  (\boldsymbol{\iota}\t\boldsymbol{\iota}^*)$, we have $f^i \in \iml \t \subset L^p(\mu)$.
  Viewing $\t$ as a nuclear operator $\mathfrak{N}[\overbar{L}^{p}(\mu)^*,L^p(\mu)]$, we obtain from~\eqref{eq:positive-tens-norm} that
  \begin{equation}
    \sum_i \int \abs{f^i}^2 (1+\abs{f}^{p-2}) d\mu = \norm{\boldsymbol{\iota}\t\boldsymbol{\iota}^*}_\pi < \infty.
  \end{equation}
  Therefore, $f$ defines an element of $L^p(\mu,\ell^2)$, and $ \boldsymbol{\iota}(\tf{f})\boldsymbol{\iota}^* = \boldsymbol{\iota}\t\boldsymbol{\iota}^*$.
  By Proposition~\ref{prop:injective-maps}, we conclude that $\tf{f} = \t$.

  \ref{lem:main:it:2}: Let $\t = f_1^*f_1 = f_2^*f_2$. Then $\boldsymbol{\iota}(f_1^*f_1)\boldsymbol{\iota}^* = \boldsymbol{\iota}(f_2^*f_2)\boldsymbol{\iota}^*
  \in \mathfrak{N}[L^2(\nu)]$ and $f_2|_{L^2(\nu)} := f_2\boldsymbol{\iota}^*, f_1|_{L^2(\nu)} := f_1\boldsymbol{\iota}^* \in \mathcal{HS}[L^2(\nu), \overbar{\ell^2}]$. Hence, by polar decomposition,
  $If_1 = f_2$ for some isometry $I$.

  \ref{lem:main:it:3}: Note that
  \begin{equation}
    f_n^*f_n - f^*f = \frac{1}{2}\brs{(f_n - f)^*(f_n + f)+ (f_n + f)^*(f_n - f)},
  \end{equation}
  so
  \begin{equation}
    \norm{f_n^*f_n - f^*f}_{\pi} \leq C \norm{f_n - f}_{L^p(\mu,\ell^2)}\norm{f_n + f}_{L^p(\mu,\ell^2)}
  \end{equation}
  by Lemma~\ref{lem:vect-lp-proj-bound}. To prove the reverse implication, we consider measures
  $d\nu_n = (\abs{f}^{p-2} + \abs{f_n}^{p-2}) d\mu$ and the maps
  $\boldsymbol{\iota}_n\colon L^p(\mu) \to L^2(\nu_n)$, with $\norm{\boldsymbol{\iota}_n}^2 \leq C(\norm{f}_{L^p(\mu,\ell^2)}^{p-2}+ \norm{f_n}_{L^p(\mu,\ell^2)}^{p-2})$.
  Lemma~\ref{lem:HS-isom} then yields the existence of isometries $I_n, I$ such that
  \begin{equation}
    \begin{multlined}
        \int \abs{I_n f_n - I f}^2(\abs{f_n}^{p-2}+\abs{f}^{p-2}) d\mu = \norm{I_n f_n \boldsymbol{\iota}_n^* - I f \boldsymbol{\iota}_n^*}_{L^2(\nu_n, \ell^2)}^2
        \\ \leq \norm{\boldsymbol{\iota}_n(f^*_nf_n - f^*f)\boldsymbol{\iota}_n^*}_\pi
        \leq C \norm{f^*_nf_n - f^*f}_{\pi}(\norm{f^*_nf_n}_{\pi}^{(p-2)/2} + \norm{f^*f}_{\pi}^{(p-2)/2}),
    \end{multlined}
  \end{equation}
  where we used the estimates on $\norm{\boldsymbol{\iota}_n}$ and \ref{lem:main:it:0} for the last inequality.
  Thus, we obtain
  \begin{equation}
    \int \abs{I_n f_n - I f}^pd\mu \leq C\norm{f_n^*f_n - f^*f}_{\pi}(\norm{f^*_nf_n}_{\pi}^{(p-2)/2}+\norm{f^*f}_{\pi}^{(p-2)/2}).
  \end{equation}

  \ref{lem:main:it:4}: Let $\t_n \oset{w^*}{\to} \t$, $\epsilon > 0$, and $\boldsymbol{\iota}_\alpha$ be such that
  $\mathcal{R}(\t) \leq \norm{\boldsymbol{\iota}_\alpha \t \boldsymbol{\iota}_\alpha^*}_{\pi} + \epsilon$.
  Because the projective norm is a dual norm, we have
  \begin{equation}
    \mathcal{R}(\t) \leq \epsilon + \norm{\boldsymbol{\iota}_\alpha \t \boldsymbol{\iota}_\alpha^*}_{\pi}
    \leq \epsilon + \liminf_n \norm{\boldsymbol{\iota}_\alpha \t_n \boldsymbol{\iota}_\alpha^*}_{\pi}
    \leq \epsilon + \liminf_n \mathcal{R}(\t_n).
  \end{equation}

  \ref{lem:main:it:5}: If we have the convergence $\abs{df_n}^{q} \to \abs{df}^{q}$ in $L^{p/q}$ for some $q \in (0,\infty)$,
  we actually have it for any such $q$ (the map $\phi \mapsto \abs{\phi}^\gamma$ is continuous from $L^r$ to $L^{r/\gamma}$).
  Define $\omega_n = (1+\abs{f_n}^{p-2} + \abs{f}^{p-2})$, $\omega_n' = (1+\abs{df_n}^{p-2} + \abs{df}^{p-2})$, together with
  the Hilbert spaces
  \begin{equation}
    H_n = L^2(\Omega, \omega_n) \oplus L^2(\Omega, \omega_n', T^*\Omega)
  \end{equation}
  and injections $\boldsymbol{j}_n \colon H^{1,p}(\Omega) \to H_n$, $\phi \mapsto (\phi, d\phi)$. Then
  $\omega_n$ (by Corollary~\ref{cor:l-inf-convergence}) and $\omega_n'$ converge in $L^{p/(p-2)}$, and we can apply Corollary~\ref{cor:seq-spaces-tensor-convergence} to obtain
  $\norm{\boldsymbol{j}_n(f_n^*f_n-f^*f)\boldsymbol{j}_n^*}_{\pi} \to 0$. By Lemma~\ref{lem:HS-isom}, we can find some isometries
  $I_n$, $I$ such that
  \begin{multline}
    \int \abs{I_nf_n - If}^2\omega_n + \int \abs{I_n(df_n) - I(df)}^2\omega_n'
    \\ =
    \norm{(\boldsymbol{j}_n(f_n^*f_n)\boldsymbol{j}_n^*)^{1/2}-(\boldsymbol{j}_n(f^*f)\boldsymbol{j}_n^*)^{1/2}}_{\mathcal{HS}}^2 \leq \norm{\boldsymbol{j}_n(f_n^*f_n-f^*f)\boldsymbol{j}_n^*}_{\pi} \to 0.
  \end{multline}
  Thus, we obtain $\norm{I_nf_n - If}_{H^{1,p}(\Omega,\ell^2)}\to 0$, which is equivalent to the stated convergence.
\end{proof}

\subsection{Clarke subdifferential}\label{subsec:clarke-subdif}
Let $E$ be a Banach space, $\overline{\R}:=\R\cup\brc{+\infty}$,
and let $f\colon E \to \overline{\R}$ be Lipschitz on a neighborhood of $x \in E$.
The \emph{Clarke directional derivative} of $f$ at $x$ in the direction $v \in E$ is defined by
\begin{equation}
 f^{\circ}(x; v) = \limsup_{\tilde{x} \to x, \tau\downarrow 0}
 \frac{1}{\tau} \brs{f(\tilde{x} + \tau v) - f(\tilde{x})}.
\end{equation}
Since $v \mapsto f^{\circ}(x; v)$ is sublinear and Lipschitz,
it is a support functional of a nonempty, convex, weakly$^*$ compact subset of $E^*$.
The Clarke subdifferential $\partial_C f(x)$ is defined to be
precisely this subset. Namely,
\begin{equation}
  \partial_C f(x) = \set*{\xi \in E^*}{\brt{\xi, v}
  \leq f^{\circ}(x; v) \ \forall v \in E}.
\end{equation}
The following properties of
$\partial_C f(x)$ can be found, for example, in \cite[Section~7.3]{Schirotzek:2007:nonsmooth-analysis}:
\begin{itemize}
  \item $\partial_C f(x) \neq \varnothing$;
  \item $\partial_C (\alpha f)(x) = \alpha\partial_C f(x)$ for $\alpha \in \R$;
  \item $\partial_C (f+g)(x) \subset \partial_C f(x) + \partial_C g(x)$;
  \item $\partial_C f(x)$ is a closed, convex, and bounded set, and hence weakly$^*$ compact;
  \item $0 \in \partial_C f(x)$ if $x$ is a local maximum/minimum.
\end{itemize}
The Clarke subdifferential
generalizes the notion
of the derivative in the case of locally Lipschitz functions
and captures many of the useful properties of the derivative relevant to optimization problems.
For applications of the Clarke subdifferential to critical metrics and eigenvalue optimization, see~\cite{Petrides-Tewodrose:2024:eigenval-clarke-subdiff}.

\begin{proposition}[{\cite[Proposition~4.6.2]{Schirotzek:2007:nonsmooth-analysis}}]
  \label{prop:subdiff-of-norm}
  Let $E$ be a Banach space, and $\omega_z(x) := \norm{x-z}$. Then
  \begin{equation}
    \begin{split}
      \partial_C\omega_z(x) &= \set*{x^* \in E^*}{\norm{x^*} = 1,\ \brt{x - z,x^*} = \norm{x - z}}\quad \text{if } x \neq z,
      \\ \partial_C\omega_z(z) &= B_{E^*},
    \end{split}
  \end{equation}
  where $B_{E^*}$ is the unit ball of $E^*$.
\end{proposition}

For the following property, see, for example, \cite[Proposition~2.2]{Vinokurov:2025:sym-eigen-val-lms}.
\begin{proposition}[Chain rule]\label{prop:clarke-compose}
  Let $G \colon F \to E$ be a continuously Fréchet differentiable map of Banach spaces defined
  on a neighborhood of $x\in F$. If
  $f\colon E \to \overline{\R}$ is Lipschitz
  on a neighborhood of $y = G(x) \in E$, one has
  \begin{equation}
    \partial_C (f \circ G)(x) \subset d_x G^*[\partial_C f (y)].
  \end{equation}
  Thus, the subdifferential of a pullback is contained in the pullback of the subdifferential.
\end{proposition}
The following proposition is a simplified version of~\cite[Theorem~12.4.1]{Schirotzek:2007:nonsmooth-analysis}
sufficient for our purposes.
\begin{proposition}[Clarke’s Multiplier Rule]\label{prop:multip-rule}
  Let $E$ be a Banach space and $f \colon E \to \overline{\R}$ be a function
  that is Lipschitz on a neighborhood of $x \in S$, where $S \subset E$ is a closed
  convex set. If $x$ is a local minimum of $f$ on $S$, then
  \begin{equation}
    \exists \xi \in \partial_C f(x) \colon \
    \brt{\tilde{x}-x, \xi} \geq 0 \quad\forall \tilde{x} \in S.
  \end{equation}
\end{proposition}

Note that for unbounded self-adjoint operators, eigenvalues are indexed in
increasing order (from bottom to top). In contrast, for
bounded operators, we will index the eigenvalues in decreasing order
(from top to bottom), starting with $k=1$.  The next proposition follows from the proofs of
\cite[Proposition~2.18]{Vinokurov:2025:higher-dim-harm-eigenval} and \cite[Lemma~2.5]{Vinokurov:2025:sym-eigen-val-lms}. The proofs rely only on the fact
that $\lambda_k(T)$ is an isolated eigenvalue of finite multiplicity.
\begin{proposition}
  \label{prop:subdiff-calc}
  Let $H$ be a Hilbert space and
  $\mathfrak{L}_{sa}[H]$ be the space of self-adjoint operators on $H$.
  If $T \in \mathfrak{L}_{sa}[H]$ is a positive operator with an isolated eigenvalue $\lambda_k(T) > 0$ of finite multiplicity, then the functional
  $\lambda_k\colon \mathfrak{L}_{sa}[H] \to \R$ is Lipschitz on a neighborhood
  of $T$. Its Clarke subdifferential $\partial_C \lambda_k(T) \subset \mathfrak{N}[H]$
  is given by
  \begin{equation}\label{abst-clark-subdif}
    \partial_C \lambda_k(T) = \operatorname{co}\,
    \set*{\phi \otimes \phi }{
      \phi \in V_{\lambda = \lambda_k},\ \norm{\phi} = 1
    },
  \end{equation}
  where $\operatorname{co} K$ denotes the convex hull of the set $K$.
\end{proposition}

\subsection{Weighted eigenvalues and weighted Sobolev spaces}\label{sec:weighted-stuff}

Throughout this section, let $\Omega$ be a compact Riemannian manifold possibly with boundary, and let $m = \dim \Omega \geq 2$. For a pair $(\alpha, \mu) \in L^1_+(\Omega) \times \mathcal{M}_+^c(\overbar{\Omega})$,
let us define
\begin{equation}
  H^1(\Omega,\mu,\alpha) := \overline{\set*{(\phi,d\phi)}{\phi \in C^\infty(\overbar{\Omega})}}^{L^2(\Omega,\mu)\oplus L^2(\Omega,\alpha,T^*\Omega)}.
\end{equation}
As usual, we identify functions $\beta \in L^1_+(\Omega)$ with absolutely continuous measures $\beta dv_g$, so
$H^1(\Omega,\beta,\alpha) := H^1(\Omega,\beta dv_g,\alpha)$.

For general $\alpha$, the functions from $H^1(\alpha):=H^1(\Omega,\alpha,\alpha)$ can exhibit a very strange behavior.
Such functions need not be locally integrable; moreover, the differential component need not coincide with distributional
(weak) differential. There may exist pairs $(\phi,0) \in \wH{\alpha}$ with nonconstant $\phi$, as well as pairs $(0,\omega) \in \wH{\alpha}$; hence
a single function may admit several distinct differentials. Also, it can happen that $H^1(\alpha) \neq W^{1,2}(\alpha)$, where
\begin{equation}
    W^{1,2}(\alpha):= \overline{\set*{(\phi,d\phi)}{\phi \in W^{1,1}_{loc}\cap L^2(\alpha),\ d\phi \in L^2(\alpha)}}^{L^2(\alpha)\oplus L^2(\alpha)}.
\end{equation}
A priori, one only has $H^1(\alpha) \subset W^{1,2}(\alpha)$.
See \cite{Zhikov:1998:weighted-sobolev} for examples of different anomalies mentioned above.

We work with general $\alpha$ mainly because it remains open in dimensions $m \geq 3$ whether the differential of a nonconstant
$p$-harmonic map can vanish on a set of positive (Lebesgue) measure. In dimension two it is not possible thanks to \cite{Manfredi:1988:plane-p-harm-func}.
\begin{remark}
  The situation becomes much better when we already know that $1/\tilde{\alpha} \in L^1_{loc}(\Omega\setminus Z)$, where
  $\tilde{\alpha} \in L^1(\Omega)$ and $Z$ is a closed set of zero measure. This is due to the following simple estimate
  $(\int \brt{d\phi_n, X})^2 \leq (\int \abs{d\phi_n}^2\tilde{\alpha})(\int \abs{X}^2/\tilde{\alpha})$, with
  a vector field $X \in C^\infty_0(\Omega\setminus Z, T\Omega)$, which implies that the differential of a function from $\wH[\Omega]{\tilde{\alpha}}$ is unique
  and coincides with the distributional one at least on $\Omega\setminus Z$.
  In such a case,
  \begin{equation}
    \wH[\Omega]{\tilde{\alpha}} \subset \set*{\phi \in W^{1,1}_{loc}(\Omega\setminus Z)}{(\phi, d\phi) \in L^2(\Omega,\tilde{\alpha})\oplus L^2(\Omega,\tilde{\alpha},T^*\Omega)}.
  \end{equation}
\end{remark}

\begin{definition}\label{def:stable-points}
  For a sequence $(\alpha_\epsilon, \mu_\epsilon) \in L^1_+(\Omega) \times \mathcal{M}_+^c(\overbar{\Omega})$, we call a point $x \in \overbar{\Omega}$ \emph{stable}
  if there exists a neighborhood $U \ni x$ such that up to a subsequence,
  \begin{equation}\label{ineq:stable-point}
    \int \alpha_\epsilon\abs{d\phi}^2 - \int \phi^2 d\mu_\epsilon \geq 0 \quad \forall \phi \in C^\infty_0(U),
  \end{equation}
  that is, the associated quadratic forms are nonnegative definite on $C^\infty_0(U)$.
\end{definition}
If $\lambda_k(\alpha_\epsilon,\mu_\epsilon) = 1$ for all $\epsilon$,
there exist at most $k$ unstable points (see~\cite[Remark~2.3]{Vinokurov:2025:higher-dim-harm-eigenval}). On the other hand,
even if a point $y \in \overbar{\Omega}$ is unstable (see~\cite[Lemma~2.5]{Vinokurov:2025:higher-dim-harm-eigenval}, cf. also \cite[Claim~2.5]{Petrides:2024:conf-class-opt-lms}),
one can find a neighborhood $V \ni y$ and a sequence
of radii $r_\epsilon \to 0$ such that
\begin{equation}\label{ineq:unstable-point}
    \int \alpha_\epsilon\abs{d\phi}^2 - \int \phi^2 d\mu_\epsilon \geq 0 \quad \forall \phi \in C^\infty_0(V\setminus \overbar{B_{r_\epsilon}(y)}).
\end{equation}

An interesting corollary of \eqref{ineq:unstable-point} is that for a constant sequence $(\alpha_\epsilon, \mu_\epsilon) \equiv (\alpha, \lambda_k \mu)$
with regular enough $\alpha$, every point $x \in \overbar{\Omega}$ is stable, and we have
\begin{proposition}\label{prop:bounded-form-p}
  Let $(\alpha,\mu) \in L_+^{m/(m-2)}(\Omega)\times \mathcal{M}_+^c(\overbar{\Omega})$ with $\lambda_k(\alpha, \mu) > 0$.
  Then $\mu \in \mathfrak{Bil}[\wH[\Omega]{\alpha}]$. In particular, the identity map on $C^\infty(\overbar{\Omega})$ extends
  to a continuous map $\wH[\Omega]{\alpha} \to H^1(\Omega,\mu,\alpha)$.
\end{proposition}
\begin{proof}
  Set $\lambda_k := \lambda_k(\alpha, \mu)$, and apply~\eqref{ineq:unstable-point} to constant sequence $(\alpha, \lambda_k \mu)$.
  Then the inequality in \eqref{ineq:unstable-point} holds for every $\phi \in C^\infty_0(V\setminus \brc{y})$, and one extends it over $\brc{y}$ by using logarithmic
  cut-offs (equivalently, the fact that points have zero $m$-capacity).
  Therefore, we can take a finite cover of $\overbar{\Omega}$ with stable neighborhoods and a subordinate partition of unity $\sum_i \psi^2_i = 1$.
  For a function $\phi \in C^\infty(\overbar{\Omega})$, we then have
  \begin{equation}\label{ineq:mu-is-bounded}
    \begin{split}
      \lambda_k \int \phi^2 d\mu &\leq \sum_i\int \psi_i^2\abs{d\phi}^2 \alpha + \frac{1}{2}\sum_i \int \brt{d\phi^2, d\psi^2_i}\alpha  + \sum_i \int \abs{d\psi_i}^2 \phi^2  \alpha
      \\                       &\leq \int \abs{d\phi}^2 \alpha + C \int \phi^2  \alpha,
    \end{split}
  \end{equation}
  where the middle term vanishes because $\sum_i \psi^2_i = 1$.
\end{proof}
Thus, when $\mu \in \mathfrak{Bil}[\wH[\Omega]{\alpha}]$, we can
integrate functions from $\wH[\Omega]{\alpha}$ with respect to $\mu$ in the sense of continuous map
$\wH[\Omega]{\alpha} \to L^2(\Omega,\mu)$, and it will always be understood in this sense.

Generalizing the classical identity $H^1 = H^1(\mu,1)$, valid when $\mu \in \mathfrak{Bil}[H^1]$ (cf. also \cite{Petrides:2024:conf-class-opt-lms}),
we obtain the following
\begin{proposition}\label{prop:weighted-sobolev-equiv}
  If $\alpha \geq c > 0$ and $0 \neq \mu \in \mathfrak{Bil}[H^1(1,\alpha)]$ for a pair $(\alpha, \mu) \in L^1_+(\Omega) \times \mathcal{M}_+^c(\overbar{\Omega})$, then
  \begin{equation}
    H^1(\mu,\alpha) = H^1(1,\alpha)
  \end{equation}
  with equivalent norms.
\end{proposition}
\begin{proof}
  Since $H^1(1,\alpha) \hookrightarrow H^1 \hookrightarrow L^2$ is compact,
  we can apply a generalized Poincaré inequality (see \cite[Lemma~4.1.3]{Ziemer:1989:weakly-diff-func}) to
  the projection $P\colon H^1(1,\alpha) \to \R \subset H^1(1,\alpha)$ onto constant functions given by $P\colon\phi \to \dashint \phi d\mu$. Hence
  \begin{equation}
    \norm{\phi}_{L^2} \leq \norm{P\phi}_{L^2} + \norm{\phi - P\phi}_{L^2}
    \leq C\brr{\norm{\phi}_{L^2(\mu)} + \norm{d\phi}_{L^2(\alpha)}},
  \end{equation}
  so the norms $\norm{\phi}_{L^2} + \norm{d\phi}_{L^2(\alpha)}$ and $\norm{\phi}_{L^2(\mu)} + \norm{d\phi}_{L^2(\alpha)}$
  are equivalent.
\end{proof}

Define $\lambda_{ess}(\alpha,\mu) := \sup_k \lambda_{k}(\alpha,\mu)$. We will see in the proof of the next result that the case $\lambda_{ess}(\alpha,\mu) = \infty$ corresponds
to compact embeddings $H^1(\mu,\alpha) \to L^2(\mu)$. Set also $\mu L^\infty = \set{\tilde{\mu} \in \mathcal{M}}{d\tilde{\mu}/d\mu \in L^\infty(\mu)}$ with
the corresponding $L^\infty$-norm.
\begin{lemma}\label{lem:lambda-alph-mu-subdiff}
  Let $(\alpha, \mu) \in L^q_+(\Omega) \times \mathcal{M}_+^c(\overbar{\Omega})$ with $q={p}/(p-2)$ and $p \in (2,\infty)$.
  If $\lambda_{k}(\alpha,\mu) < \lambda_{ess}(\alpha,\mu)$,
  then the map $\nEigen{k}{p}\colon \alpha L^\infty \times \mu L^\infty \to \R$ is Lipschitz
  on a neighborhood of $(\alpha,\mu)$,  and
  \begin{multline}
    \partial_C (-\nEigen{k}{p})(\alpha,\mu) \\\subset \operatorname{co}_{(\phi,\omega)} \tfrac{1}{\norm{\alpha}_{L^{q}}}
    \brc{\brr{\nEigen{k}{p}\cdot\brr{\tfrac{\alpha}{\norm{\alpha}_{L^{q}}}}^{q-1} - \abs{\omega}^2,\
    \lambda_k \cdot \brs{\phi^2 - 1}}},
  \end{multline}
  where $\lambda_k := \lambda_k(\alpha,\mu)$ and pairs $(\phi,\omega)\in H^{1}(\mu,\alpha)$ run over $k$th eigenfunctions (that is, $\delta(\alpha \omega) = \lambda_k \phi \mu$) normalized by $\int \phi^2 d\mu = \mu(\overbar{\Omega})$.

  If $\mu \in \mathfrak{Bil}[H^1(1,\alpha)]$ and $\alpha\geq c > 0$, then one can take $\tilde{\mu}$ from a neighborhood of $\mu \in \mathcal{M}_+^c \cap \mathfrak{Bil}[H^1(1,\alpha)]$, with the same
  Clarke subdifferential.
\end{lemma}
\begin{proof}
  Let $T_{\tilde{\alpha}}$ and $T_{\tilde{\mu}}$ be the self-adjoint operators on $H^1(\mu,\alpha)$ corresponding to the quadratic forms
  $\phi \mapsto \int \abs{d\phi}^2\tilde{\alpha}$ and $\phi \mapsto \int \phi^2d\tilde{\mu}$. Then
  \begin{multline}
    \frac{1}{\lambda_k(\tilde{\alpha},\tilde{\mu}) + 1}
    = \sup_{\substack{V \subset H^1(\mu,\alpha) \\ \dim V = k+1}} \inf_{v \in V\setminus\brc{0}} \frac{\brt{T_{\tilde{\mu}}v,v}}{\brt{(T_{\tilde{\mu}} + T_{\tilde{\alpha}})v,v}}
    \\ =  \lambda_{k+1}\brr{(T_{\tilde{\mu}} + T_{\tilde{\alpha}})^{-1/2}T_{\tilde{\mu}}(T_{\tilde{\mu}} + T_{\tilde{\alpha}})^{-1/2}}
  \end{multline}
  The condition $\lambda_{k}(\alpha,\mu) < \lambda_{ess}(\alpha,\mu)$ means that $1/(\lambda_k + 1) > 0$ is an isolated eigenvalue of finite multiplicity for the operator
  $T_{\mu}$ (since $T_{\mu} + T_{\alpha} = \id$).
  By Proposition~\ref{prop:subdiff-calc}, we can find a neighborhood of $(\alpha, \mu) \in \alpha L^\infty \times \mu L^\infty$ in which it remains isolated, and is
  a Lipschitz function of $(\tilde{\alpha}, \tilde{\mu})$.
  To compute the Clarke subdifferential, one applies the chain rule (Proposition~\ref{prop:clarke-compose}) and Proposition~\ref{prop:subdiff-calc}, obtaining
  \begin{equation}
      \frac{1}{(\lambda_k + 1)^2}
      \partial_C (-\lambda_k)\brr{\alpha, \mu} \subset
      \operatorname{co}_{(\phi,\omega)}\brc{\brr{-\frac{\abs{\omega}^2}{\lambda_k + 1}, \frac{\lambda_k\phi^2}{\lambda_k + 1}  }}
      \subset L^1(\alpha)\times L^1(\mu),
  \end{equation}
  where $(\phi,\omega) \in V_{\lambda = \lambda_k}$ and $\int \phi^2 d\mu =  \frac{1}{\lambda_k + 1}$.
  It remains to multiply both sides by $(\lambda_k + 1)^2$ and apply the chain rule to
  $-\nEigen{k}{p}(\tilde{\alpha},\tilde{\mu}) = -\lambda_k(\frac{\tilde{\alpha}}{\norm{\tilde{\alpha}}_{L^q}}, \frac{\tilde{\mu}}{\tilde{\mu}(\overbar{\Omega})})$.

  The last statement follows by the same proof. Recall that $H^1(1,\alpha) = H^1(\mu,\alpha)$ by Proposition~\ref{prop:weighted-sobolev-equiv}.
\end{proof}
\begin{remark}\label{rem:gkl-sobolev-duals}
  It follows, for example, from~\cite[Lemmas~4.9,~4.10]{Girouard-Karpukhin-Lagace:2021:continuity-of-eigenval} and the density of $C^\infty(\overbar{\Omega})$ in $H^{1,\frac{m}{m-1}}(\Omega)^*$
that when $\mu \in H^{1,r}(\Omega)^*\cap \mathcal{M}_{+}(\Omega)$, where $r = \frac{m}{m-1}$ if $m \geq 3$ and $1\leq r < 2$ if $m=2$, then
the quadratic form $\phi \mapsto \int \phi^2d\mu$ is compact on $H^1(\Omega)$, and hence on
$H^1(1,\alpha) \hookrightarrow H^1$ for $\alpha \geq c > 0$.
\end{remark}
Note that if $p \in [2,\infty]$ and $\alpha \in L^{p/(p-2)}_+(\Omega)$, then clearly, the identity map extends to $H^{1,p}(\Omega) \to \wH[\Omega]{\alpha}$.
\begin{proposition}[Variational characterization]\label{prop:var-character}
  Let $(\alpha,\mu) \in L^{p/(p-2)}_+(\Omega)\times \mathcal{M}_+^c(\overbar{\Omega})$ with $p \in [2,m]$ and $\lambda_k(\alpha,\mu) > 0$. Then there exists
  a subspace $V \subset H^1(\mu,\alpha)$ of $\dim V \leq k$, including constant functions such that
  \begin{equation}
    \lambda_k(\alpha,\mu) \int \phi^2d\mu \leq \int \abs{d\phi}^2\alpha
    \quad \forall \phi \in H^{1,p}(\Omega)\colon\  \phi \bot_\mu V.
  \end{equation}
  Equality holds iff $\phi$ is an eigenfunction corresponding to $\lambda_k(\alpha,\mu)$. One can also view
  $V$ as a subspace $V \subset  H^{1,p}(\Omega)^*\cap \mathcal{M}^c(\overbar{\Omega})$.
\end{proposition}
\begin{proof}
  By the inclusion $L^{p/(p-2)} \subset L^{m/(m-2)}$ and Proposition~\ref{prop:bounded-form-p}, the identity map extends even further:
  \begin{equation}
    H^{1,p} \to \wH{\alpha} \to H^1(\mu,\alpha).
  \end{equation}
  From the proof of Lemma~\ref{lem:lambda-alph-mu-subdiff}, we see that $1/(\lambda_k + 1)$ is the $(k+1)$th eigenvalue of the operator
  $T_\mu \in \mathfrak{L}[H^1(\mu,\alpha)]$ associated with the quadratic form $(\phi,\omega) \mapsto \int \phi^2d\mu$.
  Moreover, $\set{1/(\lambda_i + 1)}{\lambda_i < \lambda_k}$ is the top part of its spectrum. So one can take $V = \bigoplus_{\lambda_i < \lambda_k}V_{\lambda_i}$,
  and invoke the variational characterization for the eigenvalues of $T_\mu$.
  Finally, for the eigenfunctions, $\mu$-orthogonality is equivalent to the orthogonality
  with respect to the inner product on $H^1(\mu,\alpha)$. Thus, the variational characterization follows.

  Note that the density $H^{1,p} \to L^2(\mu)$ implies that there is an injection $L^2(\mu) \oset{\cdot\mu}{\to} (H^{1,p})^*$.
\end{proof}

\section{Proofs of the examples from Section~\ref{sec:examples}}
\subsection{Proof of Theorem~\ref{thm:1st-max-on-minimal}}
Since $g = u^* g_{\Sph^n}$, we have $\abs{du}^2_{g} = m$, and hence
$(\int \abs{du}^{p}_{g} dv_g)^{2/p} = m \Vol_g (M)^{2/p}$.
Without loss of generality, we can assume that $\lambda_1(\alpha,\mu) > 0$
and $\mu(M) = 1 = \norm{\alpha}_{L^{p/(p-2)}}$.
By Proposition~\ref{prop:var-character}, we then have
\begin{equation}
  \lambda_1(\alpha,\mu) \int \phi^2d\mu \leq \int \abs{d\phi}^2\alpha dv_g \quad \forall {\phi \in H^{1,p}(M)\colon\ \int \phi d\mu = 0}.
\end{equation}
Therefore, Hersch's trick, Hölder's inequality and \cite[Theorem~1.1]{ElSoufi-Ilias:1986:hersch} yield
\begin{equation}
  \begin{split}
    \lambda_1(\alpha,\mu) &\leq \int \abs{du_F}_g^2\alpha dv_g \leq \brr{\int \abs{du_F}_{g}^p dv_g}^{2/p}
    \\                    &\leq \brr{\int \abs{du_F}^m_{g} dv_g}^{2/m} \Vol_g(M)^{2(1/p-1/m)}
    \\                    &= m\Vol_{u_F^* g_{\Sph^n}}(M)^{2/m}\Vol_g(M)^{2(1/p-1/m)}
    \\                    &\leq m \Vol_g (M)^{2/p},
  \end{split}
\end{equation}
where $u_F := F \circ u$ and $F$ is a conformal automorphism of $\Sph^n$ (see also the proof of \cite[Proposition~3.1]{ElSoufi-Ilias:1986:hersch}).

If all the inequalities are equalities, then the coordinate functions of
$u_F$ are $\lambda_1(\alpha,\mu)$-eigenfunctions. Equality in Hölder's
inequality gives $\alpha = c\abs{du_F}_g^{p-2}$ for some $c >0$.
Then $\delta(\alpha d u_F) = \lambda_1 u_F \mu$ and $\abs{u_F}^2 = 1$ imply (see Lemma~\ref{lem:regularity-in-good-pt})
$\lambda_1 \mu = \alpha \abs{d u_F}_g^2 = c \abs{d u_F}_g^{p}$.
If $p<m$, equality in the $L^p$--$L^m$ inequality forces
$\abs{du_F}$ to be constant. If $p=m$ and $M\neq\Sph^m$, then the equality
case in the conformal-volume estimate forces $F$ to be an isometry, and hence
$\abs{du_F}$ is again constant.
In the exceptional case $p=m$ and $M=\Sph^m$, $F$ may be a non-isometric
conformal automorphism, which is precisely the constant maximizing pair transformed by a conformal
diffeomorphism of $\Sph^m$.

\subsection{Proof of Theorem~\ref{thm-max-on-sphere}}
The case $p = 2$ was treated in \cite{Vinokurov:2025:higher-dim-harm-eigenval}, and the case of a general $p$ is essentially the same.
The upper bound $\NEigen{k+2}{p}(\Sph^m, g_{\Sph^m}) \leq (\int \abs{d\SHM{k}}^p_{g_{\Sph^m}} dv_{g_{\Sph^m}})^{2/p}$, the equality case,
and uniqueness are established
by repeating the same argument as in \cite[Theorem~1.6]{Vinokurov:2025:higher-dim-harm-eigenval}. Then it remains to calculate the index
of $\SHM{k}$, that is, the number of negative eigenvalues of the quadratic form
$\mathfrak{q}[\phi] = \int \abs{d\phi}^2\abs{d\SHM{k}}^{p-2} - \int \phi^2 \abs{d\SHM{k}}^{p}$.
To show that $\SHM{k}$ is a maximizer of $\nEigen{k+2}{p}$, we need to prove that  $\ind \mathfrak{q}\leq k + 2$.

Let us set $n = m - k - 1$, $\tilde{n} = n + 2 - p$, and $\tilde{m} = \tilde{n} + k + 1$. By repeating the calculation in the proof of \cite[Theorem~1.8]{Vinokurov:2025:higher-dim-harm-eigenval},
we obtain that $\ind  \mathfrak{q} = \sum_{\ell \in \mathbb{N}} m_{\ell} \ind \mathfrak{q}_\ell$, where $m_\ell$ is the multiplicity
of the space of eigenfunctions on $\Sph^k$ with the eigenvalue $\ell(k - 1 + \ell)$,
\begin{equation}
  \mathfrak{q}_{\ell}[\psi] =
  \int_{-1}^{1} \dot{\psi}^2 \brt{t}_{\frac{\tilde{n}+1}{2}-\alpha}^{\frac{k+1}{2}+\ell} dt
  +L \int_{-1}^{1} \psi^2 \brt{t}_{\frac{\tilde{n}-1}{2}-\alpha}^{\frac{k-1}{2}+\ell} dt,
\end{equation}
$\brt{t}^x_y = (1+t)^x(1-t)^y$, $4L = \ell(\tilde{m}-1+\ell) - n-\alpha\brr{k+1+2\ell}$, and
$\alpha$ is the least root of
\begin{equation}\label{eq:jacobi-root}
  \alpha^2 - (\tilde{n}-1)\alpha + n = 0.
\end{equation}
Equation~\eqref{eq:jacobi-root} has real solutions if and only if $m - k - 1 = n \geq (1+\sqrt{p})^2$. When this condition is not satisfied, we actually have
$\ind \mathfrak{q} = \infty$; see \cite[Remark~1.7]{Vinokurov:2025:higher-dim-harm-eigenval}.
Further computations show (see \cite[Section~3.2]{Vinokurov:2025:higher-dim-harm-eigenval}) that
\begin{equation}
  \ind \mathfrak{q}_{\ell} = \#\set*{s \in \mathbb{N}}{ s < \tfrac{\alpha - \ell}{2}},
\end{equation}
so $\ind \mathfrak{q} = \ind \mathfrak{q}_0 + (k+1) \ind \mathfrak{q}_1 \leq 1 + (k + 1)$ provided $\alpha \leq 2$.
Thus, we have
\begin{equation}
  \frac{1}{2}\brr{\tilde{n} - 1 - \sqrt{(\tilde{n}-1)^2 - 4n}} = \frac{2n}{\tilde{n} - 1 + \sqrt{(\tilde{n}-1)^2 - 4n}} \leq 2,
\end{equation}
which holds exactly when $m - k - 1 = n \geq 2(p+1)$. Moreover, one always has $2(p+1) \geq (1+\sqrt{p})^2$.

\section{Proof of the existence and regularity}\label{sec:exis-and-reg}
\begin{proposition}\label{prop:upper-cont-and-bound}
  Let $(M, g)$ be a closed Riemannian manifold of dimension $m \geq 3$ and $p \in (2,m]$.
  The functional $\mathcal{M}_+^c(M)\times \mathcal{M}_+^c(M)\ni (\alpha, \mu) \mapsto \lambda_k(\alpha,\mu)$ is weakly$^*$ upper semicontinuous
  (that is, $\alpha$ may even be a measure),
  and
  \begin{equation}
    \begin{split}
      \NEigen{k}{p}(g) :&= \sup \set*{\nEigen{k}{p}(\alpha,\mu)}{\alpha, \mu \in C^\infty(M),\ \alpha,\mu > 0}.
      \\                 &= \sup \set*{\nEigen{k}{p}(\alpha,\mu)}{\alpha \in L^{\frac{p}{p-2}}_+(M),\ \mu \in L^{1}_+(M)}.
    \end{split}
  \end{equation}
  Moreover, for $(\alpha,\mu) \in L^{1}_+ \times L^{1}_+$, we have $\lambda_k(\alpha_n,\mu_n) \to \lambda_k(\alpha,\mu)$,
  where $\alpha_n = \max \brc{\alpha, \frac{1}{n}}$ and $\mu_n = \min \brc{\mu, n}$.
\end{proposition}
\begin{proof}
  The upper semicontinuity directly follows from the variational characterization; see \cite[Proposition~1.1]{Kokarev:2014:measure-eigenval}.
  Let us now take a pair $(\alpha,\mu) \in L^{1}_+(M) \times L^{1}_+(M)$ and approximate it by cut-off functions
  $\alpha_{\tilde{n}} = \max \brc{\alpha, \frac{1}{\tilde{n}}}$ and $\mu_{\tilde{n}} = \min \brc{\mu, \tilde{n}}$
  so that the variational characterization yields
  \begin{equation}
    \lambda_k(\alpha,\mu) \leq \lambda_k(\alpha_{\tilde{n}},\mu_{\tilde{n}}).
  \end{equation}
  Together with the upper semicontinuity, the desired convergence $\lambda_k(\alpha_n,\mu_n) \to \lambda_k(\alpha,\mu)$ follows.

  When $\mu \in L^\infty_{+}$ and $\alpha \in L^1_{\geq c}$, Proposition~\ref{prop:weighted-sobolev-equiv}
  yields $H^{1}(\mu,\alpha) = H^1(1,\alpha)$ and $\lambda_k(\alpha, \mu_s) \to \lambda_k(\alpha, \mu)$ for smooth
  $0< \mu_s \to \mu$ in $L^{m/2}$  (see Remark~\ref{rem:gkl-sobolev-duals}).
  Thus, to establish the equality of the suprema, it suffices to assume that $\alpha \in L^{q}_{\geq c}$ and $\mu \in C^{\infty}_{\geq c}$
  with $\norm{\alpha}_{L^q} = 1 =\norm{\mu}_{L^1}$,
  and then prove that $\lambda_k(\alpha,\mu) \leq \liminf_{\epsilon\to 0}\lambda_k(\alpha_\epsilon,\mu)$ as $\epsilon \to 0$ for some
  $\alpha_\epsilon \in  C^\infty_{\geq c}$ with $\norm{\alpha_\epsilon}_{L^q} \to \norm{\alpha}_{L^q}$, where $q = p/(p-2)$.

  Let $J_\epsilon \colon C^\infty(M)^* \to C^\infty(M)$ be mollifiers such that
  \begin{itemize}
    \item $J_\epsilon(1) = 1$ and $\norm{J_\epsilon^*\alpha}_{L^q} \to \norm{\alpha}_{L^q}$,
    \item $\norm{J_\epsilon \phi - \phi}_{L^2} \leq o_\epsilon(1)\norm{\phi}_{H^1}$ \hfill and \hfill $\abs{d(J_\epsilon \phi)}^2\leq (1+ o_\epsilon(1))J_\epsilon(\abs{d\phi}^2)$,
  \end{itemize}
  where $o_\epsilon(1)$ denotes arbitrary quantities tending to zero as $\epsilon \to 0$. For example, the heat kernel on $M$ provides such mollifiers.

  Set $\alpha_\epsilon := J_\epsilon^*\alpha \in C^\infty_{\geq c}$, and let $E_\epsilon \subset C^\infty$ be the subspace generated by the first $k+1$
  eigenfunctions of $\delta(\alpha_\epsilon d\phi) = \lambda \phi \mu$. Observe that for $\phi \in E_\epsilon$, we have
  \begin{equation}
    \norm{J_\epsilon \phi - \phi}_{L^2(\mu)}^2 \leq o_\epsilon(1)\norm{\phi}_{H^1}^2
    \leq \frac{o_\epsilon(1)}{c}\norm{\phi}_{H^1(1,\alpha)}^2
    \leq o_\epsilon(1)(1+\NEigen{k}{p})\norm{\phi}_{L^2(\mu)}^2,
  \end{equation}
  and $\NEigen{k}{p} := \NEigen{k}{p}(g) < \infty$ if $p \in (2,m]$.
  Therefore, $\dim J_\epsilon E_\epsilon = \dim E_\epsilon$ for small $\epsilon$, and
  \begin{multline}
    \lambda_k(\alpha,\mu) \leq \sup_{\phi \in E_\epsilon\setminus\brc{0}} \frac{\int \alpha\abs{d(J_\epsilon\phi)}^2}{\norm{J_\epsilon\phi}^2_{L^2(\mu)}}
    \leq \sup_{\phi \in E_\epsilon\setminus\brc{0}}\frac{1+o_\epsilon(1)}{1-o_\epsilon(1)} \frac{\int (J_\epsilon^*\alpha)\abs{d\phi}^2}{\norm{\phi}^2_{L^2(\mu)}}
    \\= (1+o_\epsilon(1)) \lambda_k(\alpha_\epsilon,\mu).
  \end{multline}
  So, $\lambda_k(\alpha,\mu) \leq \liminf_{\epsilon\to 0}\lambda_k(\alpha_\epsilon,\mu)$. In fact, the upper semicontinuity of $\lambda_k$ then implies
  $\lambda_k(\alpha_\epsilon,\mu) \to \lambda_k(\alpha,\mu)$.
\end{proof}

\subsection{Constructing a maximizing sequence}
The following result is due to~\cite[Theorem~1]{Ekeland:1979:var-principle} (with the rescaled metric $\gamma d$).
\begin{proposition}[Ekeland's variational principle]
  Let $(X,d)$ be a complete metric space and $f \colon X \to \R\cup\brc{+\infty}$ be a lower semicontinuous function bounded from below.
  Let $\epsilon > 0$, and let $\tilde{x} \in X$ be such that $f(\tilde{x}) < \inf_X f + \epsilon$. Then, for any $\gamma > 0$, there exists $z \in X$
  such that $f(z) \leq f(\tilde{x})$, $d(z,\tilde{x}) < \gamma$, and $z$ is a strong minimizer of $x\mapsto f(x) + \frac{\epsilon}{\gamma}d(z,x)$, that is,
  \begin{equation}
    f(z) < f(x) + \frac{\epsilon}{\gamma}d(z,x),\quad \forall x \in X\setminus\brc{z}.
  \end{equation}
\end{proposition}
  For $\epsilon > 0$, let us define
  \begin{equation}
    X_\epsilon := \set*{(\alpha,\mu) \in L^{q}_+\times L^1_+}{\alpha \geq \epsilon,\ \mu \leq \epsilon^{-1}}
  \end{equation}
  equipped with the metric induced by the sum of the $L^q$- and $L^1$-norms.
\begin{proposition}\label{prop:max-sequence}
  Let $(M,g)$ be a closed connected Riemannian manifold of $\dim M = m \geq 3$,
  $p \in (2, m]$, $q = \frac{p}{p-2}$, and $(\tilde{\alpha}_\epsilon, \tilde{\mu}_\epsilon) \in X_\epsilon$ be a sequence ($\epsilon \searrow 0$)
  with $\norm{\tilde{\alpha}_\epsilon}_{L^q} = 1 = \tilde{\mu}_\epsilon(M)$ and
  $\lambda_k(\tilde{\alpha}_\epsilon, \tilde{\mu}_\epsilon) \to \NEigen{k}{p}(g)$.
  Then there exist a sequence $(\alpha_\epsilon, \mu_\epsilon) \in X_\epsilon$
  and a sequence of maps $\F_\epsilon \in H^{1,p}(M, \R^{n_\epsilon})$
  such that $\norm{\alpha_\epsilon - \tilde{\alpha}_\epsilon}_{L^q} \to 0$,
  $\norm{\mu_\epsilon - \tilde{\mu}_\epsilon}_{L^1} \to 0$,
  $\lambda_\epsilon := \lambda_k(\alpha_\epsilon, \mu_\epsilon) \to \NEigen{k}{p}(g)$, and
  \begin{enumerate}[label={(\roman*)}]
    \item $\delta (\alpha_\epsilon d\F_\epsilon) = \lambda_\epsilon \F_\epsilon \mu_\epsilon$ weakly;
    \item $\abs{\F_\epsilon}\leq 1$, and the norms $\norm{\F_\epsilon}_{H^{1,p}}$ are uniformly bounded;
    \item $\norm{\lambda_\epsilon\alpha_\epsilon^{\frac{2}{p-2}} - \abs{d\F_\epsilon}^2}_{L^{p/2}} \to 0$;
    \item\label{prop:max-sequence:it:almost-sphere} $\norm{1 - |\F_\epsilon|^2}_{L^1} \to 0$ and for any $\varrho > 0$,
     $v_g\brr{\brc{|\F_\epsilon|^2 < 1-\varrho}} \leq C/\norm{\mu_\epsilon}_{L^\infty}$.
  \end{enumerate}
\end{proposition}
\begin{proof}
  Let $f \colon X_\epsilon \to \R\cup\brc{+\infty}$ be the function that equals $-\nEigen{k}{p}$ on $\set{(\alpha, \mu) \in X_\epsilon}{\mu(M) \geq 1/2}$, and $+\infty$ otherwise.
  Set $\gamma_\epsilon^2 := \epsilon + \NEigen{k}{p}(g) - \lambda_k(\tilde{\alpha}_\epsilon, \tilde{\mu}_\epsilon) \to 0$.
  By Ekeland's variational principle, we can find a pair $(\alpha_\epsilon, \mu_\epsilon) \in X_\epsilon$ such that
  $\norm{\tilde{\alpha}_\epsilon - \alpha_\epsilon}_{L^q} + \norm{\tilde{\mu}_\epsilon - \mu_\epsilon}_{L^1} < \gamma_\epsilon$,
  the pair $(\alpha_\epsilon, \mu_\epsilon)$ minimizes
  \begin{equation}\label{eq:ekeland-min}
    (\alpha, \mu) \mapsto -\nEigen{k}{p}(\alpha, \mu) + \gamma_\epsilon \norm{\alpha - \alpha_\epsilon}_{L^q} +
    \gamma_\epsilon\norm{\mu - \mu_\epsilon}_{L^1}
  \end{equation}
  on a neighborhood of $(\tilde{\alpha}_\epsilon, \tilde{\mu}_\epsilon) \in X_\epsilon$, and $\nEigen{k}{p}(\alpha_\epsilon, \mu_\epsilon) \to \NEigen{k}{p}(g)$.
  Now, consider the space $Y = (\alpha_\epsilon + L^\infty)\times L^\infty$.
  Proposition~\ref{prop:subdiff-of-norm} and Lemma~\ref{lem:lambda-alph-mu-subdiff} allow us to calculate the Clarke subdifferential of \eqref{eq:ekeland-min} restricted to $Y$
  at $(\alpha_\epsilon, \mu_\epsilon)$.
  We then apply Proposition~\ref{prop:multip-rule} to~\eqref{eq:ekeland-min} with $S = Y\cap X_\epsilon$,
  obtaining an eigenmap $\F_\epsilon \in H^1(1, \alpha_\epsilon,\R^{n_\epsilon})$ such that
  \begin{equation}\label{ineq:eigen-clark-multiplier-alpha}
    \int \brr{\nEigen{k}{p}\cdot\brr{\tfrac{\alpha_\epsilon}{\norm{\alpha_\epsilon}_{L^{q}}}}^{q-1} - \abs{d\F_\epsilon}^2 + \eta_\epsilon}\alpha  \geq 0
  \end{equation}
  and
  \begin{equation}\label{ineq:eigen-clark-multiplier-mu}
    \int \brr{\abs{\F_\epsilon}^2 - 1 - \sigma_\epsilon} d(\mu-\mu_\epsilon) \geq 0,
  \end{equation}
  where $\norm{\eta_\epsilon}_{L^{p/2}},\norm{\sigma_\epsilon}_{L^\infty} \to 0$, and $\alpha, \mu \in L^\infty$: $\alpha+\alpha_\epsilon \geq \epsilon$,
  $0\leq \mu \leq \epsilon^{-1}$.
  By taking arbitrary $\alpha \in L_+^\infty$ and $\mu = \mu_\epsilon|_{Z}$ for arbitrary sets $Z \subset M$, we obtain
  \begin{equation}\label{ineq:alpha-mu-upper-bounds}
    c_{1,\epsilon}\alpha_\epsilon^{q-1} + \eta_\epsilon -  \abs{d\F_\epsilon}^2 \geq 0
    \quad\text{and}\quad
    (1 + \sigma_\epsilon  - \abs{\F_\epsilon}^2)\big|_{\brc{\mu_\epsilon > 0}} \geq 0,
  \end{equation}
  where we set $c_{1,\epsilon}:= \nEigen{k}{p}(\alpha_\epsilon,\mu_\epsilon)\norm{\alpha_\epsilon}_{L^{q}}^{1-q}$. Note that $\abs{c_{1,\epsilon} - \lambda_\epsilon} \to 0$.
  To obtain an analogue of the last inequality to the whole $M$, one applies a weak maximum principle.
  If $c_{2,\epsilon}^2:= 1 + \norm{\sigma_\epsilon}_{L^\infty} \to 1$ and
  $0 \leq h:= \max \brc{0, \abs{\F_\epsilon} - c_{2,\epsilon}} \in H^{1,p}(M)$, we have
  \begin{equation}
    \delta(\alpha_\epsilon d\abs{\F_\epsilon}) - \lambda_\epsilon \abs{\F_\epsilon}\mu_\epsilon \leq 0,
  \end{equation}
  and since $h|_{\brc{\mu_\epsilon > 0}} = 0$,
  \begin{equation}
    0 \leq \int_{\brc{\abs{\F_\epsilon} \geq c_{2,\epsilon}}} \alpha_\epsilon\abs{d\abs{\F_\epsilon}}^2 = \int \alpha_\epsilon \brt{d\abs{\F_\epsilon}, dh}
    \leq \lambda_\epsilon\int \abs{\F_\epsilon} h d\mu_\epsilon = 0,
  \end{equation}
  resulting in
  \begin{equation}\label{ineq:mu-upper-bound-ext}
    \abs{\F_\epsilon} \leq c_{2,\epsilon}
  \end{equation}
  since $M$ is connected.

  On the other hand, we have the reverse inequalities
  \begin{equation}\label{ineq:alpha-mu-lower-bounds}
    \brr{c_{1,\epsilon}\alpha_\epsilon^{q-1} + \eta_\epsilon - \abs{d\F_\epsilon}^2}\Big|_{\brc{\alpha_\epsilon > \epsilon}} \leq 0
    \quad\text{and}\quad
    (1 + \sigma_\epsilon  - \abs{\F_\epsilon}^2)\big|_{\brc{\mu_\epsilon < \epsilon^{-1}}} \leq 0,
  \end{equation}
  since one can always take $\mu = \epsilon^{-1}$ in~\eqref{ineq:eigen-clark-multiplier-mu} and $\alpha = \max\brc{\epsilon - \alpha_\epsilon, -n}$, $n \to \infty$
  in~\eqref{ineq:eigen-clark-multiplier-alpha}. Putting~\eqref{ineq:alpha-mu-upper-bounds}, \eqref{ineq:mu-upper-bound-ext}, and \eqref{ineq:alpha-mu-lower-bounds}
  together, one has
  \begin{equation}\label{ineq:alpha-almost}
    -\eta_\epsilon \leq c_{1,\epsilon}\alpha_\epsilon^{q-1} - \abs{d\F_\epsilon}^2
    \leq -\eta_\epsilon\1_{\brc{\alpha_\epsilon > \epsilon}} + c_{1,\epsilon}\epsilon^{q-1}\1_{\brc{\alpha_\epsilon = \epsilon}},
  \end{equation}
  \begin{equation}\label{ineq:mu-almost}
    \abs{{\F}_\epsilon/ c_{2,\epsilon}}^2 \leq 1,
    \quad\text{and}\quad
    1-2\tfrac{\norm{\sigma_\epsilon}_{L^\infty}}{c_{2,\epsilon}^2} - \abs{{\F}_\epsilon / c_{2,\epsilon}}^2  \leq \1_{\brc{\mu_\epsilon = \epsilon^{-1}}}.
  \end{equation}
  One also has $\norm{\1_{\brc{\mu_\epsilon = \epsilon^{-1}}}}_{L^1} \leq \epsilon\int d\mu_\epsilon \leq \epsilon(1+\gamma_\epsilon)$.
  Finally, \eqref{ineq:alpha-almost} and \eqref{ineq:mu-almost} with $\tilde{\F}_\epsilon := \F_\epsilon / c_{2,\epsilon}$ yield the desired properties.
\end{proof}

By rescaling $(\alpha_\epsilon', \mu'_\epsilon) := (\lambda_\epsilon^{\frac{p-2}{2}}\alpha_\epsilon, \lambda_\epsilon^{\frac{p}{2}}\mu_\epsilon)$,
we have $(\alpha'_\epsilon)^{\frac{2}{p-2}} = \lambda_\epsilon\alpha_\epsilon^{\frac{2}{p-2}}$ and the following:
\begin{corollary}\label{cor:max-sequence}
  Under the assumptions of Proposition~\ref{prop:max-sequence}, there
  exist sequences $(\alpha_\epsilon, \mu_\epsilon) \in L^{q}_+(M)\times L^\infty_+(M)$ and $\F_\epsilon \in H^{1,p}(M, \R^{n_\epsilon})$
  such that
  $\lambda_k(\alpha_\epsilon,\mu_\epsilon) = 1$, both $\norm{\alpha_\epsilon}_{L^q}^{q-1}$ and $\mu_\epsilon(M)^{2/p}$ converge to $\NEigen{k}{p}(g)$, and
  \begin{enumerate}[label={(\roman*)}]
    \item $\delta (\alpha_\epsilon d\F_\epsilon) = \F_\epsilon \mu_\epsilon$ weakly;
    \item $\abs{\F_\epsilon}\leq 1$, and the norms $\norm{\F_\epsilon}_{H^{1,p}}$ are uniformly bounded;
    \item $\norm{\alpha_\epsilon - \abs{d\F_\epsilon}^{p-2}}_{L^{\frac{p}{p-2}}} \to 0$;
    \item\label{cor:max-sequence:it:almost-sphere} $\norm{1 - |\F_\epsilon|^2}_{L^1} \to 0$ and for any $\varrho > 0$,
     $v_g\brr{\brc{|\F_\epsilon|^2 < 1-\varrho}} \leq C \norm{\mu_\epsilon}_{L^\infty}^{-1}$.
  \end{enumerate}
\end{corollary}
\begin{remark}\label{rem:p-m-2}
  For $p = 2$, one could replace $\lambda_k(\alpha, \mu)$ by $\lambda_k(1, \mu)$ and obtain similar maximizing sequences with $\alpha_\epsilon \equiv 1$
  (cf.~\cite[Proposition~4.1]{Vinokurov:2025:higher-dim-harm-eigenval}).
  See also~\cite{Petrides:2024:conf-class-opt-lms} for another construction of a maximizing sequence using Ekeland's variational principle, in the case $p=2$.
\end{remark}

\subsection{Convergence and Regularity near stable points}

Consider a sequence of tensors $\F_\epsilon^*\F_\epsilon$
generated by Corollary~\ref{cor:max-sequence}. By Lemma~\ref{lem:main}, it is bounded
in $H^{1,p}\widehat{\otimes}_\pi H^{1,p}$, so
passing to a subsequence, we have a tensor
$\F^*\F \in H^{1,p}\widehat{\otimes}_\pi H^{1,p}$, a function
$\alpha \in L^{p/(p-2)}_+$ and a Radon measure
$\mu$ such that $\mu_{\epsilon} \oset{w^*}{\to} \mu$
in $\mathcal{M}_+$, $\alpha_{\epsilon} \oset{w^*}{\to} \alpha$
in $L^{{p}/(p-2)}_+$, and $\F_\epsilon^*\F_\epsilon \oset{w^*}{\to} \F^*\F$
in $H^{1,p}\widehat{\otimes}_\pi H^{1,p}$.
By Corollary~\ref{cor:l-inf-convergence} and Lemma~\ref{lem:main}, up to isometry,
we can arrange $\F_\epsilon \to \F$ in $L^p$, so Corollary~\ref{cor:max-sequence}\ref{cor:max-sequence:it:almost-sphere} then implies
\begin{equation}\label{eq:limit-map-to-sphere}
  \abs{\F} \equiv 1 \quad\text{ on } M.
\end{equation}
\begin{lemma}\label{lem:strong-conv}
  Let $\tilde{\Omega} \subset M$ be a bounded open set in a Riemannian manifold such that
  ${\mu_\epsilon \oset{w^*}{\to} \mu}$ in
  $\mathcal{M}_+(\tilde{\Omega})$ and $\alpha_{\epsilon} \oset{w^*}{\to} \alpha$
  in $L^{p/(p-2)}_+(\tilde{\Omega})$, where $p \in (2,\infty)$.
  Assume also that $\F_\epsilon \in H^{1,p}(\tilde{\Omega},\ell^2)$,
  $\F_\epsilon^*\F_\epsilon \oset{w^*}{\to} \F^*\F$ in
  $H^{1,p}(\tilde{\Omega})\widehat{\otimes}_\pi H^{1,p}(\tilde{\Omega})$
  such that
  \begin{enumerate}[label=(\roman*)]
    \item $\delta (\alpha_\epsilon d\F_\epsilon) = \F_\epsilon
    \mu_\epsilon$ weakly;
    \item \label{it:num-cond} $\norm{\alpha_\epsilon - \abs{d\F_\epsilon}^{p-2}}_{L^{\frac{p}{p-2}}(\tilde{\Omega})} \to 0$;
    \item \label{it:lower-than}$\limsup_{\epsilon\to0}\int (\abs{\F_\epsilon}^2 -\abs{\F}^2)\phi^2d\mu_\epsilon \leq 0
      \quad\forall \phi \in C^\infty_{0}(\tilde{\Omega}).$
  \end{enumerate}
  If a point $x \in \tilde{\Omega}$ is stable (see Definition~\ref{def:stable-points})
  for the sequence $(\alpha_\epsilon,\mu_\epsilon)$, then there exists a neighborhood $\Omega \subset \tilde{\Omega}$ such that,
  up to a subsequence, the following hold:
  \begin{enumerate}
    \item $\int \phi^2 d\mu \leq \int \alpha\abs{d\phi}^2 $ for all $\phi \in C^\infty_0(\Omega)$;
    \item $\alpha_{\epsilon} \to \abs{d\F}^{p-2}$ in $L^{\frac{p}{p-2}}(\Omega)$, that is, $\alpha = \abs{d\F}^{p-2}$;
    \item $\tf{\F_\epsilon} \to \tf{\F}$ in $H^{1,p}(\Omega)\widehat{\otimes}_\pi H^{1,p}(\Omega)$;
    \item $\delta(\alpha d\tilde{\F}) = \tilde{\F} \mu$ for any $\tilde{\F} \in H^{1,p}(\Omega,\ell^2)$ such that $\tf{\F} = \tf{\tilde{\F}}$.
  \end{enumerate}
\end{lemma}
\begin{proof}
  Let us choose a neighborhood $\Omega \subset \tilde{\Omega}$ and a subsequence $\brc{\mu_\epsilon}$
  from the definition of a stable point. Then \eqref{ineq:stable-point} implies
  that $\brc{\mu_\epsilon}$ is bounded in $\mathfrak{Bil}[H^{1}_{0}(\Omega,\alpha_\epsilon)]$
  and hence in $\mathfrak{Bil}[H^{1,p}_{0}(\Omega)]$ as well. It is then precompact in the weak$^*$
  topology and $\mathfrak{Bil}[H^{1,p}_{0}(\Omega)] = (H^{1,p}_0(\Omega)\widehat{\otimes}_\pi H^{1,p}_0(\Omega))^*$. By looking at the integrals
  $\int \phi^2 d\mu_\epsilon$ of smooth functions
  $\phi \in C^\infty_{0}(\Omega)$, one sees that $\mu$ is the only
  accumulation point of $\brc{\mu_\epsilon} \subset\mathfrak{Bil}[H^{1,p}_{0}]$.

  Further, choose a function $\phi \in C^\infty_0(\Omega)$, set $q = \frac{p}{p-2} \in (1,\infty)$ and $q' = \frac{p}{2}$.
  Note that one also has
  \begin{equation}\label{ineq:fatou}
    \begin{split}
      \int \alpha^q \phi^2 &\leq \liminf_{\epsilon\to0} \int \alpha^q_\epsilon \phi^2
      = \liminf_{\epsilon\to0} \int \abs{d\F_\epsilon}^p\phi^2,
      \\
      \int \abs{d\F}^p\phi^2 &\leq \liminf_{\epsilon\to0} \int \abs{d\F_\epsilon}^p\phi^2
    \end{split}
  \end{equation}
  because of the weak$^*$ convergence of $\alpha_\epsilon$ and Lemma~\ref{lem:main}. Exactly as in
  \cite[Lemma~3.6]{Vinokurov:2025:sym-eigen-val-lms}, one proves that
  \begin{equation}\label{ineq:limsup}
    \limsup_{\epsilon\to0}\int \alpha_\epsilon\abs{d\F_\epsilon}^2\phi^2\leq \int \alpha\abs{d\F}^2\phi^2
    \quad \forall \phi \in C^\infty_0(\Omega).
  \end{equation}
  By~\eqref{ineq:fatou}, Hölder's inequality (with the density $\phi^2$) and \ref{it:num-cond}, we obtain
  \begin{equation}\label{ineq:liminf-sup}
    \begin{split}
      \limsup_{\epsilon\to0}\int \abs{d\F_\epsilon}^{p}\phi^2 &\leq \int \alpha\abs{d\F}^2\phi^2
      \leq \brr{\int \alpha^q \phi^2}^{1/q}\brr{\int \abs{d\F}^p \phi^2}^{1/q'}
      \\ &\leq \liminf_{\epsilon\to0} \int \abs{d\F_\epsilon}^p\phi^2.
    \end{split}
  \end{equation}
  Therefore, all the inequalities are equalities. In particular, $\int \abs{d\F_\epsilon}^{p}\phi^2 \to \int \abs{d\F}^{p}\phi^2$.
  Since it is true for all $\phi \in C^\infty_0(\Omega)$, after shrinking $\Omega$ slightly,
  we conclude that
  \begin{equation}\label{eq:norm-conv}
    \norm{\abs{d\F_\epsilon}^{p-2}}_{L^{q}(\Omega)} \to \norm{\abs{d\F}^{p-2}}_{L^{q}(\Omega)}.
  \end{equation}
  Hölder's inequality in the middle of~\eqref{ineq:liminf-sup} yields $\alpha = \abs{d\F}^{p-2}$.
  Hence, both $\alpha_\epsilon$ and $\abs{d\F_\epsilon}^{p-2}$ converge weakly$^*$ to $\alpha = \abs{d\F}^{p-2}$, which together with~\eqref{eq:norm-conv}
  yields
  \begin{equation}\label{alpha-strong-conv}
    \alpha_\epsilon, \abs{d\F_\epsilon}^{p-2} \to \abs{d\F}^{p-2} \quad \text{in } L^{q}(\Omega).
  \end{equation}
   Then Lemma~\ref{lem:main}\ref{lem:main:it:5} tells us that $\tf{\F_\epsilon} \to \tf{\F}$ in $H^{1,p}(\Omega)\widehat{\otimes}_\pi H^{1,p}(\Omega)$,
   and, up to isometry, $u_\epsilon \to u$ in $H^{1,p}(\Omega,\ell^2)$. As a consequence, for any $\phi \in C^\infty_0(\Omega)$,
  \begin{multline}
    \int \alpha\brt{d\F,d\phi} = \lim_{\epsilon\to 0} \int \alpha_\epsilon\brt{d\F_\epsilon,d\phi}
    =\lim_{\epsilon\to 0} \int \phi \F_{\epsilon} d\mu_{\epsilon}
    \\=\lim_{\epsilon\to 0}
    \int \phi \brr{\F_{\epsilon} - \F}d\mu_\epsilon
    +\lim_{\epsilon\to 0}
    \int \phi \F d\mu_\epsilon = \int \phi \F d\mu,
  \end{multline}
  where $\mu_{\epsilon} \overset{w^*}{\to} \mu$ as bilinear forms on $H^{1,p}_{0}$. The limiting equation holds for any
  $\tilde{\F}$ such that $\tf{\F}=\tf{\tilde{\F}}$, since any two such maps differ by an isometry (see Lemma~\ref{lem:main}).
\end{proof}

\begin{lemma}\label{lem:regularity-in-good-pt}
  Let $\mu_\epsilon \oset{w^*}{\to} \mu \in \mathcal{M}_+(\Omega)$,
  $\F_\epsilon^*\F_\epsilon \oset{w^*}{\to} \F^*\F$ in
  $H^{1,p}(\Omega)\widehat{\otimes}_\pi H^{1,p}(\Omega)$ as in Lemma~\ref{lem:strong-conv} (with $\Omega = \tilde{\Omega}$),
  and let $p \in (2,m]$. If
  $\abs{\F} \equiv 1$ and $\Omega$ contains only a discrete set of unstable points, then
  \begin{equation}
    \delta (\abs{d\F}^{p-2}d\F) = \abs{d\F}^p\F,
  \end{equation}
  that is, $\F$ is a (weakly) $p$-harmonic map into $\mathbb{S}^\infty \subset \ell^2$.
  Moreover, $\mu^c = \alpha\abs{d\F}^2$, where $\mu^c$ denotes the nonatomic part of $\mu$, and
  $\F$ is locally spectrally stable on $\Omega$ (see Definition~\ref{def:stable-map}).
\end{lemma}
\begin{proof}
  Let us decompose $\mu$ into its nonatomic and atomic parts
  \begin{equation}\label{eq:meas-decomp}
    \mu=\mu^c+\sum_{x \in A(\mu)}w_x\delta_x,
  \end{equation}
  where $w_x > 0$ and $A(\mu)$ is the set of atoms of $\mu$.
  If the set of unstable points is denoted by $D$, the stability inequality \eqref{ineq:stable-point} implies that
  $A(\mu) \subset D$. By Lemmas~\ref{lem:strong-conv} and~\ref{lem:main}, $\delta (\alpha du) = u \mu = u \mu^c$ on $\Omega\setminus D$. Since
  discrete sets have $p$-capacity zero, it extends on the whole $\Omega$, and, up to a representative of $\F^*\F$, we have
  \begin{equation}\label{eigen.map}
    \delta (\alpha d\F) = \F \mu^c \quad\text{on } \Omega.
  \end{equation}

  The condition $\abs{\F}^2 = 1$ implies $\F \cdot d\F = 0$, and
  after testing equation~\eqref{eigen.map} against $\phi \F \in H^{1,p}$,
  we conclude that
  \begin{equation}\label{eq:mu-alpha}
    \mu^c=\alpha|d\F|^2.
  \end{equation}
  Now, let us recall that $\abs{d\F}^{p-2} = \alpha$. Given the equality above,
  equation~\eqref{eigen.map} transforms into
  \begin{equation}
    \delta (\abs{d\F}^{p-2} d\F) = \abs{d\F}^{p}u \quad\text{on } \Omega.
  \end{equation}
  Thus, $\F\colon \Omega \to \mathbb{S}^{\infty}$ is a weakly $p$-harmonic map.

  Inequality~\eqref{ineq:stable-map} obviously holds in a neighborhood of a
  stable point. On a neighborhood $U$ of an unstable point $y$ we have \eqref{ineq:unstable-point}, and
  after taking the limit, one has
  \begin{equation}
    \int \phi^2 \abs{d\F}^{p} \leq \int \abs{d\F}^{p-2} \abs{d\phi}^2
    \quad \forall \phi \in C^\infty_0\brr{U\setminus\brc{y}}
  \end{equation}
  Again, the inequality continues to hold over a discrete set.
\end{proof}

\subsection{Bubbling analysis}\label{app:blow-up}

Let $\brr{\Omega,g}$ be a compact Riemannian manifold (possibly with
a nonempty boundary $\partial \Omega$) and $\dim \Omega = m \geq 2$.
Denote by $\mathcal{M}^c_+(\Omega)$ the space of all nonnegative, nonatomic Radon measures
on $\Omega$. Consider a sequence $\brc{\mu_\varepsilon} \subset \mathcal{M}^c_+(\Omega)$
such that $\mu_\varepsilon \oset{w^*}{\to}\mu$. Note that $\mu$ is decomposable into its nonatomic and atomic parts,
\begin{equation}
  \mu = \mu^c + \sum_{x \in A(\mu)} w_x \delta_{x},
\end{equation}
where $A(\mu) = \brc{x \in \Omega \colon \mu(\brc{x}) > 0}$ and $w_x  = \mu(\brc{x})$.

\begin{definition}\label{def:bubble-tree}
  Fix $\delta > 0$. We say that a converging sequence
  $\mathcal{M}^c_+(\Omega) \ni \mu_\varepsilon \oset{w^*}{\to}\mu$
  has the $\delta$-\emph{bubbling property} if, up to a subsequence,
  there exists
  a finite collection of maps
  $\brc{T_{i,\epsilon} \colon U_i \to \R^m}_{1\leq i \leq N}$, where
  $U_i \subset \Omega$ is a coordinate chart centered at $x_i \in A(\mu)$
  with the following properties.
  \begin{enumerate}[label=(\roman*)]
    \item The maps $T_{i,\epsilon}$ have the form $T_{i,\epsilon}(x) = (x - x_{i,\epsilon})/{\sigma_{i,\epsilon}}$, where
      $x_{i,\epsilon} \to x_i$, $\sigma_{i,\epsilon} \to 0$.
    \item For any $R > 0$ and small enough $\epsilon$, either
      $B_{\sigma_{i,\epsilon} R}(x_{i,\epsilon}) \cap B_{\sigma_{j,\epsilon} R}(x_{j,\epsilon}) = \varnothing$,
      or one of them is contained in the other, say
      $B_{\sigma_{i,\epsilon} R}(x_{i,\epsilon}) \subset B_{\sigma_{j,\epsilon} R}(x_{j,\epsilon})$ and the images
      $T_{j,\epsilon}[B_{\sigma_{i,\epsilon} R}(x_{i,\epsilon})]$ converge to a point in $\R^m$,
      in other words,
      \begin{equation}\label{eq:dif-rate-condition}
        \frac{\operatorname{dist}(x_{i,\epsilon},x_{j,\epsilon})}{\sigma_{i,\epsilon} + \sigma_{j,\epsilon}}
        +\frac{\sigma_{i,\epsilon}}{\sigma_{j,\epsilon}}
        +\frac{\sigma_{j,\epsilon}}{\sigma_{i,\epsilon}} \to \infty.
      \end{equation}
    \item There exist $w^\infty_{i} \geq 0$ (weights at infinity) such that $(T_{i,\epsilon})_* [\mu_\epsilon] \overset{w^*}{\to} \mu_i \in \mathcal{M}_+(\R^m)$,
      \begin{equation}\label{ineq:alomst-no-loss}
        \mu(\Omega) = \mu^c(\Omega) + \sum_i \mu_{i}^c(\R^m) + \sum_i w^\infty_{i},
        \quad\text{and}\quad
        \sum_i w^\infty_{i} < \delta \mu(\Omega).
      \end{equation}
    \item\label{it:infinity-functions} There exist functions $\phi_{i,\epsilon} \in H^{1,m}\cap C_0(\R^m)$ such that
    $\supp T_{i,\epsilon}^*\phi_{i,\epsilon} \cap \supp T_{j,\epsilon}^*\phi_{j,\epsilon} = \varnothing$, $\operatorname{dist}(0, \supp \phi_{i,\epsilon}) \to \infty$,
    $\int T_{i,\epsilon}^*\phi_{i,\epsilon} d\mu_\epsilon \to w^\infty_i$, and $\int \abs{d\phi_{i,\epsilon}}^m \to 0$.
  \end{enumerate}
  Here the convergence $(T_\epsilon)_* [\mu_\epsilon] \overset{w^*}{\to}
  \mu$ means
  \begin{equation}
    \int_\Omega \phi\brr{\frac{x - x_\epsilon}{\sigma_\epsilon}}d\mu_\epsilon(x) \to \int_{\R^m} \phi(y) d\mu(y)\quad \forall \phi \in C_0(\R^m).
  \end{equation}
\end{definition}
\begin{remark}
  Let us briefly clarify the interpretation of the objects appearing in the
  above bubble-tree decomposition.
  The maps $T_{i,\varepsilon}$ resolve concentration scales at which a
  nontrivial part of the measure $\mu_\varepsilon$ is detected. The
  nonatomic part $\mu_i^c$ of the corresponding limiting measure
  represents a bubble at this scale, while atoms of $\mu_i$ in
  $\mathbb R^m$ give rise, after further rescaling, to bubbles at finer
  scales.

  The weight $w_i^\infty$, on the other hand, records the mass escaping
  every fixed compact subset of $\mathbb R^m$ in the $i$-th rescaled
  coordinates. Equivalently, under the compactification
  $\Sph^m \simeq \mathbb R^m \cup \{\infty\}$, it corresponds to mass at the
  opposite pole $\infty$. This can be viewed as mass lying
  between the scales of the retained bubbles, somewhat analogously to
  neck regions in the usual bubble-tree picture. This analogy is only heuristic:
  for our purposes, it is enough to keep track of this mass through
  suitable test functions in the Rayleigh quotient, without carrying out a separate neck analysis.

  Indeed, test functions on the retained bubbles can be cut off and pulled back to $\Omega$ using $T_{i,\varepsilon}$.
  For the mass at infinity, condition~\ref{it:infinity-functions} instead provides auxiliary
  test functions detecting $w_i^\infty$ whose $m$-energies tend to zero.
  Thus each nonzero weight at infinity contributes one such test function
  to the dimension count, while
  \begin{equation}
    \sum_i w_i^\infty < \delta\,\mu(\Omega).
  \end{equation}
\end{remark}

\begin{proposition}\label{prop:no-atoms}
  Let $p \in [2,m]$, $\mathcal{M}^c_+(\Omega) \ni \mu_\varepsilon \overset{w^*}{\to} \mu \neq 0$
  and $\nEigen{k}{p}(\alpha_\epsilon,\mu_\epsilon) \geq c > 0$.
  \begin{itemize}
    \item Then for any $\delta > 0$,
    $\brc{\mu_\varepsilon}$ has the $\delta$-bubbling property
    with $\#\set{i}{\mu^c_i \neq 0\ \text{or}\ w^\infty_i \neq 0}\leq k$. If
    $\supp \mu_\epsilon \subset \partial \Omega$, one has
    $T_{i,\epsilon} \colon U_i \to \R^m_+$.

    \item Moreover, if $p < m$, then $\mu$ has no atoms, that is, $\mu \in \mathcal{M}^c_+(\Omega)$.
  \end{itemize}
\end{proposition}
\begin{proof}
  It is enough to assume that $\norm{\alpha_\epsilon}_{L^{\frac{p}{p-2}}} = 1$ and $\lambda_k(\alpha_\epsilon,\mu_\epsilon) \geq c > 0$.

  The first part of the proposition was proved in~\cite[Appendix~A.1]{Vinokurov:2025:sym-eigen-val-lms} and follows the argument of~\cite[Proposition~A.6]{Vinokurov:2025:sym-eigen-val-lms}.
  Briefly, one builds a bubble tree and shows that for any $t > 0$, for each atom $y_i$ (including atoms at infinity) and each nonzero $\mu_i^c$, there exists a sequence of test functions $\phi_{i,\epsilon}$
  such that $\int \phi_{i,\epsilon}^2 d\mu_\epsilon \geq c_1$, $\int \abs{d\phi_{i,\epsilon}}^m < t$, and the supports are pairwise disjoint
  for different $i$. Consequently, the construction terminates after at most $k$ such sequences.

  To prove the second part, take any atom $x_0$ yielding a map $T_\epsilon \colon B_r(x_0) \to \R^m$ such that
  $\tilde{\mu}_\epsilon := (T_\epsilon)_* [\mu_\epsilon] \oset{w^*}{\to} \tilde{\mu} \in
  \mathcal{M}^c_+(\R^m\setminus D)$, where $\tilde{\mu}\neq 0$ and $D$ is a finite set. Then we choose any
  $k+1$ separately supported functions
  $\psi_i \in C^\infty_0(\R^m\setminus D)$ with
  \begin{equation}
    \int_{\mathbb{R}^m} \psi_i^2 d\tilde{\mu} \ge c_2.
  \end{equation}
  Thus,
  \begin{align}
    \lambda_k(\alpha_\epsilon,\mu_\varepsilon) &\le \max_{0\le i \le k}
    \frac{\int_{\Omega} \abs{d(T_\epsilon^*\psi_i)}^2_g \alpha_\epsilon dv_g}
    {\int_{\Omega} (T_\epsilon^*\psi_i)^2 d\mu_\varepsilon}
    \le  \max_{0\le i \le k}
    \frac{\brr{\int_{\Omega} \abs{d(T_\epsilon^*\psi_i)}^p_g dv_g}^{2/p}}
    {\int_{\Omega} (T_\epsilon^*\psi_i)^2 d\mu_\varepsilon}
    \\ &\le  \sigma_\varepsilon^{2(\frac{m}{p}-1)}\max_{0\le i \le k}
    \frac{c_3\brr{\int_{\mathbb{R}^m} \abs{d\psi_i}^p}^{2/p}}
    {\int_{\mathbb{R}^m} \psi_i^2 d\tilde{\mu}_\epsilon} \to 0,
  \end{align}
  as $\sigma_\varepsilon \to 0$.
\end{proof}

Let us use the convention
$\nEigen{k}{m}(\alpha,\mu) =: \overline{\lambda}_k(\alpha,\mu) = 0$ if either $\alpha = 0$ or $\mu = 0$.
Also, define $\#w^\infty := \#\set{i}{w^\infty_i \neq 0}$ and $\abs{w^\infty} := \sum_i w^\infty_i$.
\begin{proposition}\label{prop:no-atoms-m}
  Let $\mathcal{M}^c_+(\Omega) \ni \mu_\varepsilon \overset{w^*}{\to} \mu$
  and $L_{+}^{\frac{m}{m-2}}(\Omega) \ni \alpha_\epsilon \overset{w^*}{\to} \alpha$
  with $\mu_\epsilon(\Omega) = 1 = \norm{\alpha_\epsilon}_{L^\frac{m}{m-2}}$.
  If $\lambda_k(\alpha_\epsilon,\mu_\varepsilon) \ge c > 0$, then
  for any $\delta > 0$,
  there exist (possibly zero) $\brc{\mu_i}_{1\leq i \leq k}
  \subset \mathcal{M}_+^c(\mathbb{S}^m)$, $\brc{\alpha_i}_{1\leq i \leq k}
  \subset L^{m/(m-2)}_{+}(\mathbb{S}^m)$, weights at infinity $w^\infty_i \geq 0$,
  and numbers $k_i \in \mathbb{N}$ such that $\abs{w^\infty}< \delta$ and
  \begin{equation}
    (1 - \abs{w^\infty})^{m/2}\brs{\limsup_{\epsilon\to 0}\overbar{\lambda}_{k}^{m/2}(\alpha_\epsilon,\mu_\epsilon)}
    \leq \overbar{\lambda}_{k_0}^{m/2}\brr{\alpha,\mu^c} +
    \sum_{i=1}^k \overbar{\lambda}_{k_i}^{m/2}\brr{\alpha_i,\mu_i},
  \end{equation}
  where $\sum_{i = 0}^k k_i = k - \#w^\infty$.
\end{proposition}
\begin{proof}
  The proof is analogous to that of~\cite[Corollary A.7]{Vinokurov:2025:sym-eigen-val-lms}, so we point
  out only the differences to be made. By $\rho_\epsilon$, $\rho$, we will denote the corresponding powers of
  $\alpha$, $\alpha_\epsilon$ so that $\rho_\epsilon^{m-2} = \alpha_\epsilon$, etc.

  Fix $\delta > 0$ to be chosen in the end.
  Then by Proposition~\ref{prop:no-atoms},
  the sequence has the $\delta$-bubbling property with $(1-\abs{w^\infty}) \leq \mu^c(\Omega) + \sum_{i = 1}^k \mu_{i}^c(\R^m) \leq 1$.
  Let $\tilde{\mu}_i = \mu^c_{i}$ be the nonatomic
  part of the bubble measure $\mu_i$, where
  $(T_{i,\epsilon})_* [\mu_\epsilon] \oset{w^*}{\to} \mu_i \in \mathcal{M}_+(\R^m)$, and
  let $g_{i,\epsilon}$ be the pullbacks of the metrics $\rho_\epsilon^2 g$ under the maps
  $(T_{i,\epsilon})^{-1}$. Note that for any fixed $B_R \subset \R^m$,
  \begin{equation}
    \Vol_{\rho^2g}(\Omega) + \limsup_{\epsilon\to0}\sum_i\Vol_{g_{i,\epsilon}}(B_R) \leq  \limsup_{\epsilon\to0}\Vol_{\rho^2_\epsilon g}(\Omega) = 1,
  \end{equation}
  and, up to a subsequence, after taking
  weak $L^{\frac{m}{m-2}}$-limits, we have
  \begin{equation}
    \int_{\mathbb{S}^m}\abs{d\phi}_{g_{i,\epsilon}}^2 dv_{g_{i,\epsilon}} \to
    \int_{\mathbb{S}^m}\abs{d\phi}_{g_{\Sph^m}}^2 \tilde{\rho}^{m-2}_i dv_{g_{\Sph^m}}
  \end{equation}
  for any $\phi \in C^\infty_0(\R^m)$
  (where $\mathbb{S}^m \approx \R^m \cup \brc{\infty}$) and some (possibly zero) $\tilde{\rho}_i \in L^{m}_{+}(\mathbb{S}^m)$ with
  \begin{equation}
    \int_\Omega \rho^m dv_g +  \sum_i\int_{\Sph^m} \tilde{\rho}_i^m dv_{g_{\Sph^m}} \leq 1.
  \end{equation}
  Also define
  $\tilde{\mu}_0 := \mu^c$, $\tilde{\rho}_0 = \rho$ and denote
  \begin{equation}
    \tilde{\mu} = \tilde{\mu}_0 \sqcup \cdots \sqcup \tilde{\mu}_k,
    \quad \tilde{\rho} = \tilde{\rho}_0 \sqcup \cdots \sqcup \tilde{\rho}_k.
  \end{equation}

  In \cite[Corollary A.7]{Vinokurov:2025:sym-eigen-val-lms}, we pull back the corresponding
  test functions from
  the disjoint union $(\tilde{\Omega}, \tilde{\mu}) := (\Omega,\mu^c)
  \sqcup (\Sph^m,\tilde{\mu}_1) \sqcup \cdots \sqcup (\Sph^m,\tilde{\mu}_k)$ to $\Omega$. If one also adds
  the functions from Definition~\ref{def:bubble-tree}\ref{it:infinity-functions}, then for $\#w^\infty + \sum_{i = 0}^k k_i = k$, we obtain
  \begin{equation}
   \begin{aligned}
    (1-\abs{w^\infty})\brs{\limsup_{\epsilon\to 0}\overline{\lambda}_k(\rho_\epsilon^{m-2},\mu_\epsilon)}
    &\leq \sum_i \norm{\tilde{\rho}_i}_{L^m}^{m-2} \overline{\lambda}_{k_i}(\tilde{\rho}^{m-2}_i,\tilde{\mu}_i).
    \\ &\leq \brr{\sum_i \overline{\lambda}_{k_i}^{m/2}(\tilde{\rho}^{m-2}_i,\tilde{\mu}_i)}^{2/m}.
   \end{aligned}
  \end{equation}
\end{proof}

\subsection{Proof of Theorem~\ref{thm:main-0}}
By Proposition~\ref{prop:upper-cont-and-bound}, the supremum remains unchanged when extended to arbitrary
$(\alpha,\mu) \in L^{p/(p-2)}_+\times L^1_+$. Now,
let $(\alpha,\mu) \in L^{p/(p-2)}_+\times L^1_+$ be a maximizer for $\nEigen{k}{p}$. Define
$\tilde{\alpha}_\epsilon := \max\brc{\alpha, \epsilon}$, $\tilde{\mu}_\epsilon := \min\brc{\mu, 1/\epsilon}$.
By the same proposition, $\nEigen{k}{p}(\tilde{\alpha}_\epsilon,\tilde{\mu}_\epsilon) \to \nEigen{k}{p}(\alpha,\mu)$.
We then find a maximizing sequence $(\alpha_\epsilon,\mu_\epsilon)$ of Corollary~\ref{cor:max-sequence}, applied to $(\tilde{\alpha}_\epsilon,\tilde{\mu}_\epsilon)$,
and we have that $(\alpha_\epsilon,\mu_\epsilon) \to (c_1\alpha,c_2\mu)$ in $L^{p/(p-2)}_+\times L^1_+$ for some nonzero constants.
It remains to invoke Lemmas~\ref{lem:strong-conv},~\ref{lem:regularity-in-good-pt}, with $\Omega = M$ for the latter.
Note that condition~\ref{it:lower-than} in Lemma~\ref{lem:strong-conv} is trivially satisfied
since $\abs{\F_\epsilon} \leq 1 = \abs{\F}$ by~\eqref{eq:limit-map-to-sphere}.
Thus, we have a $p$-harmonic map $u\in H^{1,p}(M, \Sph^\infty)$ with the
required properties, where the regularity of $u$ follows from Theorem~\ref{thm:reg-harm-map}. Moreover,
$\alpha = \abs{d\F}^{p-2}$, and $\mu = \mu^c = \abs{d\F}^{p}$. The same arguments work with $\alpha \equiv 1$ when $p = 2$; see Remark~\ref{rem:p-m-2}
and \cite[Proposition~2.9]{Vinokurov:2025:higher-dim-harm-eigenval}.

\subsection{Proof of Theorems~\ref{thm:main-2},~\ref{thm:main-1}}

We start from an arbitrary maximizing sequence and apply Corollary~\ref{cor:max-sequence} and Lemma~\ref{lem:regularity-in-good-pt}
(with $\Omega = M$).
Thus, we have a $p$-harmonic map $u\in H^{1,p}(M, \Sph^\infty)$ with $\alpha = \abs{d\F}^{p-2}$,
the limiting measure $\mu$ has the form
\begin{equation}\label{eq:sup-decom}
  \mu=\abs{d\F}^p +\sum_{x \in A(\mu)}w_x\delta_x,
\end{equation}
$\Lambda_{k,p}(g) = \mu(M)^{2/p}$, and $\norm{\alpha}_{L^\frac{p}{p-2}} \leq \mu(M)^{1-\frac{2}{p}}$

\paragraph{The case $\mathbf{p < m}$.}
Since $\lambda_k(\alpha_\epsilon,\mu_\epsilon) = 1$,
we conclude from Proposition~\ref{prop:no-atoms} that there are no atoms, so $\mu = \mu^c= \abs{d\F}^p$.
Finally, Proposition~\ref{prop:upper-cont-and-bound} implies $1 \leq \lambda_k(\abs{d\F}^{p-2},\abs{d\F}^p)$, so
\begin{equation}
  \Lambda_{k,p}(g) = \mu(M)^{2/p} \leq \frac{\mu(M)}{\norm{\alpha}_{L^\frac{p}{p-2}}} \leq \nEigen{k}{p}(\abs{d\F}^{p-2},\abs{d\F}^p) \leq \Lambda_{k,p}(g).
\end{equation}
Thus, the inequalities are, in fact, equalities.

\paragraph{The case $\mathbf{p = m}$.}
First, prove the following upper bound:
\begin{equation}
  \Lambda_{k,m}([g])^{m/2} \leq \overline{\lambda}_{k_0}(\abs{du}^2 g)^{m/2} + \sum_i\Lambda_{k_i,m}(\Sph^m,[g_{\Sph^m}])^{m/2}.
\end{equation}
for some $m$-harmonic map $u \in C^{1,\gamma}(M,\Sph^n)$ and $\sum k_i = k$.

Recall the decomposition~\eqref{eq:sup-decom}.
If $A(\mu) \neq \varnothing$, we handle the bubbles
as described in Definition~\ref{def:bubble-tree}.
By Proposition~\ref{prop:no-atoms-m}, for any $\delta > 0$, we have
\begin{equation}\label{ineq:k-m-upper-bound}
  (1-\abs{w^\infty})^{m/2}\Lambda_{k,m}([g])^{m/2} \leq \overline{\lambda}_{k_0}(\abs{du}^2 g)^{m/2}
  + \sum_{i=1}^k \overline{\lambda}_{k_i,m}\brr{\alpha_i,\mu_i}^{m/2},
\end{equation}
where $\alpha_i \in L_+^{\frac{m}{m-2}}(\Sph^m)$, $\sum_{i=0}^k k_i = k - \#w^\infty$, and $\abs{w^\infty} < \delta$.
It remains to show that $\mu_i \in L^1_+(\Sph^m)$.
\begin{claim}
  If one starts from a maximizing sequence given
  by Corollary~\ref{cor:max-sequence}, one obtains
  $L^1$-bubble measures $\mu_i$.
\end{claim}
\begin{proof}
  Recall from Section~\ref{app:blow-up} that $\mu_i$
  are the nonatomic parts of the
  weak$^*$ limits
  of $\tilde{\mu}_\epsilon := (T_\epsilon)_*[\mu_\varepsilon]$, $T_\epsilon(x) = (x - x_\epsilon)/{\sigma_\epsilon}$,
  and $\alpha_i$ are the weak $L^{\frac{m}{m-2}}$-limits of $\tilde{\alpha}_\epsilon$, where
  \begin{equation}
    \tilde{\alpha}^{\frac{2}{m-2}}_\epsilon(y) = \sigma^2_\epsilon(1 + \abs{y}^2)^2{\alpha}^{\frac{2}{m-2}}_\epsilon(\sigma_\epsilon y + x_\epsilon),
  \end{equation}
  $g_\epsilon(y) = \frac{1}{(1 + |y|^2)^2}g(\sigma_\epsilon y + x_\epsilon)$,
  and $\tilde{g}_\epsilon = \tilde{\alpha}^{\frac{2}{m-2}}_\epsilon g_\epsilon$ is the pullback of ${\alpha}^{\frac{2}{m-2}}_\epsilon g$. Note that $g_\epsilon \to g_{\Sph^m}$ in
  $C^\infty(\R^m)$, that is, on every bounded set.

  The idea of the proof follows that of \cite[Proposition~5.5]{Karpukhin-Nadirashvili-Penskoi-Polterovich:2022:existence},
  and it is to show that the pairs $(\tilde{\alpha}_\epsilon,\tilde{\mu}_\epsilon)$
  have properties of Corollary~\ref{cor:max-sequence} on
  any bounded ball of $\R^m$. Let us fix a ball $B_R = B_R(0)$ and consider the push-forward equation
  \begin{equation}
    \Delta_{\tilde{g}_\epsilon}\tilde{\F}_\epsilon := \delta_{g_\epsilon}(\tilde{\alpha}_\epsilon d\tilde{\F}_\epsilon) =
    \tilde{\F}_\epsilon \tilde{\mu}_\epsilon,
    \text{  where  }
    \tilde{\F}_\epsilon(y) = \F_\epsilon(\sigma_\epsilon y + x_\epsilon),
    \ y \in B_R.
  \end{equation}
  Then the norms $\norm{d\tilde{\F}_\epsilon}_{L^m(B_R)}$ are conformally
  invariant, hence uniformly bounded for the metrics $g_\epsilon \to g_{\Sph^m}$.
  We also have
  \begin{equation}
    \begin{split}
        \int_{B_R}\abs{\tilde{\alpha}_\epsilon - \abs{d\tilde{\F}_\epsilon}^{m-2}_{g_\epsilon}}^{\frac{m}{m-2}} dv_{g_\epsilon}
        &= \int_{B_R}\abs{(T_\epsilon^{-1})^*\alpha_\epsilon - \abs{d\tilde{\F}_\epsilon}^{m-2}_{\tilde{g}_\epsilon}}^{\frac{m}{m-2}} dv_{\tilde{g}_\epsilon}
    \\ &\leq \int_{M}\abs{\alpha_\epsilon - \abs{d\F_\epsilon}^{m-2}_{g}}^{\frac{m}{m-2}} dv_{g} \to 0.
    \end{split}
  \end{equation}
  So one can extract
  weakly convergent subsequences
  $\tf{\tilde{\F}}_\epsilon \oset{w^*}{\to} \tf{\tilde{\F}}$ and
  $\tilde{\mu}_\epsilon \oset{w^*}{\to} \tilde{\mu}$ and
  apply Lemma~\ref{lem:strong-conv} and Lemma~\ref{lem:regularity-in-good-pt} (they still hold under the convergence
  $g_\epsilon \to g_{\Sph^m}$)
  with $\Omega = B_R$ to show the regularity of $\tilde{\mu}$.
  To complete the argument, it remains to verify property~\ref{cor:max-sequence:it:almost-sphere} of Corollary~\ref{cor:max-sequence}, which can be done
  exactly as in \cite[Claim~3.9]{Vinokurov:2025:sym-eigen-val-lms}: for any $\varrho > 0$,
  \begin{align}
    \norm{\tilde{\mu}_\epsilon}_{L^\infty(B_R, g_\epsilon)}v_{g_\epsilon}\brr{\brc{|\tilde{\F}_\epsilon|^2 < 1-\varrho}}
    &\leq c_R \norm{\mu_\epsilon}_{L^\infty(M, g)}v_g\brr{\brc{|\F_\epsilon|^2 < 1-\varrho}}
    \\ &\leq c_R C.
  \end{align}
  Thus, either a subsequence of $\norm{\tilde{\mu}_\epsilon}_{L^\infty(B_R, g_\epsilon)}$ is unbounded, in which case
  $v_{g_\epsilon}\brr{\brc{|\tilde{\F}_\epsilon|^2 < 1-\varrho}} \to 0$ implying $\norm{1-|\tilde{\F}_\epsilon|^2}_{L^1(B_R, g_\epsilon)} \to 0$,
  and we proceed with Lemma~\ref{lem:regularity-in-good-pt}; or the norms $\norm{\tilde{\mu}_\epsilon}_{L^\infty(B_R, g_\epsilon)}$ are bounded, and
  we still have $\forall R > 0,\ \tilde{\mu} \in L^\infty(B_R)$, implying $\tilde{\mu} \in L^1(\Sph^m)$.
\end{proof}

Therefore, \eqref{ineq:k-m-upper-bound} is proved. On the other hand, the lower bounds on $\Lambda_k([g])$
were established in~\cite{Colbois-ElSoufi:2003:extremal} by using gluing techniques so that the spectrum
converges to the spectrum of the disjoint union of the manifold and up to $k$ spheres.
\begin{equation}\label{ineq:lower-bound}
  \Lambda_{k_0}([g])^{m/2} + \sum_i\Lambda_{k_i}(\Sph^m,[g_{\Sph^m}])^{m/2} \leq \Lambda_{k}([g])^{m/2} \leq \Lambda_{k,m}([g])^{m/2}.
\end{equation}
The remaining step is to prove by induction on $k$ that
\begin{equation}
  \Lambda_{k,m}(\Sph^m,[g_{\Sph^m}])=\Lambda_{k}(\Sph^m,[g_{\Sph^m}]).
\end{equation}
Indeed, if this is true for all $k_i < k$, and the last inequality in~\eqref{ineq:lower-bound} is strict, then
the upper bounds~\eqref{ineq:k-m-upper-bound} (for $M = \Sph^m$ and small enough $\delta$) imply that $\#w^\infty = 0$,
so $\abs{w^\infty} = 0$ and
each $(\alpha_i,\mu_i)$ realizes the corresponding $\Lambda_{k_i,m} = \Lambda_{k_i}$ by the induction assumption
(if some $k_{i_0} = k$, this means the remaining $k_i$ are zeros and $(\alpha_{i_0},\mu_{i_0})$ is still a maximizer). Then by Theorem~\ref{thm:main-0}, each pair $(\alpha_i,\mu_i)$ corresponds to $(\abs{du_i}^{m-2}, \abs{du_i}^m)$
for some $m$-harmonic maps.

\section{\texorpdfstring{$p$}{Lg}-Harmonic maps to \texorpdfstring{$\Sph^{\infty} = \Sph(\ell^2)$}{Lg}}
\label{sec:harm-maps}

Throughout this section, let $\Omega$ (unless specified otherwise) be a bounded
domain in a Riemannian manifold $(M,g)$. We assume that all geodesic balls
under consideration are contained in $\Omega$ and have radius below the
injectivity radius. The dependence $C=C(g)$ is understood through finitely
many $C^k$-bounds for $g$ and $g^{-1}$ on $\Omega$.

The goal of the current section is to extend the regularity arguments of Chang--Wang--Yang~\cite{Chang-Wang-Yang:1999:reg-harm-maps}
and Xin--Yang~\cite{Xin-Yang:1995:stable-p-harm-maps} to $p$-harmonic maps into the Hilbert unit sphere
$\Sph^\infty$.
Fortunately, Hilbert-valued function spaces and PDE theory exhibit
minimal differences compared to their standard counterparts.
For example, see \cite{Hytonen-Neerven-Veraar-Weis:2016:analysis-banach-1,
Hytonen-Neerven-Veraar-Weis:2023:analysis-banach-3} and the references therein
for a detailed treatment of Banach-valued analysis.

\begin{remark}\label{rem:scalar-to-vector}
  Since the Sobolev spaces $H^{k,p}(\R^m,\ell^2)$ are isomorphic to
  $L^p(\R^m,\ell^2)$ via Bessel potentials (see Lemma~\ref{lem:sobolev-lp-isom}),
  Lemma~\ref{lem:scalar-to-vector} tells us, for example, that the Poincaré inequalities, Sobolev embeddings, trace embeddings, etc.,
  still hold (at least, for $p \in (1, \infty)$).

  For $L^{\infty}$- and $C^{0,\gamma}$-estimates, the situation is even simpler. Consider a Banach-valued function $u\colon \Omega \to E$, and set $u_{b}(\cdot) = \brt{u(\cdot), b}$
  for $b \in E^*$. By the Hahn--Banach theorem, for fixed $x,y \in \Omega$, one has
  \begin{equation}
    \frac{\abs{u(x)-u(y)}}{\abs{x-y}^\gamma} = \sup \set*{\frac{\abs{u_{b}(x)-u_{b}(y)}}{\abs{x-y}^\gamma}}{b \in E^*,\ \norm{b}\leq 1},
  \end{equation}
  and similarly for $L^{\infty}$ (see~\cite[Proposition 1.2.17]{Hytonen-Neerven-Veraar-Weis:2016:analysis-banach-1}).
  Thus, $L^{\infty}$- and $C^{0,\gamma}$-bounds follow from uniform bounds for the corresponding scalar functions.
\end{remark}

We therefore see that the standard partial $C^{1,\gamma}$-regularity results for $p$-harmonic maps should still be valid
for $\Sph^\infty$ (another reason is that these bounds actually depend on the curvature of the target manifold rather than its dimension).
Unfortunately, the embeddings $H^{1,p}(\Omega,\ell^2) \to L^p(\Omega,\ell^2)$ are no longer compact, so one should avoid using this
fact in the argument.

For $p > 1$, we
call a map $u\colon \Omega \to \Sph^\infty \subset \ell^2$ (weakly) $p$-harmonic if $u \in H^{1,p}(\Omega, \Sph^\infty)$
and it satisfies the following PDE in the weak sense:
\begin{equation}
  \delta(\abs{du}^{p-2}du) = \abs{du}^{p} u,
\end{equation}
where $\delta = \delta_g$ denotes the formal adjoint to $d$. We will work mostly with the case $p \geq 2$. Therefore,
most results are stated only for this case, although some of them (for example, $\epsilon$-regularity) might actually hold for $1 < p < 2$ as well.

Let us set
\begin{equation}
  E_p[u] = E_p[u; \Omega] = \int_{\Omega} \abs{du}^p.
\end{equation}
\begin{definition}
  A (weakly) $p$-harmonic map $u \in H^{1,p}(\Omega, \Sph^\infty)$ is called \emph{stationary}
  if one has
  \begin{equation}
    \frac{d}{dt}\bigg|_{t=0} E_p[u \circ \phi_t] = 0
  \end{equation}
  for all differentiable 1-parameter families of diffeomorphisms
  $\phi_t\colon \Omega \to \Omega$ satisfying $\phi_0 = \id$ and
  $\phi_t|_{\Omega \setminus K} = \id|_{\Omega \setminus K}$ for some
  compact subset $K \subset \Omega$.
\end{definition}
The classical energy monotonicity formula does not really depend on the dimension
of the target space and the same proof holds in $\ell^2$-setting.
\begin{proposition}[Monotonicity formula, {\cite*[Proposition~10]{Takeuchi:1994:p-harm-regularity-to-spheres}}]
  Let $\Omega \subset M$ be a bounded domain, $2\leq p$.
  For a stationary $p$-harmonic map $u \in H^{1,p}(\Omega, \Sph^\infty)$, there exists
  a constant $C = C(m, \kappa)$, where
  $\kappa \geq \operatorname{sec}_g|_\Omega$ is an upper bound on the sectional curvature, such that
  for any $B_r(x) \subset \Omega$, one has
  \begin{equation}\label{eq:monoton-formula}
    \frac{d}{dr}  \brr{e^{Cr}r^{p-m}E_p[u; B_r(x)]} \geq 0.
  \end{equation}
\end{proposition}

\begin{definition}\label{def:stable-map}
  A $p$-harmonic map $u \in H^{1,p}(\Omega, \Sph^\infty)$ is called \emph{spectrally stable} if
  \begin{equation}\label{ineq:stable-map}
    \int \phi^2 \abs{du}^p \leq \int \abs{d\phi}^2 \abs{du}^{p-2}
    \quad \forall \phi \in C^\infty_0\brr{\Omega}
  \end{equation}
\end{definition}
\begin{remark}
  If we look at the definition of \emph{stable} $p$-harmonic maps:
  \begin{equation}
    \frac{d^2}{dt^2}\bigg|_{t=0}\int \abs{d\brr{\frac{u + tw}{\abs{u + tw}}}}^p
    \geq 0 \quad \forall w \in C_0^\infty(\Omega, \ell^2);
  \end{equation}
  we arrive at the following equivalent definition (for $w$ such that $w(x) \bot u(x)$ a.e.):
  \begin{equation}
    \int \abs{du}^{p-2}\brr{\abs{dw}^2 - \abs{du}^{2}\abs{w}^2} + (p-2)\int \abs{du}^{p-4}\brt{du,dw}^2 \geq 0
  \end{equation}
  Thus, for $p > 2$, spectral stability is stronger than the classical stability.
  At the same time, by arguing as in the proof of~\cite[Theorem~2.7]{Schoen-Uhlenbeck:1984:min-harm-maps},
  this definition is equivalent to the classical definition
  of stability into the \emph{infinite-dimensional sphere} for $p=2$. When $p < 2$, the second variation of the
  energy does not exist if $\abs{du}^{p-2} \not \in L^1_{loc}$.
\end{remark}

\begin{lemma}\label{lem:stable-energy-min}
  Spectrally stable $p$-harmonic maps $u \in H^{1,p}(\Omega, \Sph^\infty)$ are $p$-energy minimizing,
  that is,
  \begin{equation}
    \int_{\Omega} \abs{du}^p \leq \int_{\Omega} \abs{dw}^p
    \quad \forall w \in H^{1,p}\brr{\Omega, \Sph^\infty} \text{ with } w - u \in H^{1,p}_0(\Omega, \ell^2).
  \end{equation}
  In particular, they are stationary.
\end{lemma}
\begin{proof}
  We will need the $p$-harmonic map equation $\delta (\abs{du}^{p-2}du) = \abs{du}^p u$ and the identity
  $2\brt{u, u - w} = \abs{u - w}^2$ for $\abs{u} = \abs{w} = 1$. If
  $w \in H^{1,p}\brr{\Omega, \Sph^\infty}$ and
  $u - w \in H^{1,p}_0(\Omega, \ell^2)$, one has
  \begin{multline}
    2\int \brt{du, du - dw}\abs{du}^{p-2} = 2\int \brt{u, u - w} \abs{du}^p
    \\ = \int \abs{u - w}^2 \abs{du}^p \leq \int \abs{du - dw}^2 \abs{du}^{p-2},
  \end{multline}
  where we used the stability inequality~\eqref{ineq:stable-map}. Whence we obtain
  \begin{equation}
    \int_{\Omega} \abs{du}^{p} \leq \int_{\Omega} \abs{dw}^2 \abs{du}^{p-2}
  \end{equation}
  The inequality of the lemma then immediately follows from Hölder's inequality.
\end{proof}

\begin{lemma}\label{lem:p-harm:exist}
  Let $\Omega$ be a bounded domain of $(M,g)$, $p \in (1,\infty)$, and $f \in H^{1,p}(\Omega,\ell^2)$. Then there exists a $p$-harmonic function
  $u \in H^{1,p}(\Omega,\ell^2)$ (i.e. $\delta(\abs{du}^{p-2}du) = 0$ weakly) such that
  \begin{equation}
    E_p[u;\Omega] = \int_{\Omega} \abs{du}^p = \inf \set*{E_p[w;\Omega]}{w \in f + H^{1,p}_0(\Omega,\ell^2)}.
  \end{equation}
\end{lemma}
\begin{proof}
  The existence of a minimizer follows by the standard direct method (for example, see~\cite[Proposition~14.11.2]{Taylor:2023:pde-3}):
  $H^{1,p}(\Omega,\ell^2)$ is reflexive, $E_p$ is convex, and there exists an upper bound
  \begin{equation}
    \norm{w}_{H^{1,p}(\Omega,\ell^2)}^p \leq C (E_p[w;\Omega] + \norm{f}_{H^{1,p}(\Omega,\ell^2)}^p),
  \end{equation}
  implied by the Poincaré inequality.
\end{proof}

In dealing with regularity of $p$-harmonic maps it is useful to keep in mind the following elementary inequalities for vectors
in a Hilbert space $H$: for $p > 1$,
$a,b \in H$,
\begin{equation}\label{eq:V-def}
  V(a) := \abs{a}^{(p-2)/2}a,
\end{equation}
and some constant $C = C(p)$, we have
\begin{equation}\label{ineq:diff-lp-eq}
  \begin{split}
    \abs{V(a) - V(b)}^2 &\leq C\brt{\abs{a}^{p-2}a - \abs{b}^{p-2}b, a-b}, \\
    C^{-1}\abs{a-b}^2\brr{\abs{a} + \abs{b}}^{p-2} &\leq \abs{V(a) - V(b)}^2 \leq C\abs{a-b}^2\brr{\abs{a} + \abs{b}}^{p-2}, \\
    \text{which imply }  \abs{a - b}^p &\leq C\brt{\abs{a}^{p-2}a - \abs{b}^{p-2}b, a-b} \text{ for } p \geq 2.
  \end{split}
\end{equation}
The inequalities need to be proved only in dimension two by restricting to $Span \brc{a,b}$.

Note that the first two inequalities in~\eqref{ineq:diff-lp-eq} imply that any $p$-harmonic function with a given trace is unique,
hence energy minimizing by Lemma~\ref{lem:p-harm:exist}.
\begin{proposition}[Convex-hull property]\label{prop:conv-hull-prop}
  Let $\Omega$ have a smooth boundary, and let $u \in H^{1,p}(\Omega, \ell^2)$ be a $p$-harmonic function, where $p \in (1,\infty)$.
  Then $u(x) \in \overbar{\operatorname{conv}}\, \operatorname*{ess\,ran}(u|_{\partial\Omega})$ almost everywhere.
\end{proposition}
\begin{proof}
  Let $C= \overbar{\operatorname{conv}}\, \operatorname*{ess\,ran}(u|_{\partial\Omega})$ and
  $P \colon \ell^2 \to C$ be the minimal distance projector, that is, $P(a) = \operatorname{argmin}_{c \in C} \abs{a - c}$. Since $P$ is a $1$-Lipschitz map,
  $P(u) \in H^{1,p}$ and $\abs{dP(u)}\leq \abs{du}$  by Proposition~\ref{prop:lip-sobolev-comp}. As we discussed, $u$ is the unique energy minimizing function in its trace,
  so it must be $P(u) = u$.
\end{proof}

\begin{proposition}\label{prop:p-harm-func-reg}
  Let $p \geq 2$ and $B_R \subset M$ be a (smooth) geodesic ball. Then
  there exist constants $c = c(m,p,g) > 0$ and $\gamma = \gamma(m,p) \in (0,1)$ such that
  any $p$-harmonic function $u \in H^{1,p}(B_R,\ell^2)$ is (locally) in $C^{1, 2\gamma/{p}}$, and
  \begin{equation}\label{eq:p-harm-func-grad-sup}
    \sup_{B_{r/2}} \abs{du}^p \leq c \dashint_{B_r} \abs{du}^p,
  \end{equation}
  \begin{equation}\label{eq:p-harm-grad-decay}
    \Phi[u;r] \leq c \brr{\frac{r}{R}}^{2\gamma}\Phi[u;R],
  \end{equation}
  where
  \begin{equation}\label{eq:p-func-grad-decay}
    \Phi[u;r] = \dashint_{B_r} \abs{V(du) - [V(du)]_{B_r}}^2.
  \end{equation}
\end{proposition}
\begin{proof}
  Keeping in mind Lemma~\ref{lem:scalar-to-vector} and Remark~\ref{rem:scalar-to-vector}, the proposition can be proved, for example, by following the arguments
  of \cite{Giaquinta-Modica:1986:p-harm-func-reg} ($p$-harmonic functions are energy minimizing by~\eqref{ineq:diff-lp-eq}). Let us briefly mention the main points.

  One begins by following~\cite[Theorem~3.1]{Giaquinta-Modica:1986:p-harm-func-reg}.
  The difference quotient criterion for $\ell^2$-valued Sobolev functions is still valid (see Proposition~\ref{prop:diff-quot}).
  Hence, we can argue as in the proof of~\cite[Proposition~2]{Bojarski-Iwaniec:1987:diff-quot} to obtain that
  \begin{equation}\label{eq:Vdu-is-H1}
    V(du) \in H^{1,2}_{loc} \implies \abs{V(du)} = \abs{du}^{p/2} \in H^{1,2}_{loc}.
  \end{equation}
  On a manifold, the same conclusion follows by replacing difference quotients with parallel transport along geodesics.

  We then reproduce the formulas~\cite[(1.5)-(1.9)]{Uhlenbeck:1977:p-harm-func}
  by calculating $\Delta(|du|^{p-2}du)$ in two different ways: as
  $\Delta=\delta d+d\delta$ and by the Weitzenböck formula in normal coordinates. Compared with the Euclidean case, the latter introduces only
  an additional zeroth-order curvature term.
  Thus, the analogue of~\cite[(3.6)]{Giaquinta-Modica:1986:p-harm-func-reg}
  has the form $\delta(A(du)d f) - C f \leq  -c_p |\nabla V(du)|^2$,
  where $C = C(g,p)$, $f = |du|^p$, and $A(du)$ is the uniformly elliptic coefficient field arising from the
  linearized $p$-Laplacian divided by $|du|^{p-2}$.
  Note that we can first perform these calculations for smooth functions and then derive the final result by approximation, since we
  already have~\eqref{eq:Vdu-is-H1}.
  So $|du|^p$ is a nonnegative subsolution of a
  uniformly elliptic equation with a bounded zeroth-order term, and the local boundedness argument of
  \cite{Giaquinta-Modica:1986:p-harm-func-reg} applies.
  Hence $\abs{du} \in L^\infty_{loc}$, and $u$ is locally Lipschitz.

  To prove decay~\eqref{eq:p-harm-grad-decay} in the Euclidean case, $g = g_{\R^m}$, we compare $u$
  with the minimizer of a certain quadratic functional (\cite[(3.14)]{Giaquinta-Modica:1986:p-harm-func-reg} and below).
  The existence of this minimizer follows from the direct method (cf. Lemma~\ref{lem:p-harm:exist}). In the case of the $p$-Laplacian,
  this amounts to comparing $u$ with the function $v$ solving
  \begin{equation}\label{eq:elliptic-system}
    \begin{dcases}
      \int_{B_{r/2}(x_0)} \brr{\abs{dv}^2 + (p-2) \brt{t_0, dv}^2} \to \min
      \\ v - u \in H^{1}_0(B_{r/2}(x_0), \ell^2)
    \end{dcases},
  \end{equation}
  where $t_0  = d_{x_0}u/{\abs{d_{x_0}u}} \in \R^m \otimes \ell^2$. Writing $t_0 = \sum_{i=1}^n s_i \otimes b_i$ with $b_i \in \ell^2$,
  set $V:= Span\, \brc{b_i}_{i=1}^n$ so that $n = \dim V \leq m$. Choosing an orthonormal decomposition
  $\ell^2 = V \oplus V^{\bot} \approx \R^n \oplus \ell^2$,
  one sees that~\eqref{eq:elliptic-system} splits into
    \begin{equation}
    \begin{dcases}
      \operatorname{div} (A \nabla v') = 0 \ \text{ in } \ B_{r/2}(x_0)
      \\ v' = u'  \ \text{ on } \ \partial B_{r/2}(x_0)
    \end{dcases},\quad
    \begin{dcases}
      \Delta_{\R^m} v'' = 0 \ \text{ in } \ B_{r/2}(x_0)
      \\ v'' = u''  \ \text{ on } \ \partial B_{r/2}(x_0)
    \end{dcases},
  \end{equation}
  where $v = (v',v'')$, $v' \in H^{1}(B_{r/2}(x_0), \R^n)$, and $v'' \in H^{1}(B_{r/2}(x_0), \ell^2)$.
  In particular, the nontrivial elliptic system is confined to a finite-dimensional component.
  Therefore, the estimates \cite[(3.15–3.18)]{Giaquinta-Modica:1986:p-harm-func-reg} apply (they, in fact, hold for arbitrary Hilbert-valued
  elliptic systems; cf.~\cite[Chapter~III]{Giaquinta:1983:nonlinear-elliptic-systems}). We then proceed as in the remainder of the proof of
  \cite[Theorem~3.1]{Giaquinta-Modica:1986:p-harm-func-reg}.

  The case of a general metric $g$ follows by
  comparing~\eqref{eq:p-harm-grad-decay} with the decay for $p$-energy minimizers associated with the frozen metric $g(x_0)$;
  see, for example, \cite[Theorem~4.3]{Giaquinta-Modica:1986:p-harm-func-reg}, although this is somewhat stronger than needed.

  See~\cite[Lemma~3]{Duzaar-Mingione:2004:p-harm-approx} for the Hölder exponent of $u$.
\end{proof}

\begin{definition}
  For an arbitrary Banach space $E$ and a domain $\Omega \subset \R^m$,
  we say that a function $f \in L^1_{loc}(\Omega,E)$ belongs to $\BMO(\Omega,E)$ if
  \begin{equation}
    \norm{f}_{\BMO(\Omega,E)} := \sup_{B \subset \Omega}
    \ \dashint_{B}\abs{f - [f]_{B}} < \infty,
  \end{equation}
  where $B = B_r(p)$ and $[f]_{B} := \dashint_{B} f = \frac{1}{\abs{B}}\int_B f $.
\end{definition}
Vector-valued $\BMO$ shares many properties with its scalar counterpart. We will mostly need the facts from
\cite[Section~2.3]{Moser:2005:part-reg-harm-maps}, and
one can check directly that their proofs work for the Banach-valued case as well:
\begin{itemize}
  \item one can equivalently replace balls with cubes in the definition of $\BMO$;
  \item $f \in \BMO(\Omega, E) \implies \abs{f} \in \BMO(\Omega, \R)$ (this is one of the reasons, since many arguments involve just the norm);
  \item all the seminorms $\norm{f}_{\BMO_p(\Omega,E)} = \sup_{B \subset \Omega} \ \brr{\dashint_{B}\abs{f - [f]_{B}}^p}^{1/p}$ are equivalent
  for all $p\in (0,\infty)$ (see also~\cite[Theorem~3.2.30]{Hytonen-Neerven-Veraar-Weis:2016:analysis-banach-1});
  \item For domains $\Omega' \Subset \Omega \subset M$ and a function $\phi \in C^\infty_0(\Omega')$, there exists a constant $C$
  (see~\cite[Lemma~2.6]{Moser:2005:part-reg-harm-maps}) such that for $f \in \BMO(\Omega, E)$,
  \begin{equation}\label{ineq:bmo-contin}
    \norm{\phi\cdot(f - [f]_{\Omega'})}_{\BMO(\R^m, E)} \leq C  \norm{f}_{\BMO(\Omega, E)}.
  \end{equation}
\end{itemize}
\begin{definition}
  For a Banach space $E$,
  a function $f \in L^1(\R^m, E)$ is said to belong to the Hardy space $\mathcal{H}^1(\R^m, E)$ if
  \begin{equation}
    \norm{f}_{\mathcal{H}^1(\R^m, E)} := \norm{\mathcal{G}f}_{L^1(\R^m)} < \infty
  \end{equation}
  where
  \begin{equation}
    (\mathcal{G}f)(x) :=\sup_{\phi}\sup_{ t>0} \abs{(f * \phi_t)(x)},
  \end{equation}
  $\phi_t(x) = t^{-m}\phi(x/t)$, and $\phi$ runs over all $C^\infty_0(B_1(0))$ with $\norm{d\phi}_{L^\infty} \leq 1$.
\end{definition}
We call a function $a \colon \R^m \to E$ an atom, if
$\int a = 0$, $\supp a \subset B$ for some ball $B$, and $\norm{a}_{L^\infty} \leq \frac{1}{\abs{B}}$.
Repeating the proof of \cite[Proposition~2.1]{Moser:2005:part-reg-harm-maps}, we see
that the atomic decomposition still holds:
\begin{equation}
  \norm{f}_{\mathcal{H}^1(\R^m, E)} \approx \inf \set*{\sum_i \abs{\lambda_i}}{ f = \sum_i \lambda_i a_i,\ \lambda_i \in \R, \text{ and $a_i$ is an atom}}.
\end{equation}
From the atomic decomposition, it is then straightforward to check that for $f \in \mathcal{H}^1(\R^m, E)$, $h \in L^\infty(\R^m,E^*)$, one has
\begin{equation}\label{ineq:hardy-bmo}
  \int_{\R^m} \brt{f,h} \leq C\norm{f}_{\mathcal{H}^1(\R^m, E)} \norm{h}_{\BMO(\R^m,E^*)},
\end{equation}
and this duality extends to the whole $\BMO(\R^m,E^*)$ by continuity.
In fact, (modulo constants) $\BMO(\R^m,E^*) \hookrightarrow \mathcal{H}^1(\R^m, E)^*$ isometrically with the equality if and only if $E$ has the Radon-Nikodým property
(see~\cite{Blasco:1988:hardy-bmo-duality}), but we will not use this fact.
\begin{proposition}\label{prop:div-0-hardy}
  Let $E$ be a Banach space, $(\Omega,g) \subset \R^m$ be a bounded domain with a smooth metric $g$, and $p \in (1,\infty)$. Then
  there exists a constant $C = C(p,g,m)$ such that for every $f \in H^{1,p}_0(\Omega, \ell^2)$ and every 1-form
  $\omega \in L^{p'}(\Omega, T^*\R^m \otimes \mathfrak{L}[\ell^2,E])$
  with $\delta_g \omega = 0$ and $1/p' + 1/p = 1$, we have $\brt{\omega, df}_g \in \mathcal{H}^1(\R^m, E)$
  and
  \begin{equation}
    \norm{\brt{\omega, df}_g}_{\mathcal{H}^1(\R^m, E)} \leq C \norm{df}_{L^{p}(\Omega, \ell^2)}\norm{\omega}_{L^{p'}(\Omega, T^*\R^m \otimes \mathfrak{L}[\ell^2,E])},
  \end{equation}
  where $\langle\brt{\omega, df}_g, h \rangle_{E,E^*} = \brt{\omega, df\otimes h}_{g}$ pointwise and $(df \otimes h)(x) \in T^*\R^m \otimes\ell^2 \widehat{\otimes}_{\pi} E^*$.
\end{proposition}
\begin{proof}
  We extend $\omega$ and $f$ by zero to the whole $\R^m$.
  Since $\delta_g \omega =0$, we have that $\int \brt{\omega, \phi df}_g dv_g = -\int \brt{\omega, fd\phi}_g dv_g$ for all $\phi \in C^\infty_0(\R^m)$.
  On a bounded domain, we have $(C')^{-1}g_{\R^m} \leq g \leq C' g_{\R^m}$.
  The rest of the proof goes exactly as in~\cite[Theorem~2.4]{Moser:2005:part-reg-harm-maps}. Note that $f$ is Hilbert-valued, so we have the Sobolev
  inequality, and for $\omega$, we only need the integrability of the norm.
\end{proof}

\begin{lemma}[Small energy regularity]\label{lem:small-energy-reg}
  Let $(B_r(x_0), g) \subset T_{x_0} M \approx \R^m$ be a smooth geodesic ball and $p \geq 2$.
  Then for any $\gamma \in (0,1)$, there exists $\epsilon = \epsilon(g,m,p,\gamma) > 0$  such that
  for any stationary $p$-harmonic map $u \in H^{1,p}(B_r(x_0), \Sph^\infty)$ with
  $r^{p-m}E[u; B_r(x_0)] < \epsilon^p$, one has $u \in C^{0,\gamma}(B_{r/2}(x_0),\Sph^\infty)$.
\end{lemma}
\begin{proof}
 Our proof is an adaptation of~\cite*{Chang-Wang-Yang:1999:reg-harm-maps} for the
 setting of $\ell^2$-valued $p$-harmonic maps.

 Up to rescaling the ball and the metric, we can assume that $r = 1$ and $E[u; B_1(x_0)] < \epsilon^p$.
 Since we also assume that $B_1(x_0) = B_1(0) \subset \R^m$, we will work with Euclidean balls inside $B_1(0)$.

 Let us rewrite the $p$-harmonic map equation into spheres, $\delta(\abs{du}^{p-2}du^i) = |du|^p u^i$,
 in an alternative way. Since we are dealing with sphere-valued maps, one has
 $\sum_j u^j du^j = 0$, and as a consequence, for any constant vector $A \in \ell^2$, we have
\begin{equation}
  \delta(\abs{du}^{p-2}du^i) = \sum_j \brt{\abs{du}^{p-2} (u^i du^j - u^j du^i), d(u^j-A^j)} = \delta [\omega(u-A)]^i
\end{equation}
 weakly. Here
 \begin{equation}
  \omega^{ij}  = \abs{du}^{p-2} \brr{u^j du^i - u^i du^j}
  \quad\text{and}\quad \delta \omega = 0,
 \end{equation}
  since
 \begin{equation}
  \delta(\abs{du}^{p-2} u^i du^j) = \abs{du}^{p} u^i u^j - \abs{du}^{p-2}\brt{du^i, du^j}.
 \end{equation}
 We can estimate
 \begin{equation}
  \abs{\omega}_{T^*M\otimes \mathfrak{L}[\ell^2]} \leq \abs{\omega}_{T^*M\otimes \mathcal{HS}[\ell^2]} =  \brr{\sum_{ij} \abs{\omega^{ij}}^2}^{1/2} \leq 2 \abs{du}^{p-1},
 \end{equation}
 so $\omega \in L^{p'}(B_1, T^*M \otimes \mathfrak{L}[\ell^2])$, where $p' = p/(p-1)$.

 Now, let us fix $x \in B_{3/4}(0)$ and set $B_r := B_r(x)$. Let $h \in  H^{1,p}(B_{1/5}, \ell^2)$
 be the $p$-harmonic extension of $u|_{\partial B_{1/5}}$ (produced by Lemma~\ref{lem:p-harm:exist}), that is, $\delta(|dh|^{p-2}dh) = 0$,
 and take a function $\phi \in C_0^{\infty}(B_{1/4})$ with $\phi|_{B_{1/5}} \equiv 1$.
 The previous calculation together with inequalities~\eqref{ineq:diff-lp-eq}, \eqref{ineq:hardy-bmo}, and Proposition~\ref{prop:div-0-hardy} yield
 \begin{align}
  \int_{B_{1/5}}\abs{du-dh}^p
      &\leq c_1 \int_{B_{1/5}} \brt{\abs{du}^{p-2}du -\abs{dh}^{p-2}dh, d(u-h)}
  \\  &= c_1 \int_{B_{1/5}} \brt{\omega, d(u-h) \otimes (u-A)}
  \\  &= c_1 \int_{\R^m} \brt{\omega, d(u-h) \otimes \phi(u-A)}
  \\  &\leq c_2 \norm{\phi(u-A)}_{\BMO} \norm{du - dh}_{L^{p}(B_{1/5}, \ell^2)}\norm{\omega}_{L^{p'}(B_{1/5}, \cdots)}
  \\  &\leq c_3 \norm{u}_{\BMO(B_{1/4},\ell^2)}\norm{du - dh}_{L^{p}(B_{1/5},\ell^2)}\norm{du}_{L^{p}(B_{1/5},\ell^2)}^{p-1}
  \\  &\leq c_4 \epsilon \norm{du - dh}_{L^{p}(B_{1/5},\ell^2)}\norm{du}_{L^{p}(B_{1/5},\ell^2)}^{p-1},
 \end{align}
 where we used~\eqref{ineq:bmo-contin}, with the appropriate $A$, and the fact that
 \begin{equation}
  \norm{u}_{\BMO(B_{1/4},\ell^2)}^p \leq C \sup_{B \subset B_{1/4}} r^{p-m}E_p[u;B],
 \end{equation}
 which follows by the Hölder and Poincaré inequalities.
 Thus,
 \begin{equation}
  \int_{B_{1/5}}\abs{du-dh}^p \leq c_5 \epsilon^{p'} \int_{B_{1/5}}\abs{du}^p \leq C'' \epsilon^p.
 \end{equation}
 Now, for any $0 < s < 1/2$, using the fact that
 \begin{equation}
  \sup_{B_{s/5}}\abs{dh}^p \leq \int_{B_{1/5}} \abs{dh}^p \leq C''\int_{B_{1/5}} \abs{du}^p
 \end{equation}
 by~\eqref{eq:p-harm-func-grad-sup}, we obtain
 \begin{align}
  \brr{\frac{s}{5}}^p\dashint_{B_{s/5}}\abs{du}^p
      &\leq c_6 s^{p-m} \int_{B_{1/5}}\abs{du-dh}^p + c_6 s^p\dashint_{B_{s/5}}\abs{dh}^p
  \\  &\leq c_7 s^{p-m} \epsilon^{p'} \int_{B_{1/5}}\abs{du}^p +  c_6 s^p \sup_{B_{s/5}}\abs{dh}^p
  \\  &\leq c_8 \brr{\epsilon^{p'}s^{p-m}  + s^p} \int_{B_{1/5}}\abs{du}^p,
 \end{align}
 where $c_8 = c_8(p,m,g)$. Now, we first choose $0 < s < 1/2$ such that $c_8 s^p < \frac{1}{2}s^{p\gamma}$, and then
 $\epsilon > 0$ such that $c_8 \epsilon^{p'}s^{p-m} < \frac{1}{2}s^{p\gamma}$. As a result, we have
 \begin{equation}\label{ineq:energy-decay}
  (sr)^p\dashint_{B_{sr}(x)}\abs{du}^p \leq s^{p\gamma} r^p\dashint_{B_{r}(x)}\abs{du}^p
 \end{equation}
 for $r = 1/5$, all $x \in  B_{3/4}(0)$, all metrics $g'$ in a $C^\infty$-neighborhood of $g$, and stationary $p$-harmonic maps $u$ on $(B_1(0), g')$
 satisfying the small-energy hypothesis. By rescaling,
 $u_{\rho}(y) = u(\rho y+x)$, we see that~\eqref{ineq:energy-decay} holds for all $0 < r \leq 1/5$. By taking $r_k = s^k/5$ and
 $r \in (r_{k+1}, r_k]$, we derive
 an exponential energy decay
 \begin{equation}\label{ineq:morrey-decay}
  r^p\dashint_{B_{r}(x)}\abs{du}^p \leq c_9 r^{p\gamma}.
 \end{equation}
 It remains to apply Morrey's lemma (see Remark~\ref{rem:scalar-to-vector} and, for example, \cite[Lemma~2.1]{Moser:2005:part-reg-harm-maps})
 to conclude that $u \in C^{0,\gamma}(B_{1/2}(0),\Sph^\infty)$.
\end{proof}

\begin{lemma}\label{lem:cont-harm-smooth}
  Let $u\in H^{1,p}(\Omega, \Sph^\infty)$ be a stationary $p$-harmonic map with $p \geq 2$.
  If $u \in C^0(\Omega, \Sph^\infty)$,
  then $u \in C^{1,\beta}(\Omega,\Sph^\infty)$ for some $\beta = \beta(m,p) \in (0,1)$.
\end{lemma}
\begin{proof}
  Take $x_0 \in \Omega$ and $\phi \in C^\infty_0(B_r(x_0))$ with $\phi|_{B_{r/2}(x_0)} \equiv 1$. By testing $\delta(\abs{du}^{p-2}du) = \abs{du}^{p} u$
  against $[u-u(x_0)]\phi$ and
  applying Young's inequality, it is fairly standard (since $u$ is continuous) to derive the Caccioppoli-type inequality
  \begin{equation}
    r^p\dashint_{B_{r/2}(x_0)}\abs{du}^p \leq c_1\dashint_{B_{r}(x_0)}\abs{u - u(x_0)}^p \leq c_2\brr{\operatorname*{osc}_{B_{r}(x_0)} u}^p
  \end{equation}
  for some $r_0 = r_0(x_0)$ and all $r < r_0$.

  Therefore, by decreasing the ball if necessary, we can apply the previous lemma to establish that
  $u$ satisfies the Morrey decay~\eqref{ineq:morrey-decay} with  $p/(p+1) < \gamma < 1$
  and that $u \in C^{0,\gamma}(\Omega, \ell^2)$. Then
  we pick an appropriate ball $B_r(x) \subset \Omega$, choose
  the $p$-harmonic extension $h \in  H^{1,p}(B_{r}, \ell^2)$ of $u|_{\partial B_{r}}$ by Lemma~\ref{lem:p-harm:exist}, and
  compare $V(du)$ with $V(dh)$, using~\eqref{eq:p-func-grad-decay} and the convex-hull property of $h$ (Proposition~\ref{prop:conv-hull-prop}).
  See, for example, \cite[Lemma~6]{Duzaar-Mingione:2004:p-harm-approx} for the details.
\end{proof}

\begin{lemma}\label{lem:subseq-strong-conv}
  Let $p \geq 2$, $\brc{u_\epsilon} \subset H^{1,p}(\Omega, \Sph^\infty)$ be a bounded
  sequence of spectrally stable $p$-harmonic maps. Then there exists a spectrally stable $p$-harmonic map
  $u \in H^{1,p}(\Omega, \Sph^\infty)$ such that, up to a subsequence,
  $\abs{du_\epsilon}^p \to \abs{du}^p$ in $L^1(\Omega')$ for every $\Omega' \Subset \Omega$. In fact, up to isometry,
  $u_\epsilon \to u$ in $H^{1,p}(\Omega', \Sph^\infty)$.
\end{lemma}
\begin{proof}
  By Lemma~\ref{lem:main}, up to a subsequence, $\tf{u_\epsilon} \oset{w^*}{\to}\tf{u}$ in $H^{1,p}(\Omega)\widehat{\otimes}_\pi H^{1,p}(\Omega)$
  and $\tf{u_\epsilon} {\to}\tf{u}$ in $L^{p}(\Omega) \widehat{\otimes}_\pi L^{p}(\Omega)$
  for some $u \in H^{1,p}(\Omega, \ell^2)$. By Corollary~\ref{cor:l-inf-convergence},
  one concludes that $\m{\tf{u}}=\abs{u}^2 \equiv 1$.

  Then one way is to argue as in \cite*[Lemma~2.2]{Hong-Wang:1999:stable-harmonic-maps} by using
  partial regularity estimates for $p$-harmonic maps outside of a set of zero $(m\!-\!p)$-Hausdorff measure. However,
  in the infinite-dimension case, the stability inequality~\eqref{ineq:stable-map} is powerful enough so
  that one can set $\mu_\epsilon = |du_\epsilon|^pdv_g$, $\alpha_\epsilon = |du_\epsilon|^{p-2}$, choose
  a subsequence $\mu_\epsilon \oset{w^*}{\to} \mu$, $\alpha_\epsilon \oset{w^*}{\to} \alpha$, and apply
  Lemmas~\ref{lem:strong-conv},~\ref{lem:regularity-in-good-pt} with $D = \varnothing$.
\end{proof}

\begin{theorem}\label{thm:reg-harm-map}
  Let $M$ be a Riemannian manifold with $m = \dim M\geq 2$, $p\geq 2$,
  $d = \floor{p + 2\sqrt{p-1}} + 3$,
  and $u \in H^{1,p}_{loc}(M,\Sph^\infty)$ be a locally \emph{spectrally stable} $p$-harmonic map. Then
  there exists a closed subset $\Sigma$ of Hausdorff dimension at most $m-d$ such that
  $u \in C^{1,\gamma}(M\setminus \Sigma, \Sph^\infty)$ for some $\gamma = \gamma(m,p) \in (0,1)$. If $m = d$, the set of singularities
  $\Sigma$ is discrete.

  In particular, when $M$ is closed and $m \leq p+2 + 2\sqrt{p-1}$ (equivalently,
  $\max \{2, m-2\sqrt{m-2}\} \leq p < \infty$), the map $u$ is $C^{1,\gamma}$ and takes values in a finite-dimensional subsphere
  $\Sph^n \subset \Sph^\infty$.
\end{theorem}
\begin{proof}
  The proof is a generalization of \cite*{Xin-Yang:1995:stable-p-harm-maps}. Let
  $\Sigma$ be the set of singular points of $u$ (i.e., the complement
  of the largest open set where $u$ is continuous). Since spectrally stable $p$-harmonic maps
  are stationary by Lemma~\ref{lem:stable-energy-min}, it follows from
  the monotonicity formula,
  Lemmas~\ref{lem:small-energy-reg} and~\ref{lem:cont-harm-smooth} that
  \begin{equation}
    \Sigma = \set*{x \in M}{\lim_{r\to 0} r^{p-m} \int_{B_{r}(x)} \abs{du}^p \geq \epsilon_0 > 0}
  \end{equation}
  for some $\epsilon_0$.

  To estimate the Hausdorff dimension of $\Sigma$, one follows \cite*{Hong-Wang:1999:stable-harmonic-maps}
  with \cite*[Lemma~2.2]{Hong-Wang:1999:stable-harmonic-maps} replaced by Lemma~\ref{lem:subseq-strong-conv}.
  Namely,
  let us pick a point $x \in \Sigma$ and define $v_\epsilon(y) = u(\operatorname{exp}_x(\epsilon y))$.
  The monotonicity formula implies that the sequence $\brc{v_\epsilon}$ is bounded
  in $H^{1,p}(B_1(0))$, where $B_1(0) \subset T_x M$. By Lemma~\ref{lem:subseq-strong-conv}, we find a $p$-harmonic
  map $v$ such that, up to isometry and a subsequence, $v_\epsilon \to v$ in $H^{1,p}$, and therefore,
  $\abs{\partial_r v_\epsilon}^2 \to \abs{\partial_r v}^2$.
  Then~\cite*[Proposition~10]{Takeuchi:1994:p-harm-regularity-to-spheres} yields
  \begin{align}
    \lim_{\epsilon\to0}p\int_{B_1(0)}e^{Cr}r^{p-m}\abs{dv_\epsilon}^{p-2}\abs{\partial_r v_\epsilon}^2
       &= \lim_{\epsilon\to0}\int_{B_\epsilon(x)}e^{Cr}r^{p-m}\abs{du}^{p-2}\abs{\partial_r u}^2
    \\ &\leq \brr{\lim_{\epsilon\to0} e^{c_2\epsilon}\epsilon^{p-m} \int_{B_\epsilon(x)}\abs{du}^p - L}
    \\ &= 0,
  \end{align}
  where $L = \lim_{\epsilon\to 0} \epsilon^{p-m} \int_{B_{\epsilon}(x)} \abs{du}^p$.
  Hence, $v$ is radially constant, that is, $\partial_r v = 0$.

  Then we have all the ingredients to apply the arguments of \cite*[Theorem~3.1]{Xin-Yang:1995:stable-p-harm-maps} (in our case, $A=K=1$):
  if $p \leq n \leq d-1$, then
  there are no nonconstant stable $p$-harmonic maps of the form
  $\phi(\frac{x}{\abs{x}})\colon \R^{n} \to \Sph^\infty$
  where $\phi \in C^1(\Sph^{n-1}, \Sph^\infty)$.
  By the dimension reduction
  argument~\cite[Section~5]{Schoen-Uhlenbeck:1982:regularity-theory-harm-maps},
  the Hausdorff dimension of $\Sigma$ is at most $m-d$.
  When $m = d$, the proof of discreteness repeats the one at the very end of
  \cite[Section~5]{Schoen-Uhlenbeck:1982:regularity-theory-harm-maps}.

  To prove the last statement, we argue as follows. By
  the previous argument, $m \leq d - 1$ means that the singular set $\Sigma$
  is empty. So, $u \in C^{1,\gamma}(M,\Sph^\infty)$. Then we would like to show
  that the embedding $H^1_u \to L^2_u$ is compact, where $H^1_u := H^1(\abs{du}^{p-2}, \abs{du}^{p})$ and $L^2_u := L^2(\abs{du}^{p})$ (from Section~\ref{sec:weighted-stuff})
  with the norms
  \begin{gather}
    \norm{\phi}_{H^1_u}^2 := \int \abs{d\phi}^2 \abs{du}^{p-2} + \int \phi^2 \abs{du}^{p}
    \\ \shortintertext{and}
    \norm{\phi}_{L^2_u}^2 := \int \phi^2 \abs{du}^{p}.
  \end{gather}

  The following proof of compactness was inspired by notes of Mikhail Karpukhin and Daniel Stern \cite{Karpukhin-Stern:2024:private}.
  Let us recall the definition of
  $V(du)$~\eqref{eq:V-def}, $\abs{V(du)}^2 = \abs{du}^p$.
  By the standard difference quotient technique (more formally, the parallel transport along geodesics), using inequalities~\eqref{ineq:diff-lp-eq},
  one can derive that $V(du) \in H^{1}(M,\ell^2)$ (see, for example, equations \cite[(8)]{Bojarski-Iwaniec:1987:diff-quot} and
  \cite[(2.3)]{Duzaar-Fuchs:1990:p-harm-bochner}) and, for $\phi \in C^\infty(M)$,
  \begin{multline}
   \int \phi^2 \abs{dV(du)}^2 \leq c_1 \int \abs{dV(du)}\abs{\phi}\abs{du}^{p/2}\abs{d\phi}
   + c_1 \int \abs{du}^{p+2}\phi^2 \\+ c_1 \int \abs{V(du)} \abs{du} \abs{dV(du)}\phi^2.
  \end{multline}
  Therefore, applying the inequality $ab\leq \frac{\delta}{2} a^2 + \frac{1}{2\delta}b^2$ yields
  \begin{equation}
    \int \phi^2 \abs{dV(du)}^2 \leq c_2 \norm{du}_{L^\infty}^2 \norm{\phi}_{H^1_u}^2.
  \end{equation}
  By Proposition~\ref{prop:lip-sobolev-comp}, we have $\abs{d \abs{V(du)}} \leq \abs{dV(du)}$, and the previous inequality implies
  \begin{equation}\label{ineq:m-harm-comp-estim}
    \norm{d(\abs{du}^{p/2}\phi)}_{L^2} \leq \norm{\phi d\abs{V(du)}}_{L^2} + \norm{\abs{du}^{p/2}d\phi}_{L^2}
    \leq c_3 \norm{du}_{L^\infty} \norm{\phi}_{H^1_u}.
  \end{equation}
  Hence, the embedding $H^1_u \to L^2_u$ factors through the multiplication by $|du|^{p/2}$ followed by the compact embedding $H^1 \to L^2$.
  Finally, the compactness of the embedding implies that
  there are only finitely many linearly independent coordinate functions of $u$, which means
  $u(x) \in \Sph^n$ for $\Sph^n \subset \Sph^\infty$ and $\abs{du}^p dv_g$-a.e. $x\in M$.
  It is then clear that $u(x) \in \Sph^n$ for a.e. $x\in M$ with respect to the
  Riemannian volume measure $v_g$.
\end{proof}

\medskip
\printbibliography

\end{document}